\documentclass[12pt,reqno]{amsart}
\usepackage[a4paper,
  left=25mm,
  right=25mm,
  top=25mm,
  bottom=25mm
]{geometry}

\usepackage{amsmath,amsfonts,amssymb, amsthm}
\usepackage{bm}
\usepackage{graphicx}
\usepackage{ascmac}
\usepackage[legacycolonsymbols]{mathtools}
\usepackage{nccmath}
\usepackage{mathrsfs}
\usepackage{mleftright}

\usepackage[italicdiff]{physics}

\usepackage{thmtools}
\usepackage{thm-restate}
\usepackage{needspace}
\usepackage{etoolbox}
\usepackage{xspace}
\usepackage{placeins}
\usepackage{multirow}
\usepackage{makecell}

\usepackage{enumitem}
\setlist[enumerate]{label=(\arabic*), ref=(\arabic*), leftmargin=2.63em, labelsep=0.35em}
\setlist[itemize]{leftmargin=2em, labelsep=0.4em}

\usepackage[%
setpagesize=false,%
bookmarks=true,%
bookmarksnumbered=true,%
hidelinks%
]{hyperref}

\mleftright
\allowdisplaybreaks
\mathtoolsset{showonlyrefs=true}
\def\thefootnote{\arabic{footnote}}

\DeclareMathAlphabet{\mathpzc}{OT1}{pzc}{m}{it}

\makeatletter

\newif\if@noindentafterheading
\let\orig@afterheading\@afterheading
\def\@afterheading{%
  \if@noindentafterheading\global\@afterindentfalse\fi
  \orig@afterheading
}

\renewcommand\section{\@startsection{section}{1}{\z@}%
  {-3.5ex \@plus -1ex \@minus -.2ex}
  {2.6ex \@plus .2ex}
  {\normalfont\large\bfseries}%
}

\renewcommand\subsection{\@startsection{subsection}{2}{\z@}%
  {-2.5ex \@plus -.3ex \@minus -.2ex}%
  {1.5ex \@plus .2ex}%
  {\normalfont\bfseries}%
}

\let\orig@section\section
\renewcommand\section{%
  \@ifstar{\section@star}{\section@nostar}%
}
\newcommand\section@nostar{%
  \@ifnextchar[{\section@opt}{\section@noopt}%
}
\newcommand\section@opt[2][]{%
  \global\@noindentafterheadingtrue
  \orig@section[#1]{#2}%
  \global\@noindentafterheadingfalse
}
\newcommand\section@noopt[1]{%
  \global\@noindentafterheadingtrue
  \orig@section{#1}%
  \global\@noindentafterheadingfalse
}
\newcommand\section@star[1]{%
  \global\@noindentafterheadingtrue
  \orig@section*{#1}%
  \global\@noindentafterheadingfalse
}

\let\orig@subsection\subsection
\renewcommand\subsection{%
  \@ifstar{\subsection@star}{\subsection@nostar}%
}
\newcommand\subsection@nostar{%
  \@ifnextchar[{\subsection@opt}{\subsection@noopt}%
}
\newcommand\subsection@opt[2][]{%
  \global\@noindentafterheadingtrue
  \orig@subsection[#1]{#2}%
  \global\@noindentafterheadingfalse
}
\newcommand\subsection@noopt[1]{%
  \global\@noindentafterheadingtrue
  \orig@subsection{#1}%
  \global\@noindentafterheadingfalse
}
\newcommand\subsection@star[1]{%
  \global\@noindentafterheadingtrue
  \orig@subsection*{#1}%
  \global\@noindentafterheadingfalse
}

\renewcommand\subsubsection{\@startsection{subsubsection}{3}{\z@}%
  {-2.0ex \@plus -.3ex \@minus -.2ex}
  {1.0ex \@plus .2ex}
  {\normalfont\bfseries}
}

\let\orig@subsubsection\subsubsection
\renewcommand\subsubsection{%
  \@ifstar{\subsubsection@star}{\subsubsection@nostar}%
}
\newcommand\subsubsection@nostar{%
  \@ifnextchar[{\subsubsection@opt}{\subsubsection@noopt}%
}
\newcommand\subsubsection@opt[2][]{%
  \global\@noindentafterheadingtrue
  \orig@subsubsection[#1]{#2}%
  \global\@noindentafterheadingfalse
}
\newcommand\subsubsection@noopt[1]{%
  \global\@noindentafterheadingtrue
  \orig@subsubsection{#1}%
  \global\@noindentafterheadingfalse
}
\newcommand\subsubsection@star[1]{%
  \global\@noindentafterheadingtrue
  \orig@subsubsection*{#1}%
  \global\@noindentafterheadingfalse
}

\newcommand{\inputnolabels}[1]{%
  {%
    \let\label\@gobble
    \input{#1}%
  }%
}

\makeatother

\newcommand{\mcP}{\mathcal{P}}
\newcommand{\mcS}{\mathcal{S}}
\newcommand{\mbG}{\mathbb{G}}
\newcommand{\msD}{\mathscr{D}}

\newcommand{\Z}{\mathbb{Z}}								
\newcommand{\Zz}{\mathbb{Z}_{\geq 0}}						
\newcommand{\Zp}{\mathbb{Z}_{> 0}}						
\newcommand{\Q}{\mathbb{Q}}								
\newcommand{\Qb}{\overline{\mathbb{Q}}}					
\newcommand{\C}{\mathbb{C}}								
\newcommand{\Ch}{\widehat{\mathbb{C}}}					
\newcommand{\Pj}{\mathbb{P}}								

\newcommand{\A}{\alpha}									%
\newcommand{\B}{\beta}									%
\newcommand{\G}{\gamma}								%
\newcommand{\la}{\lambda}								%

\newcommand{\OL}{\vspace{1\baselineskip}}					%
\newcommand{\HL}{\vspace{.5\baselineskip}}					%
\newcommand{\q}{\quad}									
\newcommand{\sq}{\hspace{1.3em}}							
\newcommand{\f}[2]{\frac{#1}{#2}}							%
\newcommand{\npmod}[1]{\!\!\!\! \pmod{#1}}					%
\newcommand{\ceq}{\coloneqq} 							
\newcommand{\ie}{i.e.\ }									
\newcommand{\vt}[1]{\left\lvert #1 \right\rvert}					
\newcommand{\gen}[1]{\left\langle #1 \right\rangle}				
\newcommand{\fl}[1]{\left\lfloor #1 \right\rfloor}					
\newcommand{\relmiddle}[1]{\mathrel{}\middle#1\mathrel{}}		

\DeclareMathOperator{\Aut}{Aut}							
\DeclareMathOperator{\Gal}{Gal}							
\DeclareMathOperator{\Sym}{Sym}							

\pretocmd{\section}{\needspace{4\baselineskip}}{}{}
\pretocmd{\subsection}{\needspace{3\baselineskip}}{}{}
\pretocmd{\subsubsection}{\needspace{2\baselineskip}}{}{}

\newtheoremstyle{mythmstyle}  
  {12pt}   
  {12pt}   
  {\itshape} 
  {}      
  {\bfseries} 
  {.}     
  { }     
  {}      

\newtheoremstyle{mydefstyle}  
  {12pt}   
  {12pt}   
  {}       
  {}       
  {\bfseries} 
  {.}      
  { }      
  {}       

\newtheoremstyle{myremarkstyle} 
  {12pt}   
  {12pt}   
  {}       
  {}       
  {\normalfont} 
  {.}      
  { }      
  {}       

\theoremstyle{mythmstyle}
\newtheorem{lemma}{Lemma}[section]
\newtheorem{prop}[lemma]{Proposition}
\newtheorem{thm}[lemma]{Theorem}
\newtheorem{cor}[lemma]{Corollary}

\theoremstyle{mydefstyle}
\newtheorem{definition}[lemma]{Definition}
\newtheorem{exa}[lemma]{Example}
\newtheorem{rem}[lemma]{Remark}

\newcommand{\thmref}[1]{Theorem~\ref{#1}}
\newcommand{\propref}[1]{Proposition~\ref{#1}}
\newcommand{\lemref}[1]{Lemma~\ref{#1}}
\newcommand{\corref}[1]{Corollary~\ref{#1}}
\newcommand{\defref}[1]{Definition~\ref{#1}}

\newcommand{\figref}[1]{Figure~\ref{#1}}
\newcommand{\tabref}[1]{Table~\ref{#1}}
\newcommand{\secref}[1]{Section~\ref{#1}}

\newcommand{\belyi}{Bely\u{\i}\xspace}
\newcommand{\dde}{dessin d'enfant\xspace}
\newcommand{\ddes}{dessins d'enfants\xspace}
\newcommand{\Dde}{Dessin d'enfant\xspace}
\newcommand{\DdEs}{Dessins d'Enfants\xspace}
\newcommand{\AD}{\Aut{\msD}}
\newcommand{\sigman}{(1\ 2\ \ldots\ n)}

\numberwithin{equation}{section}

\newcommand{\thmclassfour}{If a uniform passport $[a^{p}, b^{q}, n]$ with $n = pa = qb$ and $2 \le p < q$ has genus
at least~$2$, then it admits a \dde with trivial automorphism group.

The same statement holds for any passport obtained by permuting $a^{p}$, $b^{q}$, and $n$.}

\newcommand{\thmclassonetofour}{If a uniform unicellular passport $[a^{p},b^{q},n]$ with $n=pa=qb$
has genus at least~$2$, then it admits a \dde with trivial automorphism group.

The same statement holds for any passport obtained by permuting $a^{p}$, $b^{q}$, and $n$.}

\begin{document}

\title{Uniform Unicellular Dessins d'Enfants with Trivial Automorphism Groups}

\author{Tatsuya Ohnishi}
\markboth{}{}

\begin{abstract}
A \belyi function on a smooth projective algebraic curve defined over a number field determines a bipartite graph called
a \emph{dessin d'enfant}.

We study the regularity and automorphism groups of dessins with uniform passports. In previous papers, we proved that every passport of the form $[n,n,n]$, $[n,b^{q},n]$, or $[b^{q},b^{q},n]$ of genus at least~$2$ admits a dessin with trivial automorphism group. In this paper, we prove the analogous result for passports of the form $[a^{p},b^{q},n]$.
The proof is mainly based on a counting argument: we compare a lower bound
for the number of permutation representations having the prescribed passport
with an upper bound for the number admitting a nontrivial automorphism.
Together with our previous results, this shows that every uniform unicellular passport of genus at least~$2$ admits a dessin with trivial automorphism group.
\end{abstract}

\maketitle

\vspace{-1\baselineskip}

\def\thefootnote{\fnsymbol{footnote}}
\makeatletter
\renewcommand\@makefntext[1]{%
  \noindent\hspace{0.5em}#1%
}
\makeatother
\begin{NoHyper}
\footnotetext{Tatsuya~Ohnishi~(\raisebox{-0.18ex}{\includegraphics[height=2.05ex]{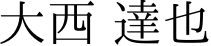}})}
\footnotetext{Graduate School of Information Science and Technology, The University of Osaka, Japan}
\footnotetext{e-mail: {\tt ohnishi-t@ist.osaka-u.ac.jp}}
\end{NoHyper}

\makeatletter
\renewcommand\@makefntext[1]{%
  \noindent\@makefnmark\ #1%
}
\makeatother

\def\thefootnote{\arabic{footnote}}


{\small
\tableofcontents
}


\section{Introduction}

\subsection{\belyi's Theorem and \DdEs}

\label{sec:belyi}
\begin{thm}[\belyi's Theorem]\cite{Belyi79}\cite{Belyi02}\cite[Theorem~1.3]{Jones16}
Let $X$ be a compact Riemann surface, that is, a smooth projective algebraic curve in
$\Pj_{\C}^{N}$ for some $N$.
Then $X$ can be defined over the field of algebraic numbers $\Qb$ if and only if there exists a nonconstant meromorphic
function $\B\colon X \to \Ch\ (\ceq \C \cup \{ \infty \})$ ramified over at most three points.
\end{thm}

This remarkable theorem allows algebraic curves defined over number fields to be studied through
combinatorial structures on surfaces.  For a compact Riemann surface $X$
defined over a number field, there exists a meromorphic function
$\B\colon X\to\Ch$, called a \emph{\belyi function}, ramified over at most three
points. By composing $\B$ with a suitable M\"obius transformation, we may assume
that its critical values are contained in $\{0,1,\infty\}$.

Such a pair $(X,\B)$ is called a \emph{\belyi pair}. From a \belyi pair, one
obtains a bipartite graph embedded in $X$, called a \emph{\dde} (child's drawing),
or simply a \emph{dessin} (see \defref{def:dessin}).

The dessin for $(X, \B)$ is drawn on an orientable surface of the same genus as $X$, where the black vertices ($\bullet$)
 and white vertices ($\circ$) represent $\B^{-1}(0)$ and $\B^{-1}(1)$, respectively\footnote{Some references
(such as \cite{Jones16}) represent $\B^{-1}(0)$ by white vertices and $\B^{-1}(1)$ by black vertices.
In this paper, we adopt the convention used in many classical references.}.
The embedded graph is $\B^{-1}([0,1])$. Its edges are the closures of the connected components of $\B^{-1}((0,1))$,
and its faces are the connected components of $X \setminus \B^{-1}([0,1])$. Each face is homeomorphic to an open
disk and contains exactly one point of $\B^{-1}(\infty)$.

By studying the properties of \ddes, one can combine insights from algebraic geometry and
combinatorics, including graph theory, to enrich both subjects. This perspective also opens up
a range of possibilities for further applications.

Two dessins $\msD$ and $\msD'$ are said to be \emph{isomorphic} if there exists
an orientation-preserving homeomorphism between the underlying surfaces
that induces an isomorphism of the embedded bipartite graphs, preserving
vertex colors and the cyclic order of incident edges at each vertex.
Equivalently, if $(X,\B)$ and $(X',\B')$ are the corresponding \belyi pairs,
then $\msD$ and $\msD'$ are isomorphic if and only if there exists a
biholomorphic map $\varphi\colon X \to X'$ such that $\B'\circ\varphi=\B$.

Two examples of \ddes are shown in \figref{fig:dessins-ex}. The dessins $\msD_{1}$ and $\msD_{2}$ correspond
to the \belyi pairs $(X_{1}, \B_{1})$ and $(X_{2}, \B_{2})$, respectively, where
\begin{align}
X_{1}&\colon y^{2} = x(x-1)(x-\sqrt[3]{2}), \q \B_{1}(x, y) = x^{3}(2-x^{3}), \\
X_{2}&\colon y^{2} = x(x-1)(x-\sqrt[3]{2}\omega), \q \B_{2}(x, y) = x^{3}(2-x^{3}),\q  \omega \ceq e^{\f{2\pi i}{3}}.
\end{align}

\begin{figure}[htbp]
\centering
\includegraphics[width=0.84\textwidth]{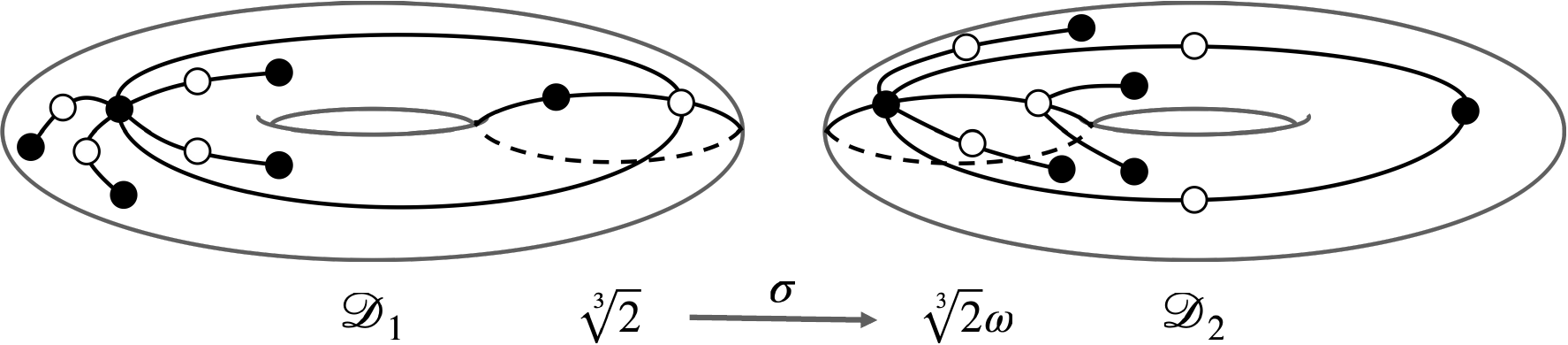} 
\caption{Two Galois-conjugate \ddes}
\label{fig:dessins-ex}
\end{figure}

Both dessins have genus~$1$ and $12$ edges; the number of edges equals the degree of the corresponding \belyi function.
Note that the degree of a \belyi function $\B\colon X\to\Ch$ is the degree of this
morphism, equivalently, the number of points in the preimage of a noncritical value. This need not
coincide with the degree of the polynomial expression defining $\B$ in affine coordinates. In fact,
in the examples above, the former is $12$, whereas the latter is $6$.

Let $\sigma$ be an automorphism of $\Qb$ sending
$\sqrt[3]{2}$ to $\sqrt[3]{2}\omega$.
Applying $\sigma$ to the coefficients of the defining equations of $(X_1,\B_1)$ yields
the \belyi pair $(X_2,\B_2)$, and hence the dessin $\msD_2$.
This provides a concrete example of the action of the absolute Galois group
$\mbG=\Gal(\Qb/\Q)$ on \ddes \cite[4.2.1]{Jones16}.

The action of $\mbG$ on the set of all dessins is faithful; that is,
for any two distinct elements of $\mbG$, there exists a dessin
whose images under these elements lie in distinct isomorphism classes.
Therefore, the study of Galois orbits of dessins provides a powerful tool for investigating the structure
of the absolute Galois group.
\HL
 
\subsection{Regularity and Automorphism Groups}

\label{sec:regaut}
Regularity and the isomorphism class of the automorphism group are invariant under the Galois action.
A dessin is said to be \emph{regular} if its \emph{monodromy group} (see \defref{def:monog}) acts
regularly (that is, freely and transitively) on the set of edges.
From a geometric point of view, the \emph{automorphism group} of a dessin is the group of deck transformations of the
associated \belyi covering.
Equivalently, it can be identified with the centralizer of the monodromy group in the symmetric group acting on the edges.

By studying regularity and automorphism groups, one can gain insight into the symmetry properties of a \dde.
The order of the automorphism group always divides the number of edges of the dessin, and the equality of these two numbers is equivalent to the dessin being regular.
Thus, the structure of the automorphism group provides a precise measure of the symmetry exhibited by the dessin.

\figref{fig:regular} shows two dessins of genus~$1$ with the same passport; that is, they have the same valency list:
each has $8$ edges, two black vertices of valency $4$, four white vertices of valency $2$, and two faces of valency $4$.
Although both dessins may appear highly symmetric at first glance, the left dessin is regular, whereas the right one is not.
The automorphism group of the left dessin has order~$8$, while that of the right dessin has order~$4$.

\begin{figure}[htbp]
\centering
\includegraphics[width=0.84\textwidth]{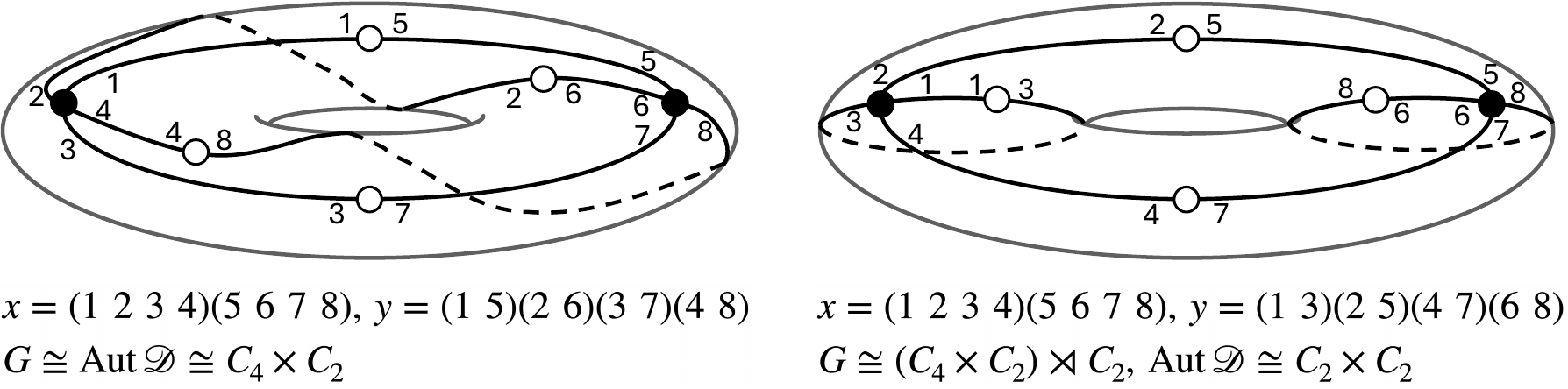} 
\caption{Regular and nonregular dessins with the same passport}
\label{fig:regular}
\end{figure}

Thus, regularity and automorphism groups may differ among dessins with the same passport.
However, they are invariant under the action of the absolute Galois group.
They therefore play an essential role in the study of families of dessins lying in the same Galois orbit.

Furthermore, regularity and automorphism groups have a significant impact on the relationship between
the field of moduli and the field of definition.

For a \dde $\msD$ corresponding to a \belyi pair $(X, \B)$, the \emph{field of moduli} $M(\msD)$
is defined as the fixed field of the subgroup
$G(\msD) = \{ \sigma \in \mbG \mid \msD \cong \msD^{\sigma} \} \le \mbG$.
That is, $M(\msD)$ is the subfield of $\Qb$ consisting of those elements that are fixed by every automorphism
$\sigma$ for which the conjugate dessin $\msD^{\sigma}$ is isomorphic to $\msD$. 
For the dessins in \figref{fig:dessins-ex}, we have
\begin{align}
M(\msD_{1})=\Q(\sqrt[3]{2}),\q M(\msD_{2})=\Q(\sqrt[3]{2}\omega).
\end{align}

A number field $K$ is called a \emph{field of definition} of $\msD$ if both $X$ and $\B$ can be defined over $K$.
Unlike the field of moduli, a field of definition of $\msD$ need not be unique, and there need not exist
a smallest field of definition.

The field of moduli depends only on the isomorphism class of $\msD$ and is contained in any field of definition of $\msD$.
However, the field of moduli is not always a field of definition.
For certain important classes of dessins, including dessins with trivial automorphism groups and regular dessins, the field of moduli is a field of definition \cite{Sijsling16}\cite{Conder13}.

In this paper, we further develop the counting methods introduced in our previous work to establish the existence
of dessins with trivial automorphism groups for the remaining uniform unicellular passports of genus at least~$2$.
These results provide new information on the interplay between regularity and automorphism groups, and may
help clarify their role in the study of fields of moduli and fields of definition.
\HL

\subsection{Uniform Passports and Dessins}

\label{sec:uniform}
A \dde is said to be \emph{uniform} if the valencies of black vertices, white vertices, and faces are each constant.
One also says that it has a uniform passport (or uniform valency list).

Uniformity expresses a local combinatorial regularity: the valencies of the black vertices, white vertices, and faces
are constant within each class. Every regular dessin is uniform, but the converse does not hold. Uniformity alone
does not imply the existence of any nontrivial global symmetry; a uniform dessin may even have a trivial
automorphism group. In contrast, regularity is a global symmetry condition.

As an example, consider uniform passports $[a^{p}, b^{q}, c^{r}]$, where $n = pa = bq = rc$, and suppose that $n = 6$.
By symmetry among black vertices, white vertices, and faces, we may assume $c \ge a \ge b$ (equivalently, $r \le p \le q$).

In genus~$0$, the uniform passports are $[6, 1^{6}, 6]$ and $[2^{3}, 2^{3}, 3^{2}]$.
The corresponding dessins are shown in \figref{fig:genus0-n6}.
In both cases, the dessins are regular.

\begin{figure}[htbp]
\centering
\includegraphics[width=0.38\textwidth]{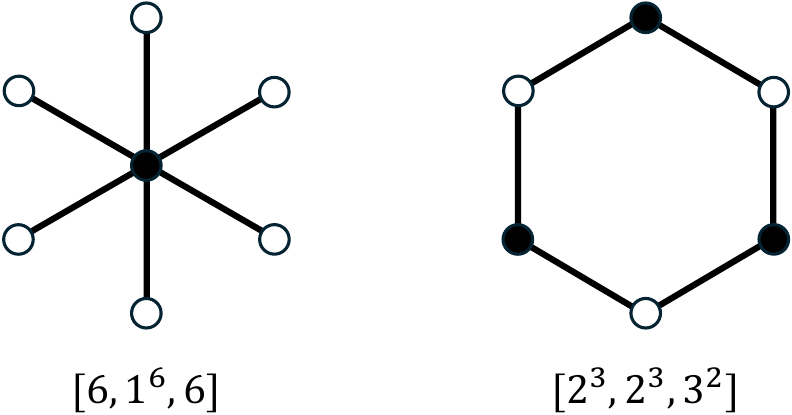}
\caption{Uniform \ddes of genus~$0$, degree~$6$}
\label{fig:genus0-n6}
\end{figure}

In genus~$1$, the uniform passports are $[3^{2}, 2^{3}, 6]$ and $[3^{2}, 3^{2}, 3^{2}]$, and the corresponding dessins are shown in \figref{fig:genus1-n6}.
The dessin on the left is regular, whereas the one on the right is not.
The automorphism group of the right-hand dessin has order~$2$, which is strictly smaller than $6$.

\begin{figure}[htbp]
\centering
\includegraphics[width=0.84\textwidth]{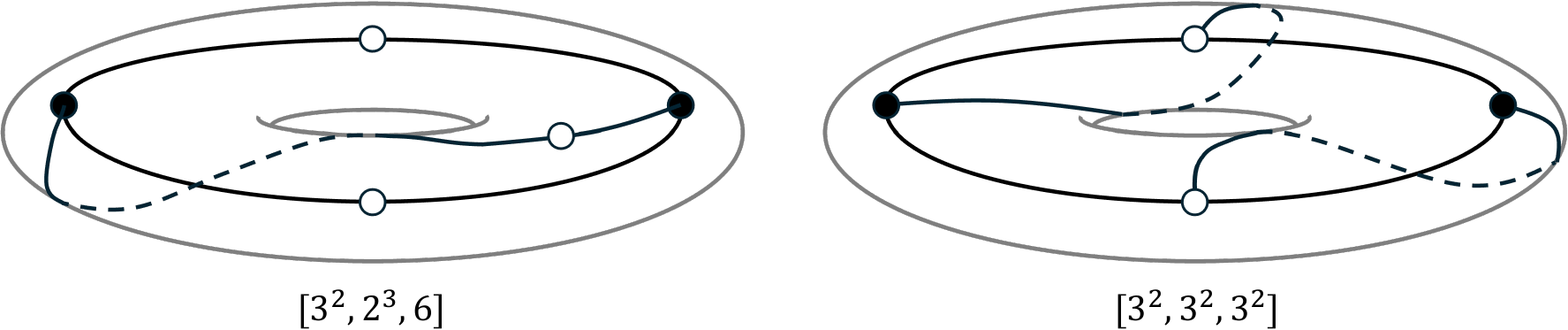}
\caption{Uniform \ddes of genus~$1$, degree~$6$}
\label{fig:genus1-n6}
\end{figure}

For genus at least~$2$, the only uniform passport is $[6, 3^{2}, 6]$, which has genus~$2$.
There are four dessins with this passport, shown in \figref{fig:genus2-n6}.
The orders of their automorphism groups are $6$, $3$, $2$, and $1$, respectively, from the upper left to the lower right.
Only the upper-left dessin is regular; the others are nonregular, and the lower-right dessin has trivial automorphism group.

\begin{figure}[htbp]
\centering
\includegraphics[width=0.84\textwidth]{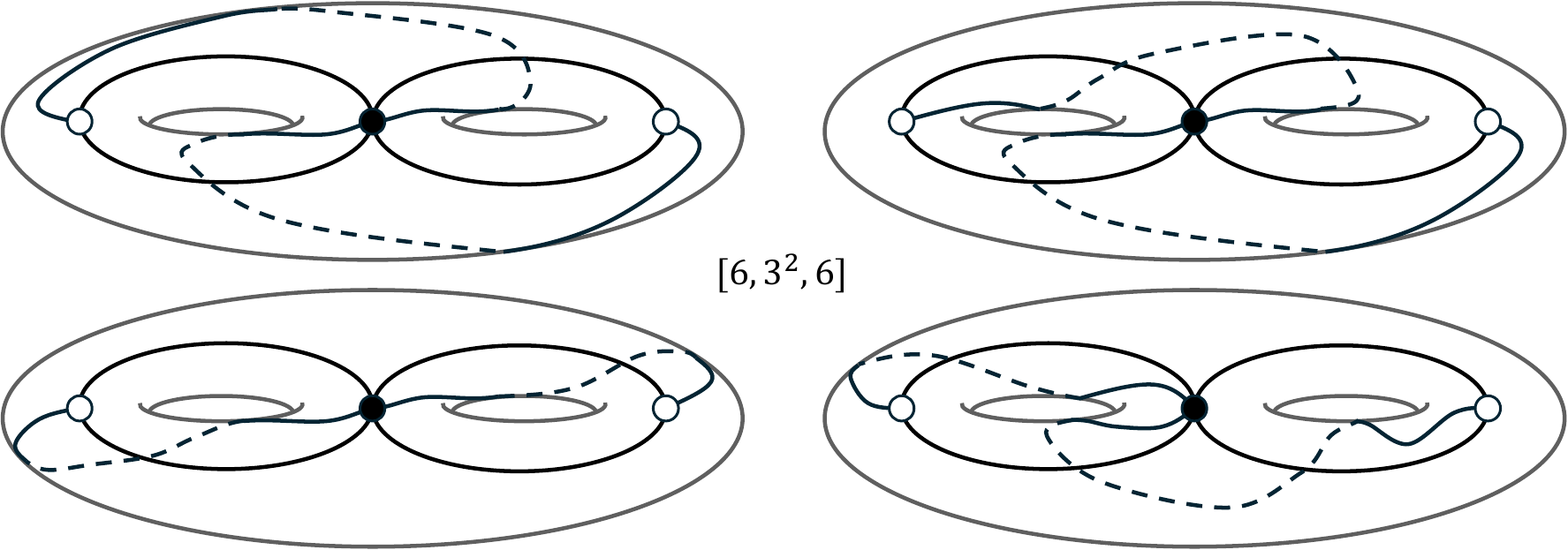}
\caption{Uniform \ddes of genus~$2$, degree~$6$}
\label{fig:genus2-n6}
\end{figure}
\HL

\subsection{Related Work}

The present work is related to the study of automorphism groups
of dessins and to the enumeration of unicellular maps and
permutation factorizations.
We briefly review results in these directions and explain their
relation to the problem considered here.

Jones~\cite{Jones14} shows how, for a finite two-generated group $G$, regular dessins with automorphism group $G$
can be enumerated, represented as quotients of a single regular dessin $U(G)$,
and analyzed under the action of certain hypermap operations.

From a topological perspective, Hidalgo~\cite{Hidalgo19} proves that
every faithful action of a finite group $G$ by orientation-preserving
homeomorphisms on a closed orientable surface of genus at least~$2$
can be realized as the action of the full automorphism group
of a dessin on a surface of the same genus.

In a complementary direction, Jones~\cite{Jones21} establishes
results on the realization of prescribed groups as automorphism
groups in various permutational categories, including those of
maps and hypermaps.
For dessins d'enfants, these results imply that, for any prescribed
hyperbolic type, every finite group is isomorphic to the automorphism
group of infinitely many pairwise non-isomorphic dessins.
In particular, there are infinitely many dessins with trivial
automorphism group, although their complete passports are not prescribed.
The present paper addresses the more specific question of whether
every prescribed uniform unicellular passport of genus at least~$2$,
whose associated type is necessarily hyperbolic,
admits a dessin with trivial automorphism group.

Goupil and Schaeffer~\cite{Goupil98} provide an explicit formula for
the number of decompositions of a fixed $n$-cycle in the symmetric
group $S_{n}$ into two permutations of prescribed cycle types.
Their formula is the starting point for the counting argument
in this paper and is used in \secref{sec:Nlower}.

Horie~\cite{Horie24} studies the enumeration of equivalence classes of dessins whose automorphism groups have
a prescribed order~$r$. The focus is on dessins with two vertices, which necessarily have
passports (valency lists) of the form $[n^{1}, n^{1}, \lambda]$ with $\lambda \vdash n$.
Explicit enumeration formulas are obtained when either the total number
of faces is prescribed or the number of faces of boundary length~$2$
(that is, valency~$1$ in our convention) is prescribed.
These formulas sum over face partitions satisfying the specified condition,
rather than fixing the entire passport.

Douvropoulos~\cite{Douvropoulos24} studies unicellular maps,
which can be viewed as dessins with passports of the form
$[\lambda,2^{q},2q]$, where $\lambda\vdash 2q$.
The main result gives a Harer--Zagier-type generating
function for such maps invariant under a prescribed rotation,
refined by statistics of the rotation action on vertices and edge pairings.
This enumeration allows $\lambda$ to vary and does not prescribe
the complete passport.
The case $b=2$ of the uniform unicellular passports $[a^{p},b^{q},n]$
considered in the present paper corresponds to specifying $\lambda=(a^{p})$
in this family.

These works provide realization theorems and enumeration formulas
under various geometric and combinatorial constraints.
Together, they provide a foundation and motivation for studying the existence of dessins with prescribed
uniform unicellular passports and trivial automorphism groups.

\subsection{Our Previous Work}

Examination of the regularity and automorphism groups of uniform \ddes suggests that,
as the genus increases, the proportion of regular dessins tends to decrease,
whereas the proportion of dessins with trivial automorphism groups tends to increase.

For a given uniform passport, the following questions naturally arise:
\begin{itemize}
\item Under what conditions do regular \ddes exist, and how many are there?
\item How are the automorphism groups distributed?
\item How does this distribution change with the genus?
\end{itemize}
\HL

In our earlier paper \cite{Ohnishi26}, we proposed a conjectural picture, established several supporting results,
and refined the picture accordingly.
It was updated in our previous paper \cite{Ohnishi2606} and is further refined in the present paper.
The current conjectural picture is shown in \tabref{tab:genusautd}.

\begin{table}[htbp]
\centering
\begin{tabular}{|c|l|c|c|c|}
\hline
Genus  & \multicolumn{1}{c|}{Passport} & $\Aut \msD \cong \{ 1 \}$ & $1 {<} \vt{\AD} {<} n$ & $\vt{\AD} {=} n${\small\ (regular)} \\
\hline
$0$ & & \multicolumn{2}{c|}{-- \textsuperscript{\ref{itm:g0reg}}} & $\checkmark$\textsuperscript{\ref{itm:g0reg}} \\ \hline
\multirow{2}{*}{$1$} & $[a^{p}, b^{q}, n]$ & \multirow{2}{*}{-- \textsuperscript{\ref{itm:g1triv}}}  & -- \textsuperscript{\ref{itm:g1nreg}} & $\checkmark$\textsuperscript{\ref{itm:anyg}} \\ \cline{2-2}\cline{4-5}
 & $[a^{p}, b^{q}, c^{r}]$  ($p, q, r \ge 2$)&   & $\checkmark$\textsuperscript{\ref{itm:g1nreg}} & $\diamond$ \\ \hline
\multirow{6}{*}{$\ge 2$} & $[n, n, n]$ ($n$: prime) & \multirow{6}{*}{$\checkmark\textsuperscript{\ref{itm:g2triv}}$} & --\textsuperscript{\ref{itm:g2nnn}} & \multirow{3}{*}{$\checkmark$\textsuperscript{\ref{itm:anyg}}} \\ \cline{2-2}\cline{4-4}
 & $[n, n, n]$ ($n$: composite) &  & $\checkmark$\textsuperscript{\ref{itm:g2nnn}} & \\ \cline{2-2}\cline{4-4}
 & $[n, b^{q}, n]$ ($q \ge 2$) &  & $\checkmark$* & \\ \cline{2-2}\cline{4-5}
 & $[b^{q}, b^{q}, n]$ ($q \ge 2$) &  & $\diamond$\textsuperscript{\ref{itm:g2trivonly}} & --\textsuperscript{\ref{itm:anyg}} \\ \cline{2-2}\cline{4-5}
 & $[a^{p}, b^{q}, n]$ ($q > p \ge 2$) &  & $\checkmark$* & \makecell{$\checkmark$ ($\gcd(p,q)=1$)\textsuperscript{\ref{itm:anyg}}\\--\,\, ($\gcd(p,q)\ne 1$)} \\ \cline{2-5}
 &$[a^{p}, b^{q}, c^{r}]$ ($p, q, r \ge 2$) & $\diamond$\textsuperscript{\ref{itm:g2nontrivonly}} & $\checkmark$* & $\diamond$ \\
\hline
\end{tabular}\\
\begin{flushleft}
\hspace{1.5em}$n$: number of edges of the dessin d'enfant \\
\hspace{1.5em}$\checkmark$: always occurs,\q --: never occurs,\q $\diamond$: depends on the passport.
\end{flushleft}
\begin{flushleft}
In this table, we assume that $n\ge 2$. If $n=1$ is allowed, the unique dessin
with passport $[1,1,1]$ belongs to both the
$\AD\cong\{1\}$ column and the
$\lvert\Aut\msD\rvert=n$ (regular) column.
\end{flushleft}
\HL
\caption{Distribution of automorphism groups for uniform passports}
\label{tab:genusautd}
\end{table}

Let $n$ be the number of edges of the dessin. Asterisks (*) indicate statements that remain conjectural at the present stage.
The numbers in parentheses refer to the corresponding items listed below.

\begin{enumerate}
\item\label{itm:g0reg} Genus~$0$: all uniform dessins are regular and have automorphism groups of order~$n$,
isomorphic to their monodromy groups.
\item\label{itm:g1nreg} Genus~$1$: a uniform passport admits a nonregular dessin if and only if it does not contain
an element of type $n^{1}$ (that is, it is not unicellular).
\item\label{itm:g1triv} Genus~$1$: no uniform passport admits a dessin with trivial automorphism group.
\item\label{itm:g2triv} Genus $\ge 2$: every passport of the form $[a^{p}, b^{q}, n]$ (the unicellular case), or any
permutation thereof, admits a dessin with trivial automorphism group, and hence admits a nonregular dessin.
\item\label{itm:anyg} Any genus: a passport of the form $[a^{p}, b^{q}, n]$ ($p, q \ge 1$), or any permutation thereof, admits a regular dessin if and only if $\gcd(p, q) = 1$.
\item\label{itm:g2nnn} Genus $\ge 2$: a passport $[n,n,n]$ admits a \dde\ $\msD$ with $\vt{\AD}=r$
if and only if $r$ is a divisor of $n$. Therefore, it admits a nonregular \dde with nontrivial automorphism group
if and only if $n$ is composite.
\item\label{itm:g2trivonly} Genus $\ge 2$: there exist passports $[b^{q}, b^{q}, n]$ ($q \ge 2$) that admit only dessins with
trivial automorphism group.
\item\label{itm:g2nontrivonly} Genus $\ge 2$: there exist passports $[a^{p}, b^{q}, c^{r}]$ ($p,q,r \ge 2$) that admit no dessins with
trivial automorphism group.
\end{enumerate}

In \cite{Ohnishi26}, we proved \ref{itm:g0reg}, \ref{itm:g1nreg}, \ref{itm:g1triv}, and \ref{itm:anyg},
and partially proved \ref{itm:g2triv} for passports of the form $[n,b^{q},n]$ with $q\ge1$.

In \cite{Ohnishi2606}, we obtained a further partial result toward \ref{itm:g2triv} for passports of the form
$[b^{q},b^{q},n]$ with $q\ge2$.
Moreover, we proved \ref{itm:g2nnn} and provided examples for \ref{itm:g2trivonly} and
\ref{itm:g2nontrivonly}, which serve as counterexamples to our initial conjectural picture.

These results clarify several aspects of the regularity and automorphism groups of uniform dessins.
In our previous work, we described the distribution of automorphism groups in genera~$0$ and~$1$,
established a necessary and sufficient condition for a uniform unicellular passport
(that is, one containing $n^{1}$) to admit a regular dessin,
and provided a group-theoretic characterization of regularity for general uniform passports.

For unicellular passports of genus at least~$2$, our previous work also developed a new method
for proving the existence of dessins with trivial automorphism groups by estimating the number
of relevant permutations. Together with an analysis of the centralizer of the monodromy group,
this counting method yielded the corresponding existence results.

\subsection{Main Results}

The main results of this paper are as follows.
\HL

In \secref{sec:class4}, we prove the following theorem for the remaining family of uniform unicellular passports,
thereby completing the proof of statement \ref{itm:g2triv} in \tabref{tab:genusautd}.
\HL

\noindent
\textbf{\thmref{thm:class4}.}
\textit{\thmclassfour}
\HL

Combining this theorem with the results of \cite{Ohnishi26, Ohnishi2606} corresponding to \ref{itm:g2triv}, we obtain the
following general theorem for uniform unicellular passports. 
\HL

\noindent
\textbf{\thmref{thm:class1-4}.}
\textit{\thmclassonetofour}
\HL

For future work, two directions seem particularly promising.

The first is to establish the conjectural statements summarized in
\tabref{tab:genusautd}. It would be desirable to determine the conditions under which uniform passports
admit a nonregular dessin with nontrivial automorphism group. In addition, one needs to characterize
when a passport $[a^{p},b^{q},c^{r}]$ $(p,q,r\ge2)$ admits a dessin with trivial automorphism group.

The second is to obtain quantitative results on the distribution of regular
dessins and automorphism groups.
\OL

\section{Preliminaries}
\label{sec:preliminaries}

The definitions and results in this section are based primarily on \cite{Jones16} and \cite{Adrianov20}.

\begin{definition}[\Dde]\cite[Definition~2]{Jones16}
\label{def:dessin}
A \emph{\dde}, or simply a \emph{dessin}, is a map consisting of a connected, finite, bipartite graph embedded
in a connected, compact, oriented surface without boundary.
Here, a bipartite graph is a graph whose vertices can be colored black and white in such a way that each edge
joins a black vertex to a white vertex.
\end{definition}

Since a compact Riemann surface provides a suitable surface on which a \dde can be embedded,
one can draw a dessin corresponding to a \belyi pair $(X, \B)$, as described in
\secref{sec:belyi}.

\HL

A \emph{partition} $\la$ of a positive integer $n$, denoted by $\la \vdash n$, is a multiset of
positive integers whose sum is $n$, where the order of the parts is irrelevant.
 
\begin{definition}[Passport of a dessin]\cite[Definition~2.10]{Adrianov20}
\label{def:passport}
Let $n$ be the number of edges of a \dde.
The triple $[\lambda_{0}, \lambda_{1}, \lambda_{\infty}]$ of partitions $\lambda_{0}, \lambda_{1}, \lambda_{\infty} \vdash n$,
which correspond respectively to the valencies of the black vertices, the white vertices, and
the faces of the dessin, is called a $\emph{passport}$ of the dessin.
\end{definition}

Both dessins in \figref{fig:dessins-ex} have passport $[621^{4},42^{4},12]$,
while the dessins in \figref{fig:regular} have passport
$[4^{2},2^{4},4^{2}]$.
Throughout this paper, we use the standard abbreviated notation for partitions:
for example, $621^{4}=(6,2,1,1,1,1)$, $42^{4}=(4,2,2,2,2)$, $4^{2}=(4,4)$, and $2^{4}=(2,2,2,2)$.

Since $\lambda_{0}, \lambda_{1}, \lambda_{\infty}$ are partitions of $n$, we have

\begin{align}
\label{eq:lambdan}
\lvert \lambda_{0} \rvert = \lvert \lambda_{1} \rvert = \lvert \lambda_{\infty} \rvert = n,
\end{align}
where $\lvert \la \rvert$ denotes the sum of the parts of the partition $\la$. 

The numbers of vertices, faces, and edges are $l(\lambda_{0}) + l(\lambda_{1})$,
$l(\lambda_{\infty})$, and $n$, respectively, where $l(\la)$ denotes the number of parts
of the partition $\la$.
Therefore, Euler's formula gives
\begin{align}
l(\lambda_{0}) + l(\lambda_{1}) + l(\lambda_{\infty}) - n = 2 - 2g,
\end{align}
where $g$ is the genus of the underlying curve $X$. Hence
\begin{align}
\label{eq:lambdag}
g = \f{n - (l(\lambda_{0}) + l(\lambda_{1}) + l(\lambda_{\infty}))}{2} + 1.
\end{align}
Since $g$ is a nonnegative integer, it follows that
\begin{align}
\label{eq:lambdal}
l(\lambda_{0}) + l(\lambda_{1}) + l(\lambda_{\infty}) \le  n + 2, \q
l(\lambda_{0}) + l(\lambda_{1}) + l(\lambda_{\infty}) \equiv n \npmod{2}.
\end{align}

Note that not every triple of partitions satisfying \eqref{eq:lambdan} and \eqref{eq:lambdal}
is realized by a \dde.
For example, although $[2^{2}, 2^{2}, 31]$ formally satisfies these conditions and would yield genus~$0$,
there exists no dessin with this passport.
Similarly, the passport $[3^{2}, 3^{2}, 42]$, which would correspond to genus~$1$, admits no dessin.

These facts can be proved by showing that there exists no corresponding monodromy group
(see \defref{def:monog}) in each case.

\begin{definition}[Uniform passports and dessins]\cite[Remark~3.2]{Jones16}
The passport of a \dde given by
$[\lambda_{0}, \lambda_{1}, \lambda_{\infty}] = [a_{1}\cdots a_{p},\ b_{1}\cdots b_{q},\ c_{1}\cdots c_{r}]$
is called \emph{uniform} if
\begin{align}
a_{1} = \cdots = a_{p}, \q b_{1} = \cdots = b_{q}, \q c_{1} = \cdots = c_{r},
\end{align}
\ie if the passport takes the form $[a^{p}, b^{q}, c^{r}]$.

A \dde is called \emph{uniform} if it has a uniform passport.
\end{definition}

For a uniform passport $[a^{p}, b^{q}, c^{r}]$ with $n = pa = qb = rc$, realized by a dessin of genus $g$,
\eqref{eq:lambdag} becomes
\begin{align}
\label{eq:genus}
g = \f{n-(p+q+r)}{2} + 1.
\end{align}

\begin{definition}[Monodromy group of a \dde]\cite[2.1.1]{Jones16}
\label{def:monog}
Define two permutations $x$ and $y$ acting on the set of edges $E$ of a \dde $\msD$ as follows.
For each edge $e \in E$, define $x \cdot e$ and $y \cdot e$ to be the next edges around the unique black vertex and
the unique white vertex incident to $e$, respectively, following the counterclockwise orientation.

The \emph{monodromy group} of $\msD$ is the subgroup $G = \gen{x, y}$ generated by $x$ and $y$ in
the symmetric group $\Sym(E)$ of all permutations of $E$.
\end{definition}

Since a \dde is a connected graph, it follows that any edge in $E$ can be mapped to any other edge by the action of $G$.
Therefore, the monodromy group $G$ acts transitively on $E$.

\begin{rem}
Throughout this paper, permutations act on the left, and products are composed from right to left,
\ie $(xy) \cdot e = x \cdot (y \cdot e)$ for an edge $e$.
\end{rem}

Another important observation concerning the monodromy group is that, in addition to $x$ and $y$ encoding
the cycles of the black and white vertices, respectively, the permutation $z = (xy)^{-1}$ encodes the cycles
corresponding to the faces.
If a cycle of $z$ has length $c$, then the corresponding face has valency $c$
in the passport. Equivalently, the boundary of a face of valency $c$ traverses
$2c$ edge-sides, counted with multiplicity.

It is known that if a transitive permutation group $G \le \Sym(E)$ on a finite set $E$
is generated by two elements, then there exists a \dde whose monodromy group is $G$;
this follows from the correspondence between constellations and
ramified coverings of the sphere \cite[Proposition~1.2.15 and Construction~1.2.17]{Lando04}.

Therefore, studying groups that act transitively is essential for investigating the properties of \ddes and, consequently, of algebraic curves.

\begin{definition}[Regular dessins]\cite[2.1.2]{Jones16}
\label{def:regular}
A \dde is called \emph{regular} if its monodromy group acts freely (that is, semiregularly) on the set of its edges.
\end{definition}

This implies that the monodromy group of a regular dessin acts freely and transitively --- hence regularly ---
on its edges.

The following criterion relates the order of the monodromy group to the number of edges.

\begin{lemma}
\label{lem:order-n}
A \dde with $n$ edges is regular if and only if the order of its monodromy group is $n$.
\end{lemma}

\begin{proof}
See \cite[Lemma~2.7]{Ohnishi26}.
\end{proof}

\begin{definition}[Automorphism group of a \dde]\cite[2.1.2]{Jones16}
For a \dde $\msD$ with edge set $E$, an \emph{automorphism} of $\msD$ is a permutation of $E$ that preserves the cyclic
order of edges around each vertex, that is, which commutes with $x$ and $y$, or equivalently, commutes with $G$.
Thus we can define the \emph{automorphism group} of $\msD$ as the centralizer:

\begin{align}
\Aut \msD \ceq C_{\Sym(E)}(G) &= \{ c \in \Sym(E) \mid cg = gc\ \text{for all}\ g \in G \} \\
\label{eq:cxcy}
&= \{ c \in \Sym(E) \mid cx = xc,\ cy = yc \},
\end{align}
where $\Sym(E)$ is the symmetric group of all permutations of $E$.

When $\Aut \msD \cong \{1 \}$, we say that $\msD$ has a \emph{trivial} automorphism group.
\end{definition}

Since $\gen{x, y} = \gen{x, z} = \gen{y, z}$, where $z = (xy)^{-1}$,
both the monodromy group and the automorphism group are invariant under
permuting the roles of black vertices, white vertices, and faces.

The automorphism group of a \dde has the following properties:

\begin{itemize}
\item $\AD$ acts freely on the edges of $\msD$.
\item $\lvert \AD \rvert$ divides the number of edges.
\item If $\lvert \AD \rvert$ equals the number of edges, then $\AD \cong G$.
\end{itemize}

The following proposition is a basic result describing the relationship between regularity and automorphism groups.

\begin{prop}
\label{prop:regaut}
A \dde $\msD$ is regular if and only if its monodromy group $G$ is isomorphic to $\AD$.
\end{prop}

\begin{proof}
See \cite[Theorem~2.1]{Jones16}.
\end{proof}

This also implies that a \dde $\msD$ is regular if and only if $\lvert \AD \rvert$ equals the number of edges.

Note that regularity does not imply $G = \Aut \msD$;
these groups act on the edges of the dessin in different ways.

\begin{prop}
\label{prop:reguni}
If a \dde is regular, then it has a uniform passport.
\end{prop}

\begin{proof}
See \cite[Proposition~4.42]{Girondo12} and \cite[Proposition~2.10]{Ohnishi26}.
\end{proof}

The converse of this proposition does not hold in general.
Thus, uniformity is necessary but not sufficient for regularity.

\begin{exa}
The uniform passport $[4^{2}, 2^{4}, 4^{2}]$ (genus~$1$) corresponds to the two dessins shown in
\secref{sec:regaut}, \figref{fig:regular}.
The dessin on the left is regular, and both its monodromy group and its automorphism group are
isomorphic to $C_{4} \times C_{2}$, the direct product of cyclic groups of orders $4$ and $2$, respectively.
In contrast, the dessin on the right is not regular; its monodromy group is isomorphic to $(C_{4} \times C_{2}) \rtimes C_{2}$
and has order~$16$, whereas its automorphism group has order~$4$ and is isomorphic to $C_{2} \times C_{2}$.
\end{exa}

To prove the main results, we use the following two lemmas concerning the gamma and digamma functions
to obtain upper and lower bounds for the numbers of relevant permutations.

\begin{lemma}
\label{lem:gamma}
For the gamma function $\Gamma$ and $x>0$, the following hold.
\begin{enumerate}
\item\label{itm:gamma-1} For $0 < s < 1$,
\begin{align}
&\f{x}{(x+s)^{1-s}} < \f{\Gamma(x+s)}{\Gamma(x)} < x^{s}.
\end{align}
\item\label{itm:gamma-2}
\begin{align}
\f{\Gamma(x+1)}{\Gamma(x)} = x.
\end{align}
\item\label{itm:gamma-3} For $s > 1$,
\begin{align}
x^{s} < \f{\Gamma(x+s)}{\Gamma(x)} < (x+s)^{s}.
\end{align}
\end{enumerate}
\end{lemma}

\begin{proof}
Part~\ref{itm:gamma-1} is Wendel's inequality \cite{Wendel48},
and \ref{itm:gamma-2} is the standard functional equation for the gamma function.
We prove \ref{itm:gamma-3}.

Assume that $s > 1$ and write $s = m + r$, where $m \in \Zp$ and $0 \le r < 1$.

By \ref{itm:gamma-2},
\begin{align}
\label{eq:gammaxs}
\f{\Gamma(x+s)}{\Gamma(x)} &= \prod_{i=0}^{m-1}(x+i)\cdot \f{\Gamma(x+m+r)}{\Gamma(x+m)}.
\end{align}

If $r = 0$, then $m \ge 2$, and
\begin{align}
&\f{\Gamma(x+s)}{\Gamma(x)} = \prod_{i=0}^{m-1}(x+i), \\
&x ^{m} < \prod_{i=0}^{m-1}(x+i) < (x+m)^{m}.
\end{align}
Hence
\begin{align}
x^{s} = x^{m} < \f{\Gamma(x+s)}{\Gamma(x)} < (x+m)^{m} = (x+s)^{s}.
\end{align}

If $r > 0$, then by \eqref{eq:gammaxs} and \ref{itm:gamma-1},
\begin{align}
\f{\Gamma(x+s)}{\Gamma(x)} &< \prod_{i=0}^{m-1}(x+i)\cdot(x+m)^{r} \\
\label{eq:gammaxs1}
&< (x+s)^{m}(x+s)^{r} = (x+s)^{s}.
\end{align}

By \ref{itm:gamma-1},
\begin{align}
\label{eq:gammaxmr}
\f{\Gamma(x+m+r)}{\Gamma(x+m)} > \f{x+m}{(x+m+r)^{1-r}}.
\end{align}
On the other hand, by the weighted AM-GM inequality,
\begin{align}
(x+m-1)^{r}(x+m+r)^{1-r} &< r(x+m-1)+(1-r)(x+m+r) \\
&= x+m -r^{2} < x+m,
\end{align}
hence
\begin{align}
\f{x+m}{(x+m+r)^{1-r}} > (x+m-1)^{r} \ge x^{r}.
\end{align}
Together with \eqref{eq:gammaxmr}, this gives
\begin{align}
\f{\Gamma(x+m+r)}{\Gamma(x+m)} > x^{r}.
\end{align}
Thus, by \eqref{eq:gammaxs},
\begin{align}
\label{eq:gammaxs2}
\f{\Gamma(x+s)}{\Gamma(x)} > \prod_{i=0}^{m-1}(x+i) \cdot x^{r} \ge x^{m}x^{r} = x^{s}.
\end{align}
By \eqref{eq:gammaxs1} and \eqref{eq:gammaxs2} we obtain
\begin{equation*} 
x^{s} < \f{\Gamma(x+s)}{\Gamma(x)} < (x+s)^{s}.\qedhere
\end{equation*}
\end{proof}
\HL

\begin{lemma}
\label{lem:digamma}
Let $\psi$ denote the digamma function defined by
\begin{align}
\psi(x)=\dv{x}\log\Gamma(x)=\f{\Gamma'(x)}{\Gamma(x)}.
\end{align}
Then
\begin{align}
\log\left(x-\f{1}{2}\right)<\psi(x)<\log x \q \left(x>\f{1}{2}\right).
\end{align}
\end{lemma}

\begin{proof}
It is well known (see, e.g., \cite[p.~28 and Proposition~9.6.43]{Cohen07}) that
\begin{align}
\label{eq:logpsipsi}
\psi(x) &= \int_{0}^{\infty}\left(\f{e^{-t}}{t}-\f{e^{-xt}}{1-e^{-t}}\right)dt, \\
\label{eq:logpsilog}
\log x &= \int_{0}^{\infty}\f{e^{-t}-e^{-xt}}{t}dt
\end{align}
for $x > 0$. Hence
\begin{align}
\label{eq:logpsi1}
\log x - \psi(x) = \int_{0}^{\infty}e^{-xt}\left(\f{1}{1-e^{-t}}-\f{1}{t}\right)dt\q(x > 0).
\end{align}

Let $f(t) = t - (1-e^{-t})$. Then
\begin{align}
f(0) &= 0,\\
f'(t) &= 1-e^{-t} > 0 \q (t > 0).
\end{align}
Hence $t > 1-e^{-t}$ for $t > 0$. Therefore,
\begin{align}
\f{1}{1-e^{-t}}-\f{1}{t} > 0\q(t > 0).
\end{align}
Thus, by \eqref{eq:logpsi1},
\begin{align}
\label{eq:logpsi2}
\log x - \psi(x) > 0 \q (x > 0).
\end{align}

By \eqref{eq:logpsipsi} and \eqref{eq:logpsilog},
\begin{align}
\psi(x) - \log\left(x-\f{1}{2}\right)
&= \int_{0}^{\infty}\left(\f{e^{-\left(x-\f{1}{2}\right)t}}{t}-\f{e^{-xt}}{1-e^{-t}}\right)dt \\
\label{eq:logpsi3}
&= \int_{0}^{\infty}e^{-xt}\left(\f{e^{\f{t}{2}}}{t}-\f{1}{1-e^{-t}}\right)dt \q \left(x > \f{1}{2}\right).
\end{align}

Let $g(u) = \sinh(u) - u$. Then
\begin{align}
g(0) &= 0 , \\
g'(u) &= \cosh(u) - 1 > 0 \q (u > 0).
\end{align}
Hence
\begin{align}
g(u) > 0\q (u > 0).
\end{align}
Therefore, for $t > 0$,
\begin{align}
e^{\f{t}{2}}-e^{-\f{t}{2}} = 2\sinh\f{t}{2} > 2\cdot \f{t}{2} = t,
\end{align}
hence
\begin{align}
1-e^{-t} > \f{t}{e^{\f{t}{2}}},
\end{align}
and therefore,
\begin{align}
\f{e^{\f{t}{2}}}{t}-\f{1}{1-e^{-t}} > 0 \q (t > 0).
\end{align}
Thus, by \eqref{eq:logpsi3} we obtain
\begin{align}
\psi(x) - \log\left(x-\f{1}{2}\right) > 0 \q \left(x > \f{1}{2}\right).
\end{align}
Combining this with \eqref{eq:logpsi2}, we conclude that
\begin{equation*} 
\log\left(x-\f{1}{2}\right)<\psi(x)<\log x \q \left(x>\f{1}{2}\right).\qedhere
\end{equation*}
\end{proof}
\OL


\section{Counting Arguments for Trivial Automorphism Groups}
\label{sec:estimation}

\subsection{Setup and Notation}
\label{sec:setup}

In this section, we consider uniform passports $[n, b^{q}, a^{p}]$ of genus $g \ge 2$
with $n=pa=qb$ and $2 \le p<q<n$ (equivalently, $2 \le b<a<n$).
We take $x$, $y$, and $z=(xy)^{-1}$ to have cycle types
$(n)$, $(b^{q})$, and $(a^{p})$, respectively.
Since $x$ is an $n$-cycle, $\gen{x}$ acts transitively on the edge set $E$.
Hence the monodromy group $G=\gen{x,y}$ also acts transitively on $E$.

By \eqref{eq:genus}, the genus is
\begin{align}
\label{eq:genus-uc}
g &= \f{n-(p+q+1)}{2} + 1 =  \f{n-(p+q)+1}{2}.
\end{align}

Since the monodromy group and the automorphism group are invariant under permutations of
the entries in the passport, it suffices to consider passports of this form;
the case $[a^{p}, b^{q}, n]$ follows by symmetry.

Let $S_{n}$ be the symmetric group on $E = \{1,\ldots,n\}$,
and fix an $n$-cycle $x\in S_{n}$, for example $x=\sigman$.
We define the following subsets of $S_{n}$:
\begin{align}
\label{eq:tncd}
\begin{aligned}
T &= T(b,q) \ceq \{ y\in S_{n} \mid y \text{ has cycle type } (b^{q}) \}, \\
N &= N(b,q,a) \ceq \{ y \in T(b,q) \mid (xy)^{-1} \text{ has cycle type } (a^{p}) \}, \\
C &= C(b,q) \ceq \{ y \in T(b,q) \mid \AD \ncong \{1\} \}, \\
D &= D(n) \ceq \{ y \in S_{n} \mid \AD \ncong \{1\} \},
\end{aligned}
\end{align}
where $\msD$ denotes the \dde corresponding to the permutation pair $(x,y)$, whose
monodromy group is $\gen{x,y}$.

\figref{fig:tncd} illustrates the relationships among these sets.
The hatched region corresponds to $C(b, q)$.

\begin{figure}[htbp]
\centering
\includegraphics[width=0.41\textwidth]{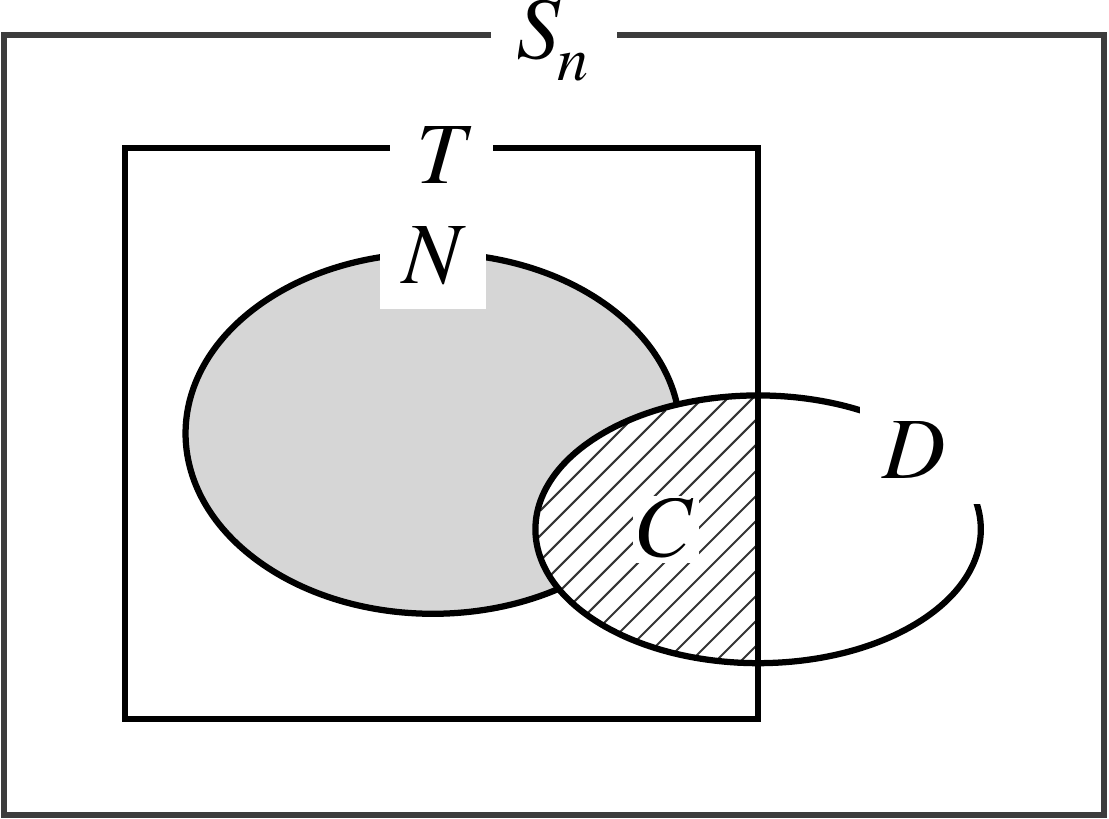} 
\caption{The sets related to the passport $[n, b^{q}, a^{p}]$}
\label{fig:tncd}
\end{figure}

For any $t \in S_{n}$ satisfying $txt^{-1}=x$, the permutation pairs
$(x,y)$ and $(x,tyt^{-1})$ define isomorphic dessins.
Thus, the cardinalities of the above sets count permutations and do not,
in general, equal the numbers of isomorphism classes of the corresponding dessins.
Nevertheless, comparing these cardinalities yields sufficient conditions
for the existence of dessins with trivial automorphism group.

The existence of a dessin $\msD$ with passport $[n, b^{q}, a^{p}]$ and trivial automorphism group
is equivalent to
\begin{align}
N \cap C^{c} \ne \emptyset
\end{align}
(the gray region in \figref{fig:tncd}). Since $C  = D \cap T$ and $N \subset T$, this is equivalent to
\begin{align}
N \cap D^{c} \ne \emptyset.
\end{align}
Thus, to prove the existence of such a dessin,
it suffices to show that
\begin{align}
\vt{N} > \vt{C}.
\end{align}
A stronger but simpler sufficient condition is
\begin{align}
\vt{N} > \vt{D}.
\end{align}
\HL


\subsection{\texorpdfstring{Lower Bound for $\vt{N}$}{Lower Bound for |N|}}
\label{sec:Nlower}

A useful tool for estimating
\begin{align}
\vt{N(b, q, a)} = \# \{ y \in T(b,q) \mid (xy)^{-1} \text{ has cycle type } (a^{p}) \}
\end{align}
is the following theorem from \cite{Goupil98}.

\begin{thm}
\label{thm:goupil}
Let $\lambda = (\lambda_{1}, \dotsc, \lambda_{l})$ and $\mu = (\mu_{1}, \dotsc, \mu_{m})$
be partitions of $n$.
Define the genus associated with the pair $(\lambda, \mu)$ by

\begin{align}
g = \frac{n-(l+m)+1}{2},
\end{align}
and assume that $g \in \Zz$.

Let $C_\lambda$ and $C_\mu$ denote the conjugacy classes in $S_{n}$
consisting of permutations of cycle types $\lambda$ and $\mu$,
respectively. 
Let $c_{\lambda \mu}^{n}$ denote the number of solutions $(\sigma, \rho) \in C_{\lambda} \times C_{\mu}$
to the equation $\sigma \rho = \pi$,
where $\pi$ is a fixed $n$-cycle in $S_{n}$.
Then $c_{\lambda \mu}^{n}$ is given by

\begin{align}
c_{\la \mu}^{n} = \f{n}{z_{\la}z_{\mu}2^{2g}}\!\sum_{\substack{g_{1},g_{2}\ge 0\\ g_{1}+g_{2}=g}}(l+2g_{1}-1)! (m+2g_{2}-1)!
\!\!\sum_{\substack{(i_{1},\dotsc,i_{l})\vDash g_{1}\\(j_{1},\dotsc,j_{m})\vDash g_{2}}}
\prod_{k=1}^{l}\binom{\la_{k}}{2i_{k}+1}\prod_{k=1}^{m}\binom{\mu_{k}}{2j_{k}+1}.
\end{align}
Here $P \vDash n$ denotes a composition of $n$, that is, a finite sequence of nonnegative integers
summing to $n$, where the order of the terms matters. We adopt the convention that $\binom{a}{b}=0$ if $b>a$.
Moreover, for a partition $\lambda = 1^{\A_{1}}\cdots n^{\A_{n}}$, where $\A_{i}$ denotes the multiplicity of $i$, we define
$z_{\lambda} = \prod_{i} \A_{i}! \, i^{\A_{i}}$.
\end{thm}

\begin{proof}
See \cite[Theorem~2.1]{Goupil98}.
\end{proof}

For the passport $[n,b^{q},a^{p}]$, fix an $n$-cycle $x$.
By \thmref{thm:goupil},
$c_{(a^{p})(b^{q})}^{n}$ is the number of pairs $(\sigma,\rho)$ such that
$\sigma$ and $\rho$ have cycle types $(a^{p})$ and $(b^{q})$, respectively, and satisfy $\sigma\rho=x$.
Moreover,
\begin{align}
\sigma\rho=x \iff x\rho^{-1}\sigma^{-1}=1.
\end{align}
Since a permutation and its inverse have the same cycle type,
$c_{(a^{p})(b^{q})}^{n}$ is precisely the number of pairs $(\sigma,\rho)$
of cycle types $(a^{p})$ and $(b^{q})$, respectively, such that
$x\rho^{-1}\sigma^{-1}=1$.
If we set $y=\rho^{-1}$, then each such pair is uniquely determined by $y$, and the condition is equivalent to $(xy)^{-1}$
having cycle type $(a^{p})$. Thus,
\begin{align}
c_{(a^{p})(b^{q})}^{n}=\vt{N(b,q,a)}.
\end{align}

Therefore, by \thmref{thm:goupil},
\begin{multline}
\vt{N(b, q, a)} = c_{(a^{p})(b^{q})}^{n}
= \f{n}{2^{2g}a^{p}p!b^{q}q!}\sum_{\substack{g_{1},g_{2}\ge 0\\ g_{1}+g_{2}=g}}(p+2g_{1}-1)! (q+2g_{2}-1)! \\
\label{eq:Nbqa}
\cdot\underbrace{\sum_{\substack{(i_{1},\dotsc,i_{p})\vDash g_{1}\\(j_{1},\dotsc,j_{q})\vDash g_{2}}}
\underbrace{\prod_{k=1}^{p}\binom{a}{2i_{k}+1}\prod_{k=1}^{q}\binom{b}{2j_{k}+1}}_{\mcP}}_{\mcS},
\end{multline}
where
\begin{align}
g = \f{n - (p+q)+1}{2}.
\end{align}

In the inner sum $\mcS$, the product $\mcP$ is zero if and only if $a<2i_{k}+1$ for some $1\le k\le p$, or
$b<2j_{k}+1$ for some $1\le k\le q$.
Therefore,
\begin{align}
\mcP > 0 &\iff a \ge 2i_{k}+1 \text{ for all } 1 \le k \le p, \text{ and } b \ge 2j_{k}+1 \text{ for all } 1 \le k \le q.
\end{align}
Hence, $\mcS$ is zero if and only if there is no pair of compositions satisfying the above condition, that is,
\begin{align}
\mcS = 0 \iff g_{1} > p\fl{\f{a-1}{2}} \text{ or } g_{2} > q\fl{\f{b-1}{2}}.
\end{align}
Since $q \ge p+1$ by assumption,
\begin{align}
p\fl{\f{a-1}{2}} - g_{1} &\ge p\cdot\f{a-2}{2} - g_{1} \ge p\cdot\f{a-2}{2} - g \\
&= \f{n-2p}{2} - \f{n - (p+q)+1}{2} = \f{q-(p+1)}{2} \ge 0.
\end{align}
Thus, the condition $g_{1} > p\fl{(a-1)/2}$ cannot occur, and hence
\begin{align}
\label{eq:inner0}
\mcS = 0 \iff g_{2} > q\fl{\f{b-1}{2}}.
\end{align}

Taking only the term corresponding to $(g_{1},g_{2})=(0,g)$ in the outer sum,
we obtain the lower bound
\begin{align}
\vt{N(b, q, a)}
&\ge \f{n}{2^{2g}a^{p}p!b^{q}q!}(p-1)! (q+2g-1)!
\sum_{\substack{(i_{1},\dotsc,i_{p})\vDash 0\\(j_{1},\dotsc,j_{q})\vDash g}}
\prod_{k=1}^{p}\binom{a}{2i_{k}+1}\prod_{k=1}^{q}\binom{b}{2j_{k}+1} \\
&= \f{n}{2^{2g}a^{p}p!b^{q}q!}(p-1)! (q+n-(p+q)+1-1)!
\sum_{(j_{1},\dotsc,j_{q})\vDash g}
\prod_{k=1}^{p}\binom{a}{1}\prod_{k=1}^{q}\binom{b}{2j_{k}+1} \\
\label{eq:N0g}
&= \f{a(n-p)!}{2^{2g}b^{q}q!}
\sum_{(j_{1},\dotsc,j_{q})\vDash g}
\prod_{k=1}^{q}\binom{b}{2j_{k}+1}.
\end{align}
This lower bound is used in \secref{sec:class4b3}.

On the other hand, taking only the term corresponding to $(g_{1},g_{2})=(g,0)$ in the outer sum,
we obtain the lower bound
\begin{align}
\vt{N(b, q, a)} &\ge \f{n}{2^{2g}a^{p}p!b^{q}q!}(p+2g-1)! (q-1)!
\sum_{\substack{(i_{1},\dotsc,i_{p})\vDash g\\(j_{1},\dotsc,j_{q})\vDash 0}}
\prod_{k=1}^{p}\binom{a}{2i_{k}+1}\prod_{k=1}^{q}\binom{b}{2j_{k}+1} \\
&= \f{n}{2^{2g}a^{p}p!b^{q}q!}(p+n-(p+q)+1-1)! (q-1)!
\sum_{(i_{1},\dotsc,i_{p})\vDash g}
\prod_{k=1}^{p}\binom{a}{2i_{k}+1}\prod_{k=1}^{q}\binom{b}{1} \\
\label{eq:Ng0}
&= \f{b(n-q)!}{2^{2g}a^{p}p!}
\sum_{(i_{1},\dotsc,i_{p})\vDash g}
\prod_{k=1}^{p}\binom{a}{2i_{k}+1}.
\end{align}

In particular, when $b=2$, the condition \eqref{eq:inner0} becomes
\begin{align}
\mcS = 0 \iff g_{2} > q\fl{\f{2-1}{2}} = 0.
\end{align}
Thus, the only nonzero term in the outer sum is the one corresponding to
$(g_{1},g_{2})=(g,0)$.
Hence equality holds in \eqref{eq:Ng0}, and
\begin{align}
\vt{N(2, q, a)} &= \f{2(n-q)!}{2^{2g}a^{p}p!}
\sum_{(i_{1},\dotsc,i_{p})\vDash g}
\prod_{k=1}^{p}\binom{a}{2i_{k}+1} \\
\label{eq:Ng0b2}
&= \f{q!}{2^{q-p}a^{p}p!}
\sum_{(i_{1},\dotsc,i_{p})\vDash g}
\prod_{k=1}^{p}\binom{a}{2i_{k}+1}.
\end{align}
This exact value is used in \secref{sec:class4b2age5}.

In \eqref{eq:Ng0}, we restrict the sum to compositions whose parts are as equal as possible.
Let
\begin{align}
\label{eq:uv}
u = \fl{\f{g}{p}} = \fl{\f{n-p-q+1}{2p}} = \fl{\f{pa-p-q+1}{2p}} = \fl{\f{a-1}{2} - \f{q-1}{2p}}, \q v = g - pu.
\end{align}
Then the sum includes all compositions obtained by arranging $p-v$ copies of $u$ and $v$ copies of $u+1$.
Restricting the sum to these compositions, we obtain the simpler lower bound
\begin{align}
\label{eq:NB}
\vt{N(b, q, a)} &\ge \f{b(n-q)!}{2^{n-p-q+1}a^{p}p!}B(a, p, q), \\
\label{eq:Bapq}
B(a, p, q) &\ceq
\begin{dcases}
\binom{p}{v}\binom{a}{2u+1}^{p-v}\binom{a}{2u+3}^{v} & (v > 0), \\
\binom{a}{2u+1}^{p} & (v = 0).
\end{dcases}
\end{align}
\HL

\noindent
(i) When $v > 0$

Since $p \le q-1$,
\begin{align}
u = \f{g-v}{p} < \f{g}{p} = \f{a-1}{2} - \f{q-1}{2p} \le \f{a-1}{2} - \f{1}{2} = \f{a-2}{2}.
\end{align}
Since $a$ and $u$ are integers, we have $u \le (a-3)/2$, hence
\begin{align}
a \ge 2u+3.
\end{align}
Moreover,
\begin{align}
u = \fl{\f{g}{p}} > \f{g}{p}-1,
\end{align}
hence
\begin{align}
1 \le v &= g-pu < p, \\
\binom{p}{v} &\ge \binom{p}{1} = p.
\end{align}
\HL

\noindent
(i-1) When $a > 2u+3$, we have
\begin{align}
1 \le 2u+1 < 2u+3 < a,
\end{align}
hence
\begin{align}
\binom{a}{2u+1} \ge \binom{a}{1},\q \binom{a}{2u+3} \ge \binom{a}{1}.
\end{align}
Therefore,
\begin{align}
B(a, p, q) &= \binom{p}{v}\binom{a}{2u+1}^{p-v}\binom{a}{2u+3}^{v}
\ge p\binom{a}{1}^{p-v}\binom{a}{1}^{v} = pa^{p}.
\end{align}
\HL

\noindent
(i-2) When $a = 2u+3$,  we have
\begin{align}
B(a, p, q) &= \binom{p}{v}\binom{a}{2u+1}^{p-v}\binom{a}{2u+3}^{v}
= \binom{p}{v}\binom{a}{a-2}^{p-v}\binom{a}{a}^{v} \\
&= \binom{p}{v}\binom{a}{a-2}^{p-v} \ge p\binom{a}{a-2}^{p-v} = p\left(\f{a(a-1)}{2}\right)^{p-v}.
\end{align}
Moreover, since $q \ge p+1$,
\begin{align}
p - v &= p - (g-pu) = p - \f{pa-(p+q)+1}{2} + p\cdot\f{a-3}{2}
= \f{q-1}{2} \ge \f{p}{2}.
\end{align}
Thus,
\begin{align}
B(a, p, q) \ge p\left(\f{a(a-1)}{2}\right)^{\f{p}{2}}.
\end{align}
\HL

\noindent
(ii) When $v = 0$

Since $u = g/p$ and $g \ge 2$,
\begin{align}
2u+1 &= \f{2g}{p} + 1  > 1, \\
2u+1 &= \f{2g}{p} + 1 = \f{n-p-q+1}{p}+1 = a - 1-\f{q-1}{p} + 1 < a.
\end{align}
Since $2u+1$ is an integer, we have
\begin{align}
2 \le 2u+1 \le a-1,
\end{align}
hence
\begin{align}
B(a, p, q) &= \binom{a}{2u+1}^{p} \ge \binom{a}{1}^{p} = a^{p}.
\end{align}

By (i-1), (i-2), and (ii),
\begin{align}
B(a, p, q) \ge 
\begin{dcases}
pa^{p} & (v > 0,\ a > 2u+3), \\
p\left(\f{a(a-1)}{2}\right)^{\f{p}{2}} & (v > 0,\ a=2u+3), \\
a^{p} & (v=0).
\end{dcases}
\end{align}
Since $p \ge 2$, we have $pa^{p}>a^{p}$. Therefore
\begin{align}
B(a, p, q) &\ge \min \left\{ a^{p},\ p\left(\f{a(a-1)}{2}\right)^{\f{p}{2}} \right\}
= a^{p}\min \left\{ 1,\ p\left(\f{a-1}{2a}\right)^{\f{p}{2}} \right\}.
\end{align}
Here,
\begin{align}
p\left(\f{a-1}{2a}\right)^{\f{p}{2}} < p\left(\f{a}{2a}\right)^{\f{p}{2}} = \f{p}{2^{\f{p}{2}}}.
\end{align}

Let $f(p) = p/2^{p/2}$. Then
\begin{align}
f(2) &= 1,\q f(3) = \f{3}{2\sqrt{2}} > 1, \\
\f{f(p+1)}{f(p)} &= \f{p+1}{2^{\f{p+1}{2}}}\cdot\f{2^{\f{p}{2}}}{p} = \f{1+\f{1}{p}}{\sqrt{2}}
\le \f{1+\f{1}{3}}{\sqrt{2}} < 1 \q (p \ge 3).
\end{align}
Hence $f(p)$ is strictly decreasing for $p\ge3$, and therefore
\begin{align}
p\left(\f{a-1}{2a}\right)^{\f{p}{2}} \le \f{3}{2\sqrt2} \q (p \ge 2).
\end{align}
Thus,
\begin{align}
\f{2\sqrt2}{3}p\left(\f{a-1}{2a}\right)^{\f{p}{2}} \le 1.
\end{align}
Therefore,
\begin{align}
B(a,p,q) &\ge a^{p} \min\left\{1,\ p\left(\f{a-1}{2a}\right)^{\f{p}{2}}\right\}
\ge a^{p} \min\left\{1,\ \f{2\sqrt2}{3}p\left(\f{a-1}{2a}\right)^{\f{p}{2}}\right\} \\
&= \f{2\sqrt2}{3} a^{p}p \left(\f{a-1}{2a}\right)^{\f{p}{2}}.
\end{align}

Combining this with \eqref{eq:NB} and collecting the terms depending on $a$ and $p$, we obtain
\begin{align}
\vt{N(b, q, a)} &\ge \f{b(n-q)!}{2^{n-p-q+1}a^{p}p!}\cdot\f{2\sqrt{2}}{3}a^{p}p\left(\f{a-1}{2a}\right)^{\f{p}{2}} \\
\label{eq:Fnabp}
&= \f{\sqrt{2}b(n-q)!}{3\cdot 2^{n-q}}\cdot\f{2^{p}}{(p-1)!}\left(\f{a-1}{2a}\right)^{\f{p}{2}}.
\end{align}

Let
\begin{align}
h(p) \ceq \f{2^{p}}{(p-1)!}\left(\f{a-1}{2a}\right)^{\f{p}{2}}.
\end{align}
Then
\begin{align}
\f{h(p+1)}{h(p)} &= \f{2^{p+1}}{p!}\left(\f{a-1}{2a}\right)^{\f{p+1}{2}}\f{(p-1)!}{2^{p}}\left(\f{a-1}{2a}\right)^{-\f{p}{2}}
= \f{2}{p}\left(\f{a-1}{2a}\right)^{\f{1}{2}} < 1.
\end{align}
Hence $h(p)$ is strictly decreasing in $p$. Since $p \le q-1$,
\begin{align}
h(p) &\ge h(q-1) = \f{2^{q-1}}{(q-2)!}\left(\f{a-1}{2a}\right)^{\f{q-1}{2}}.
\end{align}
Moreover, since $a \ge b+1$,
\begin{align}
\f{a-1}{2a} = \f{1}{2} - \f{1}{2a} \ge \f{1}{2} - \f{1}{2(b+1)} = \f{b}{2(b+1)}.
\end{align}
Combining these estimates with \eqref{eq:Fnabp}, we obtain the lower bound
\begin{align}
\vt{N(b, q, a)} &\ge \f{\sqrt{2}b(n-q)!}{3\cdot 2^{n-q}}\cdot
\f{2^{q-1}}{(q-2)!}\left(\f{b}{2(b+1)}\right)^{\f{q-1}{2}} \\
\label{eq:Nbabq}
&= \f{b(q(b-1))!}{2^{q(b-2)}3\sqrt{2}(q-2)!}\left(\f{b}{2(b+1)}\right)^{\f{q-1}{2}},
\end{align}
which depends only on $b$ and $q$. This lower bound is used in \secref{sec:class4b4} and \secref{sec:class4bge5}.
\HL


\subsection{\texorpdfstring{Upper Bound for $\vt{D}$}{Upper Bound for |D|}}

To estimate
\begin{align}
\vt{D(n)} = \# \{ y \in S_{n} \mid \AD \ncong \{1\} \},
\end{align}
we use the following facts obtained in \cite{Ohnishi2606}.

\begin{prop}
\label{prop:Dk}
Let $n \in \Zp$ with $n \ge 2$, let $x \in S_{n}$ be an $n$-cycle, and for each $1 \le k \le n-1$, define
\begin{align}
\label{eq:defDk}
D_{k} \ceq \{\,c\in S_{n} \mid cx^{k}=x^{k} c\,\},
\end{align}
the centralizer of $x^{k}$ in $S_{n}$.
Then the following hold.
\begin{enumerate}
\item \label{itm:xcom1}
$D_1=\{x^{k} \mid 0\le k\le n-1\}$.
\item \label{itm:xcom2}
For any $1\le k,l\le n-1$, if $k\mid l$, then $D_{k}\subset D_l$.
\item \label{itm:xcom3}
For any $1\le k\le n-1$, let $m=\gcd(n,k)$. Then
\begin{align}
D_{k}&=D_{m}, \\
\vt{D_{k}}&=\vt{D_{m}}=\left(\f{n}{m}\right)^m m!.
\end{align}
\end{enumerate}
\end{prop}

\begin{proof}
See \cite[Proposition~3.6]{Ohnishi2606}.
\end{proof}

\begin{cor}
\label{cor:ADtrivial}
Let $x, y \in S_{n}$, and assume that $x$ is an $n$-cycle. Let $\msD$ be the dessin corresponding to the permutation
pair $(x,y)$, whose monodromy group is $\gen{x,y}$.
Then the following hold.
\begin{enumerate}
\item \label{itm:ADtrivial1} $\AD = \{ x^{k} \mid 0 \le k \le n-1,\ x^{k}y = yx^{k} \}$.
\item \label{itm:ADtrivial2} $\AD$ is trivial if and only if $x^{k}y \ne yx^{k}$ for all $1 \le k \le n-1$.
\end{enumerate}
\end{cor}

\begin{proof}
See \cite[Corollary~3.7]{Ohnishi2606}.
\end{proof}

\begin{prop}
\label{prop:D}
Let $n \in \Zp$ with $n \ge 2$, let $x \in S_{n}$ be an $n$-cycle, and for each $1 \le k \le n-1$, define $D_{k}$ by \eqref{eq:defDk}.
Let $D$ be the set of permutations $y \in S_{n}$ such that the corresponding dessin $\msD$,
with monodromy group $\gen{x,y}$, has nontrivial automorphism group. Then
\begin{align}
D &= \bigcup_{\ell\colon \text{prime},\, \ell\, \mid\, n} D_{\f{n}{\ell}}, \\
\label{eq:vtD}
\vt{D} &\le \sum_{\ell\colon \text{prime},\, \ell\, \mid\, n} \ell^{\f{n}{\ell}}\left(\f{n}{\ell}\right)!.
\end{align}
\end{prop}

\begin{proof}
See \cite[Proposition~3.8]{Ohnishi2606}.
\end{proof}

\begin{lemma}
\label{lem:Dupper}
Let $n \in \Zp$ with $n \ge 2$, let $x \in S_{n}$ be an $n$-cycle, and for each $1 \le k \le n-1$, define $D_{k}$ by \eqref{eq:defDk}.
Let
\begin{align}
\kappa_{1} \ceq \f{2623}{1894} = 1.384\ldots, \q \kappa_{2} \ceq \f{972}{947} = 1.026\ldots.
\end{align}
Then
\begin{align}
\label{eq:Dupper}
\vt{D(n)} \le \kappa_{1}\cdot 2^{\f{n}{2}}\left(\f{n}{2}\right)!,
\end{align}
where, for noninteger $M$, we define $M! \ceq \Gamma(M+1)$.

In particular, if $n$ is odd, then
\begin{align}
\label{eq:Dupperodd}
\vt{D(n)} \le \kappa_{2}\cdot3^{\f{n}{3}}\left(\f{n}{3}\right)!.
\end{align}
\end{lemma}

\begin{proof}
See \cite[Lemma~3.9]{Ohnishi2606}.
\end{proof}
\HL


\subsection{\texorpdfstring{Upper Bound for $\vt{C}$}{Upper Bound for |C|}}

In this section, we derive an upper bound for
\begin{align}
\vt{C(b,q)} &= \# \{ y \in T(b,q) \mid \AD \ncong \{1\} \}.
\end{align}
Since $C(b,q) = T(b,q) \cap D(n)$, \propref{prop:D} gives
\begin{align}
C(b,q) &= T(b, q) \cap \Biggl(\bigcup_{\ell\colon \text{prime},\, \ell\, \mid\, n} D_{\f{n}{\ell}}\Biggr)
= \bigcup_{\ell\colon \text{prime},\, \ell\, \mid\, n} (T(b,q) \cap D_{\f{n}{\ell}}).
\end{align}
Let $C_{n/\ell}(b, q) \ceq T(b,q) \cap D_{n/\ell}$. Then
\begin{align}
C(b,q) &= \bigcup_{\ell\colon \text{prime},\, \ell\, \mid\, n} C_{\f{n}{\ell}}(b, q), \\
\label{eq:Cbqsum}
\vt{C(b,q)} &\le \sum_{\ell\colon \text{prime},\, \ell\, \mid\, n} \vt{C_{\f{n}{\ell}}(b, q)}.
\end{align}

The following proposition gives a formula for $\vt{C_{n/\ell}(b,q)}$.

\begin{prop}
\label{prop:Cnl}
Let $n \in \Zp$ with $n \ge 2$, let $\ell$ be a prime divisor of $n$, and let $x \in S_{n}$ be an $n$-cycle.
Let $T(b,q)$ be as defined in \eqref{eq:tncd}, with $n=bq$, and let $D_{k}$
be as defined in \eqref{eq:defDk} for each $1\le k\le n-1$.
Define
\begin{align}
C_{\f{n}{\ell}}(b, q) = T(b,q) \cap D_{\f{n}{\ell}} = \{ y \in S_{n} \mid y \text{ has cycle type } (b^{q}),\ x^{\f{n}{\ell}}y=yx^{\f{n}{\ell}} \}.
\end{align}
Then
\begin{align}
\vt{C_{\f{n}{\ell}}(b, q)} &=
\begin{dcases}
\left(\f{n}{\ell}\right)!\sum_{s=0}^{\fl{q/\ell}}\f{\ell^{(b-1)s}(\ell^{\f{b}{\ell}-1}(\ell-1))^{q-\ell s}}
{b^{s}s!\left(\f{b}{\ell}\right)^{q - \ell s}(q - \ell s)!} & (\ell \mid b), \\
\left(\f{n}{\ell}\right)!\f{\ell^{\f{(b-1)q}{\ell}}}{b^{\f{q}{\ell}}\left(\f{q}{\ell}\right)!} & (\ell \nmid b).
\end{dcases}
\end{align}
\end{prop}

\begin{proof}
Let $d = n/\ell$. The cycle type of $x^{d}$ is $(\ell^{d})$.
By \cite[Theorem~74.1]{Lipscomb}, $D_{d}$ is isomorphic to the wreath product
$C_{\ell}\wr S_{d}$. Hence $D_{d}$ is generated by the following two types
of permutations:
\begin{itemize}
\item permutations that interchange two of the $d$ cycles of $x^{d}$
while preserving the cyclic order within each cycle, and
\item permutations that cyclically rotate the $\ell$ elements within a single cycle of $x^{d}$.
\end{itemize}

\begin{exa}
Let $n=12$, $b=4$, $\ell=3$, and $x=(1\ 2\ \ldots\ 12)$. Then
\begin{align}
x^{d}=x^{4}=(1\ 5\ 9)(2\ 6\ 10)(3\ 7\ 11)(4\ 8\ 12).
\end{align}
\needspace{\baselineskip} 
The subgroup $D_4\le S_{12}$ consisting of the elements that commute with $x^4$ is generated by
\begin{itemize}
\item permutations that interchange adjacent cycles while preserving the cyclic order within each cycle:
\begin{align}
(1\ 2)(5\ 6)(9\ 10),\quad (2\ 3)(6\ 7)(10\ 11),\quad (3\ 4)(7\ 8)(11\ 12);
\end{align}
\item permutations that cyclically rotate the three elements within a single cycle:
\begin{align}
(1\ 5\ 9),\quad (2\ 6\ 10),\quad (3\ 7\ 11),\quad (4\ 8\ 12).
\end{align}
\end{itemize}
\end{exa}

Let $(i,j)$, where $i\in\Z/d\Z$ and $j\in\Z/\ell\Z$, denote the element in the cycle of $x^{d}$
 indexed by $i$ at the position indexed by $j$.
Then
\begin{align}
x^{d} &= (1\ 1{+}d\ \ldots\ 1{+}(\ell{-}1)d)(2\ 2{+}d\ \ldots\ 2{+}(\ell{-}1)d)\cdots(d\ 2d\ \ldots\ \ell d), \\
(i, j) &= i + dj + 1,
\end{align}
where we identify $\Z/d\Z$ and $\Z/\ell\Z$ with
$\{0,\ldots,d-1\}$ and $\{0,\ldots,\ell-1\}$, respectively.

Since $D_d$ is generated by the above two types of elements, every $y \in C_{d}(b,q) = D_{d} \cap T(b, q)$ permutes
the cycles of $x^d$ as whole cycles, and the amount of cyclic rotation is constant within each cycle. Let $\pi \in S_{d}$
denote the permutation of the cycles induced by $y$, and let $\alpha_i \in \Z/\ell\Z$ denote the rotation amount
in the cycle indexed by $i$. Then
\begin{align}
y\cdot(i, j) &= (\pi\cdot i,\ j+\A_i).
\end{align}

Suppose that the cycle decomposition of $\pi$ contains a cycle of length $r$
($1\le r\le d$). Let
\begin{align}
(i_1\ i_2\ \ldots\ i_r)
\end{align}
be such a cycle, and let $B_{i_j}$ denote the cycle of $x^{d}$
 indexed by $i_{j}$.
Then $y$ maps these cycles as
\begin{align}
B_{i_1}\to B_{i_2}\to\cdots\to B_{i_r}\to B_{i_1}.
\end{align}
For an element $u = (i_1,j)\in B_{i_1}$, we have
\begin{align}
y^r\cdot u
=(i_1,\ j+\A_{i_1}+\cdots+\A_{i_r}).
\end{align}

Let $\A=\A_{i_{1}}+\cdots+\A_{i_{r}} \in \Z/\ell\Z$.

If $\A=0$, then
\begin{align}
y^{r}\cdot u=u.
\end{align}
Moreover, since $r$ is the length of the cycle
$(i_{1}\ i_{2}\ \ldots\ i_{r})$ of $\pi$, no positive power $y^{s}$
with $s<r$ maps $u$ back to itself.
Thus, $u$ lies in a cycle of length $r$ in $y$.
Since $y$ has cycle type $(b^{q})$, we obtain $r=b$.

If $\A\ne0$, then $y^r$ maps $u$ to another element of the cycle $B_{i_1}$.
Since $\ell$ is prime, the additive order of $\A$ in
$\Z/\ell\Z$ is $\ell$.
Thus, the least positive integer $s$ such that $(y^{r})^{s}\cdot u=u$ is $\ell$.
Hence the cycle of $y$ containing $u$ has length $r\ell$.
Since $y$ has cycle type $(b^{q})$, we obtain
\begin{align}
r\ell=b,
\end{align}
and hence $r=b/\ell$. In particular, $\ell\mid b$.
Thus, every cycle of $\pi$ has length either $b$ or $b/\ell$, where the latter occurs only if $\ell \mid b$.
For a cycle of length $b$, we have
$\A_{i_1}+\cdots+\A_{i_b}=0$,
whereas for a cycle of length $b/\ell$, we have
$\A_{i_1}+\cdots+\A_{i_{b/\ell}}\ne0$.
\HL

\noindent
(i) The case $\ell \mid b$

Suppose that $\pi$ has $s$ cycles of length $b$ and $t$ cycles of length $b/\ell$.
Since these cycles partition the $d$ cycles of $x^{d}$,
\begin{align}
bs + \f{b}{\ell}t = d.
\end{align}
Since $d = n/\ell$,
\begin{align}
t = \f{d\ell}{b}-\ell s = \f{n}{b} - \ell s = q - \ell s.
\end{align}
The number of such permutations $\pi$ is
\begin{align}
\label{eq:numofpi}
\f{d!}{b^{s}s!\left(\f{b}{\ell}\right)^{t}t!} = \f{\left(\f{n}{\ell}\right)!}{b^{s}s!\left(\f{b}{\ell}\right)^{q-\ell s}(q-\ell s)!}.
\end{align}
Moreover,
\begin{align}
0 \le s \le \fl{\f{d}{b}} = \fl{\f{n}{b\ell}} = \fl{\f{q}{\ell}}.
\end{align}

For each cycle of length $b$ in $\pi$, let
$\B_1,\ldots,\B_b$ be the rotation amounts of the corresponding cycles of $x^d$.
Then
\begin{align}
\B_1+\cdots+\B_b\equiv0\pmod{\ell}.
\end{align}
The number of choices of $(\B_{1},\ldots,\B_{b})$ satisfying this condition is
\begin{align}
\label{eq:numofb}
c_{1} \ceq \ell^{b-1},
\end{align}
since $\B_1,\ldots,\B_{b-1}$ can be chosen freely from
$\{0,1,\ldots,\ell-1\}$, and $\B_b$ is uniquely determined by them.

For each cycle of length $b/\ell$ in $\pi$, let
$\B_{1},\ldots,\B_{b/\ell}$ be the rotation amounts of the corresponding cycles of $x^d$.
Then
\begin{align}
\B_{1}+\cdots+\B_{\f{b}{\ell}} \not\equiv 0 \pmod{\ell}.
\end{align}
The number of choices of $(\B_{1},\ldots,\B_{b/\ell})$ satisfying this condition is
\begin{align}
\label{eq:numofbl}
c_{2} \ceq \ell^{\f{b}{\ell}-1}(\ell-1),
\end{align}
since $\B_1,\ldots,\B_{b/\ell-1}$ can be chosen freely from
$\{0,1,\ldots,\ell-1\}$, and $\B_{b/\ell}$ has $\ell-1$ possible choices
satisfying the above condition.

For each $s$, the number of permutations $\pi$ is given by \eqref{eq:numofpi}.
Each such $\pi$ has $s$ cycles of length $b$, for each of which there are $c_{1}$ choices of
$\B_1,\ldots,\B_b$, and $t=q-\ell s$ cycles of length $b/\ell$, each having
$c_{2}$ choices of $\B_1,\ldots,\B_{b/\ell}$.
Multiplying these numbers and summing over all $s$, we obtain
\begin{align}
\label{eq:Abl1}
\vt{C_{\f{n}{\ell}}(b, q)} &= \left(\f{n}{\ell}\right)!\sum_{s=0}^{\fl{q/\ell}}\f{\ell^{(b-1)s}(\ell^{\f{b}{\ell}-1}(\ell-1))^{q-\ell s}}
{b^{s}s!\left(\f{b}{\ell}\right)^{q - \ell s}(q - \ell s)!}.
\end{align}
\HL

\noindent
(ii) The case $\ell \nmid b$

In this case, since $\ell\mid n=bq$, $\ell\nmid b$, and $\ell$ is prime,
we have $\ell\mid q$.
Since $\pi$ has no cycles of length $b/\ell$, it consists of exactly
$q/\ell$ cycles of length $b$.
By the same counting argument as in case~(i), we obtain
\begin{align}
\label{eq:Anbl2}
\vt{C_{\f{n}{\ell}}(b, q)}
&= \left(\f{n}{\ell}\right)!
\f{\ell^{\f{(b-1)q}{\ell}}}
{b^{\f{q}{\ell}}\left(\f{q}{\ell}\right)!}.\qedhere
\end{align}
\end{proof}
\HL
\needspace{\baselineskip} 
\begin{lemma}
\label{lem:C2q3q4q}
For $q \in \Zp$,
\begin{align}
\vt{C_{\f{n}{2}}(2, q)} &\le q!\left(\f{2e}{q}\right)^{\f{q}{2}}\exp\left(\sqrt{\f{q}{2}}\right), \\
\vt{C_{\f{n}{3}}(3, q)} &\le q!\left(\f{9e}{q}\right)^{\f{q}{3}}\exp\left(2\left(\f{q}{9}\right)^{\f{1}{3}}\right), \\
\vt{C_{\f{n}{2}}(4, q)} &\le (2q)!\left(\f{4e}{q}\right)^{\f{q}{2}}\exp\left(\f{\sqrt{q}}{2}\right).
\end{align}
\end{lemma}

\begin{proof}
For $\A, \B > 0$ and $m \in \Zp$,
\begin{align}
e^{\A x+\B x^{m}} = \sum_{j=0}^{\infty}\f{(\A x)^{j}}{j!}\sum_{k=0}^{\infty}\f{(\B x^{m})^{k}}{k!}
= \sum_{j=0}^{\infty}\sum_{k=0}^{\infty}\f{\A^{j}\B^{k}x^{j+mk}}{j!k!}.
\end{align}
Since the terms contributing to the coefficient of $x^{q}$ satisfy $j+mk=q$, we have
\begin{align}
\label{eq:xq}
[x^{q}]e^{\A x+\B x^{m}} = \sum_{\substack{j,k\ge 0\\j+mk=q}}\f{\A^{j}\B^{k}}{j!k!}
= \sum_{s=0}^{\fl{q/m}}\f{\A^{q-ms}\B^{s}}{s!(q-ms)!},
\end{align}
where $[x^{q}]F(x)$ denotes the coefficient of $x^{q}$ in $F(x)$.

Moreover, since all coefficients are nonnegative, Cauchy's estimate gives
\begin{align}
\label{eq:xqR}
[x^{q}]e^{\A x+\B x^{m}} \le \f{e^{\A R+\B R^{m}}}{R^{q}}\q\text{ for all }R > 0.
\end{align}
Hence
\begin{align}
\label{eq:ABR}
\sum_{s=0}^{\fl{q/m}}\f{\A^{q-ms}\B^{s}}{s!(q-ms)!} \le \f{e^{\A R+\B R^{m}}}{R^{q}}\q\text{ for all }R > 0.
\end{align}
\HL

\noindent
(i) $b = 2$, $\ell = 2$

Setting $\A=\B=1$ and $m=2$ in \eqref{eq:ABR}, we obtain
\begin{align}
\label{eq:ABR2}
\sum_{s=0}^{\fl{q/2}}\f{1}{s!(q-2s)!} \le \f{e^{R+R^{2}}}{R^{q}}\q\text{ for all }R > 0.
\end{align}

On the other hand, since $b=\ell=2$, \propref{prop:Cnl} gives
\begin{align}
\label{eq:Cn2}
\vt{C_{\f{n}{2}}(2, q)} &= \left(\f{n}{2}\right)!\sum_{s=0}^{\fl{q/2}}\f{2^{(2-1)s}(2^{\f{2}{2}-1}(2-1))^{q-2 s}}
{2^{s}s!\left(\f{2}{2}\right)^{q - 2 s}(q - 2 s)!}
= q!\sum_{s=0}^{\fl{q/2}}\f{1}{s!(q - 2 s)!}.
\end{align}
By \eqref{eq:ABR2} and \eqref{eq:Cn2},
\begin{align}
\vt{C_{\f{n}{2}}(2, q)} &\le q!\f{e^{R+R^{2}}}{R^{q}}\q\text{ for all }R > 0.
\end{align}
Since
\begin{align}
\dv{R}\log\left(\f{e^{R+R^{2}}}{R^{q}}\right) = \dv{R}(R+R^{2}-q\log R) = \f{1}{R}(2R^{2}+R-q),
\end{align}
the function $e^{R+R^{2}}/R^{q}$ attains its minimum at the positive root of $2R^{2}+R-q = 0$, namely
\begin{align}
R = \f{-1 + \sqrt{1+8q}}{4}.
\end{align}
For simplicity, we take the nearby value
\begin{align}
R = R_{0} \ceq \f{\sqrt{8q}}{4} = \sqrt{\f{q}{2}}\q(> 0).
\end{align}
Then
\begin{align}
\vt{C_{\f{n}{2}}(2, q)} &\le q!\f{e^{R_{0}+R_{0}^{2}}}{R_{0}^{q}}
= q!\f{\exp(\sqrt{\f{q}{2}}+\f{q}{2})}{\left(\sqrt{\f{q}{2}}\right)^{q}}
= q!\left(\f{2e}{q}\right)^{\f{q}{2}}\exp\left(\sqrt{\f{q}{2}}\right).
\end{align}
\HL

\noindent
(ii) $b = 3$, $\ell = 3$

Setting $\A = 2$, $\B = 3$, and $m = 3$ in \eqref{eq:ABR}, we obtain
\begin{align}
\label{eq:ABR3}
\sum_{s=0}^{\fl{q/3}}\f{3^{s}2^{q-3s}}{s!(q-3s)!} \le \f{e^{2 R+3 R^{3}}}{R^{q}}\q\text{ for all }R > 0.
\end{align}

On the other hand, since $b = \ell = 3$, \propref{prop:Cnl} gives
\begin{align}
\label{eq:Cn3}
\vt{C_{\f{n}{3}}(3, q)} &= \left(\f{n}{3}\right)!\sum_{s=0}^{\fl{q/3}}\f{3^{(3-1)s}(3^{\f{3}{3}-1}(3-1))^{q-3 s}}
{3^{s}s!\left(\f{3}{3}\right)^{q - 3 s}(q - 3 s)!}
= q!\sum_{s=0}^{\fl{q/3}}\f{3^{s}2^{q-3 s}}{s!(q - 3 s)!}.
\end{align}
By \eqref{eq:ABR3} and \eqref{eq:Cn3},
\begin{align}
\vt{C_{\f{n}{3}}(3, q)} &\le q!\f{e^{2R+3R^{3}}}{R^{q}}\q\text{ for all }R > 0.
\end{align}
Since
\begin{align}
\dv{R}\log\left(\f{e^{2R+3R^{3}}}{R^{q}}\right) = \dv{R}(2R+3R^{3}-q\log R) = \f{1}{R}(9R^{3}+2R-q),
\end{align}
the function $e^{2R+3R^{3}}/R^{q}$ attains its minimum at the positive root of $9R^{3}+2R-q = 0$.
For simplicity, we take the nearby value
\begin{align}
R = R_{0} \ceq \left(\f{q}{9}\right)^{\f{1}{3}}\q(> 0),
\end{align}
which is the positive root of $9R^{3}-q = 0$. Then
\begin{align}
\vt{C_{\f{n}{3}}(3, q)} &\le q!\f{e^{2R_{0}+3R_{0}^{3}}}{R_{0}^{q}} 
= q!\f{\exp(2\left(\f{q}{9}\right)^{\f{1}{3}}+\f{q}{3})}{\left(\f{q}{9}\right)^{\f{q}{3}}}
= q!\left(\f{9e}{q}\right)^{\f{q}{3}}\exp\left(2\left(\f{q}{9}\right)^{\f{1}{3}}\right).
\end{align}
\HL

\noindent
(iii) $b = 4$, $\ell = 2$

Setting $\A = 1$, $\B = 2$, and $m = 2$ in \eqref{eq:ABR}, we obtain
\begin{align}
\label{eq:ABR4}
\sum_{s=0}^{\fl{q/2}}\f{2^{s}}{s!(q-2s)!} \le \f{e^{R+2R^{2}}}{R^{q}}\q\text{ for all }R > 0.
\end{align}

On the other hand, since $b = 4$ and $\ell = 2$, \propref{prop:Cnl} gives
\begin{align}
\label{eq:Cn4}
\vt{C_{\f{n}{2}}(4, q)} &= \left(\f{n}{2}\right)!\sum_{s=0}^{\fl{q/2}}\f{2^{(4-1)s}(2^{\f{4}{2}-1}(2-1))^{q-2 s}}
{4^{s}s!\left(\f{4}{2}\right)^{q - 2 s}(q - 2 s)!}
= (2q)!\sum_{s=0}^{\fl{q/2}}\f{2^{s}}
{s!(q - 2 s)!}.
\end{align}
By \eqref{eq:ABR4} and \eqref{eq:Cn4},
\begin{align}
\vt{C_{\f{n}{2}}(4, q)} &\le (2q)!\f{e^{R+2R^{2}}}{R^{q}}\q\text{ for all }R > 0.
\end{align}
Since
\begin{align}
\dv{R}\log\left(\f{e^{R+2R^{2}}}{R^{q}}\right) = \dv{R}(R+2R^{2}-q\log R) = \f{1}{R}(4R^{2}+R-q),
\end{align}
the function $e^{R+2R^{2}}/R^{q}$ attains its minimum at the positive root of $4R^{2}+R-q = 0$, namely
\begin{align}
R = \f{-1 + \sqrt{1+16q}}{8}.
\end{align}
For simplicity, we take the nearby value
\begin{align}
R = R_{0} \ceq \f{\sqrt{16q}}{8} = \f{\sqrt{q}}{2}\q(> 0).
\end{align}
Then
\begin{align}
\vt{C_{\f{n}{2}}(4, q)} &\le (2q)!\f{e^{R_{0}+2R_{0}^{2}}}{R_{0}^{q}}
= (2q)!\f{\exp(\f{\sqrt{q}}{2}+\f{q}{2})}{\left(\f{\sqrt{q}}{2}\right)^{q}}
= (2q)!\left(\f{4e}{q}\right)^{\f{q}{2}}\exp\left(\f{\sqrt{q}}{2}\right).\qedhere
\end{align}
\end{proof}
\HL


\subsection{\texorpdfstring{Exact Computation of $\vt{N}$, $\vt{D}$, and $\vt{C}$}{Exact Computation of |N|, |D|, and |C|}}

In \secref{sec:class4}, we show the existence of a dessin with trivial automorphism group mainly by comparing
lower bounds for $\vt{N}$ with upper bounds for $\vt{D}$ or $\vt{C}$. However, for the remaining exceptional cases,
we compute the exact values of $\vt{N}$, $\vt{D}$, and $\vt{C}$.

The exact value of $\vt{N}$ can be computed from \eqref{eq:Nbqa}.

For $\vt{D}$, let $\mcP(n)$ denote the set of prime divisors of $n$.
By \propref{prop:D},
\begin{align}
D(n)=\bigcup_{\ell\in\mcP(n)}D_{\f{n}{\ell}}.
\end{align}
For each nonempty subset $S\subset\mcP(n)$, define
\begin{align}
\lambda_{S} \ceq \prod_{\ell\in S}\ell,
\end{align}
the product of all elements of $S$.
If $\ell \in S$, then $n/\la_{S} \mid n/\ell$. Hence by \propref{prop:Dk}\ref{itm:xcom2},
\begin{align}
D_{\f{n}{\la_{S}}} \subset D_{\f{n}{\ell}}\q \text{for all }\ell \in S.
\end{align}
Therefore,
\begin{align}
D_{\f{n}{\la_{S}}} \subset \bigcap_{\ell\in S}D_{\f{n}{\ell}}.
\end{align}

Conversely, since the elements of $S$ are distinct primes,
\begin{align}
\gcd\left\{\f{n}{\ell}\relmiddle{|} \ell\in S\right\} = \f{n}{\la_{S}}.
\end{align}
Hence, by B\'ezout's identity, there exist integers $a_{\ell}$ ($\ell\in S$) such that
\begin{align}
\f{n}{\la_{S}} = \sum_{\ell\in S}a_{\ell}\f{n}{\ell}.
\end{align}
If
\begin{align}
c \in \bigcap_{\ell\in S}D_{\f{n}{\ell}},
\end{align}
then $c$ commutes with $x^{n/\ell}$ for every $\ell\in S$. Therefore, $c$ also commutes with
\begin{align}
x^{\f{n}{\la_{S}}}=\prod_{\ell\in S}\left(x^{\f{n}{\ell}}\right)^{a_{\ell}}.
\end{align}
Hence $c \in D_{n/\la_{S}}$. Thus,
\begin{align}
\bigcap_{\ell\in S}D_{\f{n}{\ell}}
\subset D_{\f{n}{\la_{S}}}.
\end{align}
Consequently,
\begin{align}
\bigcap_{\ell\in S}D_{\f{n}{\ell}} =D_{\f{n}{\la_{S}}}.
\end{align}
Therefore, by the inclusion-exclusion principle and \propref{prop:Dk}, the formula for $\vt{D}$ is
\begin{align}
\vt{D(n)} &= \vt{\bigcup_{\ell\in\mcP(n)}D_{\f{n}{\ell}}}
= \sum_{\substack{S\subset\mcP(n)\\ S\ne\emptyset}}(-1)^{\vt{S}+1}\vt{\bigcap_{\ell\in S}D_{\f{n}{\ell}}} \\
&= \sum_{\substack{S\subset\mcP(n)\\ S\ne\emptyset}}(-1)^{\vt{S}+1}\vt{D_{\f{n}{\lambda_{S}}}} \\
\label{eq:calcD}
&= \sum_{\substack{S\subset\mcP(n)\\ S\ne\emptyset}}(-1)^{\vt{S}+1}\lambda_{S}^{\f{n}{\lambda_{S}}}\left(\f{n}{\lambda_{S}}\right)!.
\end{align}

For $\vt{C}$, we use the following lemma.

\begin{lemma}
\label{lem:TDnL}
Let $m\mid n$ with $m \ge 2$, and set $d = n/m$ and $\delta = \gcd(b, m)$. Then
\begin{align}
\vt{T(b,q)\cap D_{\f{n}{m}}}
= d!\sum_{(t_{k})_{k\mid \delta}}
\prod_{k\mid \delta} \f{1}{t_{k}!}
\left(\f{\varphi(k)m^{\f{b}{k}-1}}{b/k}\right)^{t_{k}},
\end{align}
where the sum is taken over all tuples $(t_{k})$ of nonnegative integers indexed by $k\mid \delta$ satisfying
\begin{align}
\sum_{k\mid \delta}\f{b}{k}t_{k}=d,
\end{align}
and $\varphi$ denotes Euler's totient function, defined by
\begin{align}
\varphi(k) \ceq \#\{ i \in \Zp \mid i \le k,\ \gcd(i, k) = 1 \}.
\end{align}
\end{lemma}

\begin{proof}
We extend the counting argument used in \propref{prop:Cnl}. As in the proof of that proposition,
\begin{align}
D_{\f{n}{m}}\cong C_{m} \wr S_{d}.
\end{align}
Consider a cycle of length $r$ in the induced permutation on the $d$ cycles
of $x^{d}$. If the total rotation along this cycle has order $k$ in $C_{m}$,
then the corresponding permutation on the $mr$ elements belonging to these $r$ cycles of $x^{d}$ consists of
 $m/k$ cycles of length $rk$.

Hence, for an element of $T(b,q)$, we must have $rk=b$. Thus
$k\mid \delta = \gcd(b,m)$ and $r=b/k$. For a fixed cycle of length $b/k$, there are
$\varphi(k)$ choices for its total rotation of order $k$, and, for each such
rotation, $m^{b/k-1}$ choices for the rotations assigned to the individual
cycles.

If $t_{k}$ denotes the number of cycles of length $b/k$ with total rotation
of order $k$, then
\begin{align}
\sum_{k\mid \delta}\f{b}{k}t_{k}=d.
\end{align}
Moreover, the resulting number of cycles of length $b$ is
\begin{align}
\sum_{k\mid \delta}\f{m}{k}t_{k} = \f{m}{b}\sum_{k\mid \delta}\f{b}{k}t_{k} = \f{md}{b} = q.
\end{align}
The number of choices for the induced permutation in $S_{d}$, together with the rotation data, is
\begin{align}
d!\prod_{k\mid \delta}\f{1}{t_{k}!}\left(\f{\varphi(k)m^{\f{b}{k}-1}}{b/k}\right)^{t_{k}}.
\end{align}
Summing over all possible $(t_{k})_{k\mid \delta}$ gives the result.
\end{proof}

For each nonempty $S\subset\mcP(n)$, set
\begin{align}
\delta_{S} \ceq \gcd(b,\la_{S}).
\end{align}
Since
\begin{align}
C(b,q)=\bigcup_{\ell\in\mcP(n)}
\left(T(b,q)\cap D_{\f{n}{\ell}}\right),
\end{align}
and
\begin{align}
\bigcap_{\ell\in S}
\left(T(b,q)\cap D_{\f{n}{\ell}}\right)
=
T(b,q)\cap D_{\f{n}{\lambda_S}},
\end{align}
the inclusion-exclusion principle and \lemref{lem:TDnL} give the formula for $\vt{C}$:
\begin{align}
\vt{C(b,q)}
&= \sum_{\substack{S\subset\mcP(n)\\ S\ne\emptyset}}
(-1)^{\vt{S}+1}
\vt{T(b,q)\cap D_{\f{n}{\lambda_{S}}}} \\
&= \sum_{\substack{S\subset\mcP(n)\\ S\ne\emptyset}}
\left((-1)^{\vt{S}+1}\left(\f{n}{\la_{S}}\right)!
\sum_{(t_{k})_{k\mid \delta_{S}}}
\prod_{k\mid \delta_{S}}\f{1}{t_{k}!}\left(\f{\varphi(k)\lambda_{S}^{\f{b}{k}-1}}{b/k}\right)^{t_{k}}\right),
\label{eq:calcC}
\end{align}
where the inner sum is taken over all tuples $(t_{k})$ of nonnegative integers indexed by $k\mid \delta_{S}$ satisfying
\begin{align}
\sum_{k\mid \delta_{S}}\f{b}{k}t_{k} = \f{n}{\la_{S}}.
\end{align}

The exact values appearing in \secref{sec:class4} were computed from
\eqref{eq:Nbqa}, \eqref{eq:calcD}, and \eqref{eq:calcC} using exact arithmetic.
\OL

\needspace{9\baselineskip} 
\section{\texorpdfstring{Passports $[a^{p}, b^{q}, n]$ of Genus at Least~2}
{Passports [a\^p, b\^q, n] of Genus at Least 2}}
\label{sec:class4}

In this section, we prove that every uniform passport $[a^{p}, b^{q}, n]$ with $n=pa=qb$,
$2 \le p < q$, and genus at least~$2$ admits a \dde with trivial automorphism group.
\HL


\subsection{Proof Strategy}
\label{sec:strategy}

As described in \secref{sec:setup}, the existence of a dessin with passport $[a^{p}, b^{q}, n]$ and trivial automorphism group can be established by showing that $\vt{N(b,q,a)}$ is greater than either $\vt{C(b,q)}$ or $\vt{D(n)}$.

Since $D(n)$ depends only on $n$, it is easier to handle in the proof. However, because it is a larger set, comparing its size with that of $N(b,q,a)$ requires a stronger estimate. Numerical computations also show that the gaps between $\vt{N}$ and $\vt{C}$ or
$\vt{D}$ tend to increase as $a$ and $b$ increase, whereas for small $a$ and $b$,
these values are close to one another, and $\vt{N}$ may even be smaller than
$\vt{C}$ or $\vt{D}$.

More specifically, when $b \ge 5$, one can show that $\vt{N} > \vt{D}$ except for finitely many cases. When $b=3$ or $4$, however, establishing the analogous inequality becomes difficult, and it is more practical to fix $b$ and prove $\vt{N} > \vt{C}$. For $b=2$, the estimates become even more restrictive. In particular, when $b=2$ and $a=3$ or $4$, even proving $\vt{N} > \vt{C}$ is difficult, and in many cases one actually has $\vt{N} < \vt{C}$.

Taking these observations into account, and guided by numerical computations, we adopt the following proof strategy.

\begin{itemize}
\item For $b=2$ and each of $a=3,4$, we prove the existence of such a dessin for every admissible $q$ by explicitly
constructing a pair $(x,y)$ whose automorphism group is trivial.
\item For $b=2$ and $a\ge5$, we show that $\vt{N}>\vt{C}$ except for finitely many cases.
\item For $b=3,4$, we show, for each fixed $b$, that $\vt{N}>\vt{C}$ except for finitely many cases.
\item For $b\ge5$, we show, for every such $b$, that $\vt{N}>\vt{D}$ except for finitely many cases.
\end{itemize}

For the finitely many exceptional cases remaining after these estimates, we compute the exact values of
$\vt{N}$ and either $\vt{D}$ or $\vt{C}$ to verify that $\vt{N}>\vt{D}$ or $\vt{N}>\vt{C}$, using \eqref{eq:Nbqa},
\eqref{eq:calcD}, and \eqref{eq:calcC}.
When these counting comparisons do not establish existence, we instead construct a pair $(x,y)$ whose
automorphism group is trivial.
\HL

For the estimates, we use the following facts.
\begin{itemize}
\item The bounds from Stirling's formula:
\begin{align}
\label{eq:Stirling}
&\sqrt{2 \pi M}\left(\f{M}{e}\right)^{M}\exp(\f{1}{12M+1}) \le M! \le \sqrt{2 \pi M}\left(\f{M}{e}\right)^{M}\exp(\f{1}{12M}) \\
&(M \in \Zp).
\end{align}
\end{itemize}

Let $c, d > 0$ and $m > 1$ be constants.
\begin{itemize}
\item The function
\begin{align}
\left(1 + \f{c}{q}\right)^{dq}
\end{align}
is strictly increasing for $q > 0$ and its limit as $q \to \infty$ is $e^{cd}$.
Hence we have
\begin{align}
\left(1 + \f{c}{m}\right)^{dm} \le \left(1 + \f{c}{q}\right)^{dq} < e^{cd}\q(q \ge m).
\end{align}
Moreover,
\begin{align}
\left(1 + \f{c}{q}\right)^{-dq}
\end{align}
is strictly decreasing for $q > 0$ and
\begin{align}
e^{-cd} < \left(1 + \f{c}{q}\right)^{-dq} \le \left(1 + \f{c}{m}\right)^{-dm}\q(q \ge m).
\end{align}
\item The functions
\begin{align}
\sqrt{q+c} - \sqrt{q}, \q \sqrt[3]{q+c} - \sqrt[3]{q}
\end{align}
are strictly decreasing for $q > 0$ and both tend to $0$ as $q \to \infty$.
Hence
\begin{align}
0 &< \sqrt{q+c} - \sqrt{q} \le \sqrt{m+c} - \sqrt{m}\q(q \ge m), \\
0 &< \sqrt[3]{q+c} - \sqrt[3]{q} \le \sqrt[3]{m+c} - \sqrt[3]{m}\q(q \ge m).
\end{align}
\item Let
\begin{align}
f(q) \ceq \f{\log q}{\log (q+c)}\q(q > 1).
\end{align}
Since
\begin{align}
f'(q) = \f{(q+c)\log(q+c)-q\log q}{q(q+c)(\log (q+c))^{2}} > 0,
\end{align}
$f(q)$ is strictly increasing and
\begin{align}
\lim_{q \to \infty}f(q) = 1.
\end{align}
Hence,
\begin{align}
\f{\log m}{\log (m+c)} \le \f{\log q}{\log (q+c)} < 1\q(q \ge m).
\end{align}
Moreover,
\begin{align}
\f{\log (q+c)}{\log q}
\end{align}
is strictly decreasing and
\begin{align}
1 < \f{\log (q+c)}{\log q} \le \f{\log (m+c)}{\log m}\q(q \ge m).
\end{align}
\end{itemize}
\HL


\subsection{\texorpdfstring{The Subcase $b=2$}{The Subcase b=2}}

\subsubsection{\texorpdfstring{The Subcase $b=2$, $a=3$}{The Subcase b=2, a=3}}

The passport $[3^{p}, 2^{q}, n]$ with $n = 3p=2q$ can be written as $[3^{2t}, 2^{3t}, 6t]$.
For each such passport of genus at least~$2$, we construct permutations $x$ and $y$ such that the corresponding dessin
has trivial automorphism group.

By \eqref{eq:genus}, the genus is
\begin{align}
g = \f{6t-(2t+3t+1)}{2} + 1 = \f{t+1}{2},
\end{align}
where $g \in \Zz$ and $g \ge 2$. Hence $t$ is odd and $t \ge 3$.

Since the automorphism group is invariant under permutations of the three partitions,
it suffices to consider the passport
$[6t,2^{3t},3^{2t}]$,
that is, the case in which $x$, $y$, and
$z=(xy)^{-1}$ have cycle types
$(6t)$, $(2^{3t})$, and $(3^{2t})$, respectively.
Since the automorphism group is also invariant under relabeling of the edges,
we may fix $x$.

\begin{prop}
\label{prop:c4b2a3}
For each odd integer $t \ge 3$, fix $x = (1\ 2\ \ldots\ 6t)$ and define $y \in S_{6t}$ as follows:

For $t=3$, let
\begin{align}
y &= (1\ 4)(2\ 8)(3\ 9)(5\ 13)(6\ 16)(7\ 10)(11\ 15)(12\ 17)(14\ 18).
\end{align}

For $t \ge 5$, let
\begin{align}
\begin{aligned}
\label{eq:ytdef1}
y &= \A \B_{0}\B_{1}\cdots\B_{\f{t-5}{2}} \G, \\
\A &= (1\ 4)(2\ 6)(3\ 7)(5\ 8)(9\ 6t)(10\ 13), \\
\B_{i} &= (10i{+}11\ 10i{+}16)(10i{+}12\ 10i{+}17)(10i{+}14\ 6t{-}2i{-}1) \\
&\sq(10i{+}15\ 10i{+}18)(10i{+}19\ 6t{-}2i-2)(10i{+}20\ 10i{+}23), \\
\G &= (10w{+}11\ 10w{+}15)(10w{+}12\ 10w{+}16)(10w{+}14\ 10w+17), \\
& \sq \text{where }w = \f{t-3}{2}.
\end{aligned}
\end{align}
Then the following hold:
\begin{enumerate}
\item\label{itm:b2a3-1} $y$ has cycle type $(2^{3t})$.
\item\label{itm:b2a3-2} $(xy)^{-1}$ has cycle type $(3^{2t})$.
\item\label{itm:b2a3-3} For the \dde $\msD$ corresponding to $x$ and $y$, we have $\AD \cong \{ 1 \}$.
\end{enumerate}
\end{prop}

\begin{proof}
When $t=3$, $y$ has cycle type $(2^{9})$, and
\begin{align}
(xy)^{-1} = (1\ 14\ 5)(2\ 4\ 9)(3\ 8\ 10)(6\ 13\ 17)(7\ 16\ 11)(12\ 15\ 18)
\end{align}
has cycle type $(3^{6})$. Moreover, a direct calculation shows that $x^{k}y \ne yx^{k}$ for all $1 \le k \le 17$.
Hence $\AD \cong \{1\}$.

In what follows, assume that $t \ge 5$.
\HL

\noindent
\ref{itm:b2a3-1}
In \eqref{eq:ytdef1}, every cycle in $y$ has length~$2$, and the number of these cycles is
\begin{align}
6 + 6\left(\f{t-5}{2}+1\right) + 3 = 3t.
\end{align}
Therefore, to prove that $y$ has cycle type $(2^{3t})$, it suffices to show that every element of $\{ 1, \ldots, 6t \}$ appears
in \eqref{eq:ytdef1}.

Since $w = (t-3)/2$, we have
\begin{align}
6t = 12w+18.
\end{align}
Moreover, the index $i$ of $\B_{i}$ ranges over
\begin{align}
\label{eq:irange1}
0 \le i \le \f{t-5}{2} = w-1.
\end{align}
We verify that every element $s \in \{ 1, \ldots, 6t \} = \{ 1, \ldots, 12w+18 \}$ appears in \eqref{eq:ytdef1}.
\begin{itemize}
\item If $s \le 10$, then $s$ appears in $\A$.
\item If $11 \le s \le 10w+10$, then write $s = 10u+v$ with $1 \le v \le 10$. Then $1 \le u \le w$, and $v$ ranges
from $1$ to $10$ for each $u$. In this case, $s$ appears in
\begin{align}
\begin{dcases}
\A & (u = 1,\ v = 3), \\
\B_{u-2} & (u \ge 2,\ v=3), \\
\B_{u-1} & (v  \ne 3).
\end{dcases}
\end{align}
\item The elements $10w+11$, $10w+12$, $10w+14$, $10w+15$, $10w+16$, and $10w+17$ appear in $\G$.
\item The element $10w+13$ appears in $\B_{\f{t-5}{2}} = \B_{w-1}$.
\item If $10w+18 \le s \le 12w+17$ and $s$ is even, then $s$ appears in $\B_{6w+8-s/2}$ as the entry $6t-2i-2$.
\item If $10w+18 \le s \le 12w+17$ and $s$ is odd, then $s$ appears in $\B_{6w+8-(s-1)/2}$ as the entry $6t-2i-1$.
\item $6t = 12w+18$ appears in $\A$.
\end{itemize}
Therefore, every element of $\{ 1, \ldots, 6t \}$ appears in \eqref{eq:ytdef1}, and hence $y$ has cycle type $(2^{3t})$.
\HL

\noindent
\ref{itm:b2a3-2}
Since $xy$ and $(xy)^{-1}$ have the same cycle type, it suffices to show that $xy$ has cycle type $(3^{2t})$.

\tabref{tab:xycycles1} lists the $2t$ cycles of length~$3$ in $xy$ arising from \eqref{eq:ytdef1}.

\begin{table}[htbp]
\centering
\small
\begin{tabular}{|c|c|c|}
\hline
$s$ & $y\cdot s$ & $xy\cdot s$ \\
\hline
\multicolumn{3}{|c|}{$1 \le s \le 9$} \\
\hline
1 & 4 & 5 \\
5 & 8 & 9 \\
9 & $6t$ & 1 \\
\hline
2 & 6 & 7 \\
7 & 3 & 4 \\
4 & 1 & 2 \\
\hline
3 & 7 & 8 \\
8 & 5 & 6 \\
6 & 2 & 3 \\
\hline
\multicolumn{3}{|c|}{$10 \le s \le 10w+9$, $10w+19 \le s \le 12w+18$} \\
\hline
$10i+10$ & $10i+13$ & $10i+14$ \\
$10i+14$ & $6t-2i-1$ & $6t-2i$ \\
$6t-2i$ & $10i + 9$ & $10i+10$ \\
\hline
$10i+11$ & $10i+16$ & $10i+17$ \\
$10i+17$ & $10i+12$ & $10i+13$ \\
$10i+13$ & $10i+10$ & $10i+11$ \\
\hline
$10i+12$ & $10i+17$ & $10i+18$ \\
$10i+18$ & $10i+15$ & $10i+16$ \\
$10i+16$ & $10i+11$ & $10i+12$ \\
\hline
$10i+15$ & $10i+18$ & $10i+19$ \\
$10i+19$ & $6t-2i-2$ & $6t-2i-1$ \\
$6t-2i-1$ & $10i+14$ & $10i+15$ \\
\hline
\multicolumn{3}{|c|}{$10w+10 \le s \le 10w+18$} \\
\hline
$10(w-1)+20$ & $10(w-1)+23$ & $10w+14$ \\
$10w+14$ & $10w+17$ & $10w{+}18  = 6t{-}2(w{-}1){-}2$ \\
$6t{-}2(w{-}1){-}2$ & $10(w-1)+19$ & $10(w-1)+20$ \\
\hline
$10w+11$ & $10w+15$ & $10w+16$ \\
$10w+16$ & $10w+12$ & $10(w-1)+23$ \\
$10(w-1)+23$ & $10(w-1)+20$ & $10w+11$ \\
\hline
$10w+12$ & $10w+16$ & $10w+17$ \\
$10w+17$ & $10w+14$ & $10w+15$ \\
$10w+15$ & $10w+11$ & $10w+12$ \\
\hline
\end{tabular}
\HL
\caption{Cycles of $xy$ ($s = 1, \ldots, 6t$)}\label{tab:xycycles1}
\end{table}

We now verify that every element $s \in \{ 1, \ldots, 6t \}$ appears in one of the cycles of $xy$. Note that $6t = 12w+18$.
\begin{itemize}
\item If $1 \le s \le 9$, then $s$ appears in the first three cycles.
\item If $10 \le s \le 10w+9$, then write $s = 10i+u$ with $10 \le u \le 19$. Then $0 \le i \le w-1$, and $s$ appears as $10i+u$.
\item If $10w+10 \le s \le 10w+18$, then $s$ appears in the last three cycles.
\item If $10w+19 \le s \le 12w+18$, then $6t-2(w-1)-1 \le s \le 6t$. If $s$ is even, $s$ appears as $6t-2i$ with $i = 3t-s/2$.
If $s$ is odd, $s$ appears as $6t-2i-1$ with $i = 3t-(s+1)/2$.
\end{itemize}

Since the $2t$ cycles listed in \tabref{tab:xycycles1} contain $6t$ entries in total, every element appears exactly once.
Hence $xy$ has cycle type $(3^{2t})$.
\HL

\noindent
\ref{itm:b2a3-3}
By \corref{cor:ADtrivial}\ref{itm:ADtrivial2}, it suffices to show that
\begin{align}
\text{for every}\ 1 \le k \le 6t-1, \ x^{k}y \ne yx^{k}.
\end{align}
Equivalently,
\begin{align}
&\text{for every}\ 1 \le k \le 6t-1,\\*
\label{eq:6txkytrivial}
&\sq\text{there exists}\ e \in E = \{ 1, \ldots, 6t \}\ \text{such that}\ x^{k}y\cdot e \ne yx^{k}\cdot e.
\end{align}

In what follows, the elements of $E$ are taken modulo $6t$.
\HL

\noindent
(i) Let $e = 6t-1$. Then
\begin{align}
x^{k}y\cdot e &= x^{k} \cdot 14 = k+14, \\
yx^{k}\cdot e &= y\cdot(k-1).
\end{align}
The equality $x^{k}y\cdot e = yx^{k}\cdot e$ holds only if
$y$ maps $k-1$ to $k+14$, thus increasing it by $15$.

From \eqref{eq:ytdef1}, the possible values of $k-1$ are
\begin{align}
9,\q 10i+14,\q 6t-2i-1,\q 10i+19,\q 6t-2i-2.
\end{align}
We examine these cases separately.

\begin{itemize}
\item If $k-1=9$, then
\begin{align}
y\cdot 9=6t\equiv 9+15 = 24 \pmod{6t},
\end{align}
hence
\begin{align}
0 \equiv 24 \pmod{6t}.
\end{align}
Since $t \ge 5$, this is impossible.
\item If $k-1=10i+14$, then
\begin{align}
y\cdot (10i+14)=6t-2i-1 \equiv 10i+29 \pmod{6t}.
\end{align}
Since
\begin{align}
10i+29 \le 10\cdot\f{t-5}{2}+29 = 5t+4 < 6t,
\end{align}
we have
\begin{align}
6t-2i-1 = 10i+29,
\end{align}
hence $i = (t-5)/2 = w-1$. Therefore,
\begin{align}
k = 10(w-1)+14+1 = 10w+5.
\end{align}
\item If $k-1=6t-2i-1$, then $i=0$ would give $k=6t$, which is excluded.
For $i\ge1$, the required congruence gives
\begin{align}
10i+14\equiv6t-2i-1+15\pmod{6t},
\end{align}
hence
\begin{align}
12i\equiv0\pmod{6t}.
\end{align}
However,
\begin{align}
0<12i\le12\cdot \f{t-5}{2}<6t,
\end{align}
a contradiction.
\item If $k-1=10i+19$, then
\begin{align}
y\cdot (10i+19)=6t-2i-2 \equiv 10i+19+15 \pmod{6t},
\end{align}
hence
\begin{align}
12(i+3) \equiv 0 \pmod{6t}.
\end{align}
Thus $t\mid2(i+3)$. Since $t$ is odd, $t\mid i+3$. However,
\begin{align}
3 \le i+3 \le \f{t-5}{2} + 3 = t - \f{t-1}{2}< t,
\end{align}
a contradiction.
\item If $k-1=6t-2i-2$, then
\begin{align}
y\cdot (6t-2i-2)=10i+19 \equiv 6t-2i-2+15 \pmod{6t}, 
\end{align}
hence
\begin{align}
12i+6\equiv0\pmod{6t}.
\end{align}
However,
\begin{align}
0 < 12i+6 \le 12\cdot\f{t-5}{2}+6 = 6t-24 < 6t,
\end{align}
a contradiction.
\end{itemize}
Thus, the equality $x^{k}y\cdot e=yx^{k}\cdot e$ can hold only if
\begin{align}
\label{eq:k10wp5}
k = 10w+5.
\end{align}
\HL

\noindent
(ii) Let $e=6t$. By \eqref{eq:k10wp5}, it remains to consider
$k=10w+5=10(w-1)+15$.
Since $w-1=(t-5)/2$, \eqref{eq:ytdef1} gives
\begin{align}
yx^{k}\cdot e = y\cdot k =y\cdot(10(w-1)+15) =10(w-1)+18 = 10w+8.
\end{align}
On the other hand,
\begin{align}
x^{k}y\cdot e=x^{k}\cdot9=k+9=10w+14.
\end{align}
Hence $x^{k}y\cdot e\ne yx^{k}\cdot e$.
Therefore, the equality
$x^{k}y\cdot e=yx^{k}\cdot e$ cannot hold for both $e=6t-1$ and $e=6t$.

Thus, \eqref{eq:6txkytrivial} holds. Consequently, for the corresponding \dde\ $\msD$, we have
$\AD \cong \{ 1 \}$.
\end{proof}

\begin{rem}
When $t=3$, applying the same algorithm as in the case $t\ge5$, we obtain
\begin{align}
y = (1\ 4)(2\ 6)(3\ 7)(5\ 8)(9\ 18)(10\ 13)(11\ 15)(12\ 16)(14\ 17).
\end{align}
This satisfies \ref{itm:b2a3-1} and \ref{itm:b2a3-2}.
However, it does not satisfy \ref{itm:b2a3-3}; in fact, $\AD\cong C_{2}$.
\end{rem}

\begin{prop}
\label{prop:class4b2a3}
If a uniform passport $[n, 2^{q}, 3^{p}]$ with $n = 2q = 3p$ has genus at least~$2$, then it admits a \dde with trivial automorphism group.
\end{prop}

\begin{proof}
Since $n = 2q = 3p$, we may write $n = 6t$, $p = 2t$, $q = 3t$ for $t \in \Zp$. By \eqref{eq:genus}, the genus is
\begin{align}
g &= \f{6t-(2t+3t+1)}{2}+1 = \f{t+1}{2} \ge 2.
\end{align}
Since $g$ is an integer, $t$ is odd, and since $g\ge2$, we have $t\ge3$.

Since the automorphism group is invariant under relabeling of the edges, we may fix an $n$-cycle $x$.
Then \propref{prop:c4b2a3} provides a permutation $y$ such that the pair
$(x,y)$ has monodromy group $G=\gen{x,y}$. The corresponding
dessin $\msD$ satisfies $\AD \cong \{ 1 \}$.
\end{proof}
\HL


\subsubsection{\texorpdfstring{The Subcase $b=2$, $a=4$}{The Subcase b=2, a=4}}

The passport $[4^{p}, 2^{q}, n]$ with $n = 4p = 2q$ can be written as $[4^{t}, 2^{2t}, 4t]$.
For each such passport of genus at least~$2$, we construct permutations $x$ and $y$ such that the corresponding dessin
has trivial automorphism group.

By \eqref{eq:genus}, the genus is
\begin{align}
g = \f{4t-(t+2t+1)}{2}+1 = \f{t+1}{2},
\end{align}
where $g \in \Zz$ and $g \ge 2$. Hence $t$ is odd and $t \ge 3$.

Since the automorphism group is invariant under permutations of the three partitions,
it suffices to consider the passport
$[4t,2^{2t},4^{t}]$,
that is, the case in which $x$, $y$, and
$z=(xy)^{-1}$ have cycle types
$(4t)$, $(2^{2t})$, and $(4^{t})$, respectively.
Since the automorphism group is also invariant under relabeling of the edges,
we may fix $x$.

\begin{prop}
\label{prop:c4b2a4}
For each odd integer $t \ge 3$, fix $x = (1\ 2\ \ldots\ 4t)$ and define $y \in S_{4t}$ as follows:

For $t=3$, let
\begin{align}
y = (1\ 3)(2\ 6)(4\ 10)(5\ 7)(8\ 11)(9\ 12).
\end{align}

For $t \ge 5$, let
\begin{align}
\begin{aligned}
\label{eq:ytdef2}
y &= \A \B_{0}\B_{1}\cdots\B_{\f{t-5}{2}} \G, \\
\A &= (1\ 3)(2\ 5)(4\ 6)(7\ 9), \\
\B_{i} &= (7i{+}8\ 7i{+}12)(7i{+}10\ 4t{-}i)(7i{+}11\ 7i{+}13)(7i{+}14\ 7i{+}16), \\
\G &= (7w{+}8\ 7w{+}11)(7w{+}10\ 7w+12), \\
&\sq \text{where }w = \f{t-3}{2}.
\end{aligned}
\end{align}
Then the following hold:
\begin{enumerate}
\item\label{itm:b2a4-1} $y$ has cycle type $(2^{2t})$.
\item\label{itm:b2a4-2} $(xy)^{-1}$ has cycle type $(4^{t})$.
\item\label{itm:b2a4-3} For the \dde $\msD$ corresponding to $x$ and $y$, we have $\AD \cong \{ 1 \}$.
\end{enumerate}
\end{prop}

\begin{proof}
When $t=3$, $y$ has cycle type $(2^{6})$, and
\begin{align}
(xy)^{-1} = (1\ 9\ 11\ 4)(2\ 3\ 6\ 7)(5\ 10\ 12\ 8)
\end{align}
has cycle type $(4^{3})$. Moreover, a direct calculation shows that $x^{k}y \ne yx^{k}$ for all $1 \le k \le 11$.
Hence $\AD \cong \{1\}$.

In what follows, assume that $t \ge 5$.
\HL

\noindent
\ref{itm:b2a4-1}
In \eqref{eq:ytdef2}, every cycle in $y$ has length~$2$, and the number of these cycles is
\begin{align}
4 + 4\left(\f{t-5}{2}+1\right) + 2 = 2t.
\end{align}
Therefore, to prove that $y$ has cycle type $(2^{2t})$, it suffices to show that every element of $\{ 1, \ldots, 4t \}$ appears
in \eqref{eq:ytdef2}.

Since $w = (t-3)/2$, we have
\begin{align}
4t = 8w+12.
\end{align}
Moreover, the index $i$ of $\B_{i}$ ranges over
\begin{align}
0 \le i \le \f{t-5}{2} = w-1.
\end{align}
We verify that every element $s \in \{ 1, \ldots, 4t \} = \{ 1, \ldots, 8w+12 \}$ appears in \eqref{eq:ytdef2}.
\begin{itemize}
\item If $s \in \{1,2,3,4,5,6,7,9\}$, then $s$ appears in $\A$.
\item If $s \in \{8,10,11,12,13,14\}$, then $s$ appears in $\B_{0}$.
\item If $15 \le s \le 7w+7$, then write $s = 7u+v$ with $1 \le v \le 7$. Then $2 \le u \le w$, and $v$
ranges from $1$ to $7$ for each $u$. In this case, $s$ appears in
\begin{align}
\begin{dcases}
\B_{u-2} & (v = 2), \\
\B_{u-1} & (v \ne 2).
\end{dcases}
\end{align}
\item The elements $7w+8$, $7w+10$, $7w+11$, and $7w+12$ appear in $\G$.
\item The element $7w+9$ appears in $\B_{\f{t-5}{2}} = \B_{w-1}$.
\item If $7w+13 \le s \le 8w+12$, then $s$ appears in $\B_{8w+12-s}$ as the entry $4t - i$.
\end{itemize}
Therefore, every element of $\{ 1, \ldots, 4t \}$ appears in \eqref{eq:ytdef2}, and hence $y$ has cycle type $(2^{2t})$.
\HL

\noindent
\ref{itm:b2a4-2}
Since $xy$ and $(xy)^{-1}$ have the same cycle type, it suffices to show that $xy$ has cycle type $(4^{t})$.

\tabref{tab:xycycles2} lists the $t$ cycles of length~$4$ in $xy$ arising from \eqref{eq:ytdef2}.

\begin{table}[htbp]
\centering
\small
\begin{tabular}{|c|c|c|}
\hline
$s$ & $y\cdot s$ & $xy\cdot s$ \\
\hline
\multicolumn{3}{|c|}{$1 \le s \le 7$, $s = 10$} \\
\hline
1 & 3 & 4 \\
4 & 6 & 7 \\
7 & 9 & 10 \\
10 & $4t$ & 1 \\
\hline
2 & 5 & 6 \\
6 & 4 & 5 \\
5 & 2 & 3 \\
3 & 1 & 2 \\
\hline
\multicolumn{3}{|c|}{$8 \le s \le 7w{+}7$ ($s \ne 10$), $s = 7w{+}10$, $7w{+}13 \le s \le 8w{+}12$} \\
\hline
$7i+8$ & $7i+12$ & $7i+13$ \\
$7i+13$ & $7i+11$ & $7i+12$ \\
$7i+12$ & $7i+8$ & $7i+9$ \\
$7i+9$ & $7i+7$ & $7i+8$ \\
\hline
$7(i+1)+10$ & $4t-(i+1)$ & $4t-i$ \\
$4t-i$ & $7i+10$ & $7i+11$ \\
$7i+11$ & $7i+13$ & $7i+14$ \\
$7i+14$ & $7i+16$ & $7(i+1)+10$ \\
\hline
\multicolumn{3}{|c|}{$7w+8 \le s \le 7w+12$ ($s \ne 7w+10$)} \\
\hline
$7w+8$ & $7w+11$ & $7w+12$ \\
$7w+12$ & $7w+10$ & $7w+11$ \\
$7w+11$ & $7w+8$ & $7(w-1)+16$ \\
$7(w-1)+16$ & $7(w-1)+14$ & $7w+8$ \\
\hline
\end{tabular}
\HL
\caption{Cycles of $xy$ ($s = 1, \ldots, 4t$)}\label{tab:xycycles2}
\end{table}

We now verify that every element $s \in \{ 1, \ldots, 4t \}$ appears in one of the cycles of $xy$. Note that $4t = 8w+12$.
\begin{itemize}
\item If $1 \le s \le 7$, then $s$ appears in the first two cycles.
\item If $8 \le s \le 7w+7$, then write $s = 7i+v$ with $8 \le v \le 14$. Then $0 \le i \le w-1$. The element $10$ appears in the first cycle.
If $v=10$ and $i\ge1$, set $j=i-1$. Then $s=7(j+1)+10$, and $s$ appears in the row indexed by $j$.
Otherwise $s$ appears as $7i+v$.
\item The elements $7w+8$, $7w+9$, $7w+11$, and $7w+12$ appear in the last cycle.
\item The element $7w+10$ appears as $7(i+1)+10$ with $i=w-1$.
\item If $7w+13 \le s \le 8w+12$, then $4t-(w-1) \le s \le 4t$. The element $s$ appears as $4t-i$ with $i = 4t-s$.
\end{itemize}

Since the $t$ cycles listed in \tabref{tab:xycycles2} contain $4t$ entries in total, every element appears exactly once.
Hence $xy$ has cycle type $(4^{t})$.
\HL

\needspace{2\baselineskip} 
\noindent
\ref{itm:b2a4-3}
By \corref{cor:ADtrivial}\ref{itm:ADtrivial2}, it suffices to show that
\begin{align}
\text{for every}\ 1 \le k \le 4t-1, \ x^{k}y \ne yx^{k}.
\end{align}
Equivalently,
\begin{align}
&\text{for every}\ 1 \le k \le 4t-1,\\*
\label{eq:4txkytrivial}
&\sq\text{there exists}\ e \in E = \{ 1, \ldots, 4t \}\ \text{such that}\ x^{k}y\cdot e \ne yx^{k}\cdot e.
\end{align}

In what follows, the elements of $E$ are taken modulo $4t$.
\HL

\noindent
(i) Let $e = 4t$. Then
\begin{align}
x^{k}y\cdot e &= x^{k} \cdot 10 = k+10, \\
yx^{k}\cdot e &= y\cdot k.
\end{align}
The equality $x^{k}y\cdot e = yx^{k}\cdot e$ holds only if
$y$ maps $k$ to $k+10$, thus increasing it by $10$.

From \eqref{eq:ytdef2}, the possible values of $k$ are
\begin{align}
7i+10,\q 4t-i.
\end{align}
We examine these cases separately.

\begin{itemize}
\item If $k=7i+10$, then
\begin{align}
y\cdot (7i+10)=4t-i \equiv 7i+10+10 \pmod{4t}.
\end{align}
Hence
\begin{align}
8i \equiv -20 \pmod{4t}.
\end{align}
Since $0 \le 8i \le 8\cdot(t-5)/2 = 4t-20$,
\begin{align}
8i = 4t-20,
\end{align}
therefore
\begin{align}
i = \f{t-5}{2} = w-1.
\end{align}
Thus,
\begin{align}
k = 7(w-1)+10 = 7w+3.
\end{align}
\item If $k=4t-i$, then $i=0$ would give $k=4t$, which is excluded.
For $i\ge1$, we have
\begin{align}
y\cdot (4t-i)=7i+10 \equiv 4t-i+10 \pmod{4t}.
\end{align}
Hence
\begin{align}
8i \equiv 0 \pmod{4t}.
\end{align}
However,
\begin{align}
0 < 8i \le 8\cdot \f{t-5}{2} = 4t-20 < 4t,
\end{align}
a contradiction.
\end{itemize}
Thus, the equality $x^{k}y\cdot e=yx^{k}\cdot e$ can hold only if
\begin{align}
\label{eq:k4t}
k = 7w+3.
\end{align}
\HL

\noindent
(ii) Let $e=2$. By \eqref{eq:k4t}, it remains to consider
$k=7w+3=7(w-1)+10$.
Since $w-1=(t-5)/2$, \eqref{eq:ytdef2} gives
\begin{align}
yx^{k}\cdot e = y\cdot (k+2) =y\cdot(7(w-1)+12) =7(w-1)+8 = 7w+1.
\end{align}
On the other hand,
\begin{align}
x^{k}y\cdot e=x^{k} \cdot 5=k+5=7w+8.
\end{align}
Hence $x^{k}y\cdot e\ne yx^{k}\cdot e$.
Therefore, the equality
$x^{k}y\cdot e=yx^{k}\cdot e$ cannot hold for both $e=4t$ and $e=2$.

Thus, \eqref{eq:4txkytrivial} holds. Consequently, for the corresponding \dde\ $\msD$, we have
$\AD \cong \{ 1 \}$.
\end{proof}

\begin{rem}
When $t=3$, applying the same algorithm as in the case $t\ge5$, we obtain
\begin{align}
y = (1\ 3)(2\ 5)(4\ 6)(7\ 9)(8\ 11)(10\ 12).
\end{align}
This satisfies \ref{itm:b2a4-1} and \ref{itm:b2a4-2}.
However, it does not satisfy \ref{itm:b2a4-3}; in fact, $\AD\cong C_{2}$.
\end{rem}

\begin{prop}
\label{prop:class4b2a4}
If a uniform passport $[n, 2^{q}, 4^{p}]$ with $n = 2q = 4p$ has genus at least~$2$, then it admits a \dde with trivial automorphism group.
\end{prop}

\begin{proof}
Since $n = 2q = 4p$, we may write $n = 4t$, $p = t$, $q = 2t$ for $t \in \Zp$. By \eqref{eq:genus}, the genus is
\begin{align}
g &= \f{4t-(t+2t+1)}{2}+1 = \f{t+1}{2} \ge 2.
\end{align}
Since $g$ is an integer, $t$ is odd, and since $g\ge2$, we have $t\ge3$.

Since the automorphism group is invariant under relabeling of the edges, we may fix an $n$-cycle $x$.
Then \propref{prop:c4b2a4} provides a permutation $y$ such that the pair
$(x,y)$ has monodromy group $G=\gen{x,y}$. The corresponding
dessin $\msD$ satisfies $\AD \cong \{ 1 \}$.
\end{proof}
\HL


\subsubsection{\texorpdfstring{The Subcase $b=2$, $a\ge 5$}{The Subcase b=2, a >= 5}}
\label{sec:class4b2age5}

\begin{prop}
\label{prop:c4b2age5}
If a uniform passport $[n, 2^{q}, a^{p}]$ with $n = pa = 2q$, $2 \le p < q$, and $a \ge 5$ has genus
at least~$2$, then it admits a \dde with trivial automorphism group.
\end{prop}

\begin{proof}
We prove the statement by showing that
$\vt{N(2,q,a)}>\vt{C(2,q)}$ except for a few exceptional pairs $(a,p)$.
The exceptional pairs are handled separately, either by computing the exact values of
$\vt{N}$ and $\vt{C}$ or by exhibiting a pair $(x,y)$ with trivial automorphism group.

Since $n=2q=pa\ge2\cdot5=10$, we have $q\ge5$. By \eqref{eq:genus-uc}, the genus is
\begin{align}
g &= \f{2q-(p+q)+1}{2} = \f{q-p+1}{2}.
\end{align}

Since $2 \mid n = 2q$, by \eqref{eq:Cbqsum} and \lemref{lem:C2q3q4q},
\begin{align}
\vt{C(2, q)} &\le \sum_{\ell\colon \text{prime},\, \ell\, \mid\, n} \vt{C_{\f{n}{\ell}}(2, q)}
=\ \vt{C_{\f{n}{2}}(2, q)} + \sum_{\substack{\ell\colon \text{odd prime}\\ \ell\, \mid\, n}} \vt{C_{\f{n}{\ell}}(2, q)} \\
&\le q!\left(\f{2e}{q}\right)^{\f{q}{2}}\exp\left(\sqrt{\f{q}{2}}\right) + \sum_{\substack{\ell\colon \text{odd prime}\\ \ell\, \mid\, n}} \vt{C_{\f{n}{\ell}}(2, q)}.
\end{align}
For an odd prime $\ell\mid n=2q$, we have $\ell\mid q$.
Let $m=q/\ell \in \Zp$. Since $\ell\nmid2$, \propref{prop:Cnl} gives
\begin{align}
\vt{C_{\f{n}{\ell}}(2, q)} &= \left(\f{2q}{\ell}\right)!\f{\ell^{\f{q}{\ell}}}{2^{\f{q}{\ell}}\left(\f{q}{\ell}\right)!}
= \f{\ell^{m}(2m)!}{2^{m}m!}.
\end{align}
By the bounds from Stirling's formula \eqref{eq:Stirling}, we obtain
\begin{align}
\vt{C_{\f{n}{\ell}}(2, q)} &\le \f{\ell^{m}\sqrt{4 \pi m}\left(\f{2m}{e}\right)^{2m}}{2^{m}\sqrt{2 \pi m}\left(\f{m}{e}\right)^{m}}
\exp(\f{1}{24m}-\f{1}{12m+1}) \\
&= \sqrt{2}\left(\f{2\ell m}{e}\right)^{m}\exp(\f{1}{24m}-\f{1}{12m+1}) \le \sqrt{2}\left(\f{2\ell m}{e}\right)^{m}
= \sqrt{2}\left(\f{2q}{e}\right)^{\f{q}{\ell}}.
\end{align}
Since $2q/e\ge10/e>1$, the right-hand side is decreasing as $\ell$ increases.

Let $\ell_{1}, \ldots, \ell_{s}$ be the distinct odd prime divisors of $n = 2q$. Since each $\ell_{i}$ divides $q$,
\begin{align}
\ell_{1}\cdots\ell_{s} \le q.
\end{align}
Since $\ell_{i} \ge 3$ for each $i$,
\begin{align}
\ell_{1}\cdots\ell_{s} \ge 3^{s}.
\end{align}
Hence $3^{s} \le q$ and
\begin{align}
s \le \log_{3}q = \f{\log q}{\log 3}.
\end{align}
Therefore,
\begin{align}
\sum_{\substack{\ell\colon \text{odd prime}\\ \ell\, \mid\, n}} \vt{C_{\f{n}{\ell}}(2, q)}
&\le \sum_{\substack{\ell\colon \text{odd prime}\\ \ell\, \mid\, n}} \sqrt{2}\left(\f{2q}{e}\right)^{\f{q}{\ell}}
\le \sum_{\substack{\ell\colon \text{odd prime}\\ \ell\, \mid\, n}} \sqrt{2}\left(\f{2q}{e}\right)^{\f{q}{3}}
\le  \f{\sqrt{2}\log q}{\log 3}\left(\f{2q}{e}\right)^{\f{q}{3}}.
\end{align}
\needspace{\baselineskip} 
Thus,
\begin{align}
\vt{C(2, q)} &\le V_{1}(q) + V_{2}(q), \\
V_{1}(q) &\ceq q!\left(\f{2e}{q}\right)^{\f{q}{2}}\exp\left(\sqrt{\f{q}{2}}\right), \\
V_{2}(q) &\ceq \f{\sqrt{2}\log q}{\log 3}\left(\f{2q}{e}\right)^{\f{q}{3}}.
\end{align}

Let
\begin{align}
R_{1}(q) \ceq \f{V_{1}(q)}{V_{2}(q)}.
\end{align}
Then, using the monotonicity facts stated in \secref{sec:strategy},
\begin{align}
\f{R_{1}(q+1)}{R_{1}(q)} &= \f{V_{1}(q+1)}{V_{1}(q)}\cdot\f{V_{2}(q)}{V_{2}(q+1)} \\
&= (q+1)!\left(\f{2e}{q+1}\right)^{\f{q+1}{2}}\exp\left(\sqrt{\f{q+1}{2}}\right)
\f{1}{q!}\left(\f{2e}{q}\right)^{-\f{q}{2}}\exp\left(-\sqrt{\f{q}{2}}\right) \\
&\sq\cdot\f{\sqrt{2}\log q}{\log 3}\left(\f{2q}{e}\right)^{\f{q}{3}}
\f{\log 3}{\sqrt{2}\log (q+1)}\left(\f{2q+2}{e}\right)^{-\f{q+1}{3}} \\
&= (q+1)\left(1+\f{1}{q}\right)^{-\f{q}{2}}\left(\f{2e}{q+1}\right)^{\f{1}{2}}\f{\log q}{\log(q+1)}\left(1+\f{1}{q}\right)^{-\f{q}{3}}\left(\f{e}{2q+2}\right)^{\f{1}{3}} \\
&\sq\cdot\exp(\f{\sqrt{q+1}-\sqrt{q}}{\sqrt{2}}) \\
&= 2^{\f{1}{6}}e^{\f{5}{6}}(q+1)^{\f{1}{6}}\left(1+\f{1}{q}\right)^{-\f{5q}{6}}\f{\log q}{\log(q+1)}\exp(\f{\sqrt{q+1}-\sqrt{q}}{\sqrt{2}}) \\
&\ge 2^{\f{1}{6}}e^{\f{5}{6}}(5+1)^{\f{1}{6}}e^{-\f{5}{6}}\f{\log 5}{\log(5+1)}\cdot 1 = 1.359\ldots > 1.
\end{align}
Therefore, $R_1(q)$ is strictly increasing for integers $q\ge5$. Hence
\begin{align}
R_{1}(q) &\ge R_{1}(5) = 5!\left(\f{2e}{5}\right)^{\f{5}{2}}\exp\left(\sqrt{\f{5}{2}}\right)
\f{\log 3}{\sqrt{2}\log 5}\left(\f{2\cdot 5}{e}\right)^{-\f{5}{3}}
= 39.587\ldots, \\
\vt{C(2,q)}
&\le V_1(q)+V_2(q)
=\left(1+\f{1}{R_{1}(q)}\right)V_1(q)
\le \left(1+\f{1}{R_{1}(5)}\right)V_1(q) \\
&= \kappa_{0}V_1(q), \q \kappa_0 \ceq 1+\f{1}{R_1(5)} = 1.025\ldots.
\end{align}
Since $q = n/2 = pa/2$,
\begin{align}
\vt{C(2, q)} &\le \kappa_{0}q!\left(\f{2e}{q}\right)^{\f{q}{2}}\exp\left(\sqrt{\f{q}{2}}\right)
\label{eq:b2age5C}
= \kappa_{0}\left(\f{pa}{2}\right)!\left(\f{4e}{pa}\right)^{\f{pa}{4}}\exp\left(\f{\sqrt{pa}}{2}\right).
\end{align}

On the other hand, we restrict the sum in \eqref{eq:Ng0b2} to the compositions whose parts are as equal as possible. Let
\begin{align}
u &= \fl{\f{g}{p}},\q v = g - pu.
\end{align}
As in \eqref{eq:Bapq}, define
\begin{align}
\label{eq:b2B}
B(a, p, q) &\ceq 
\begin{dcases}
\binom{p}{v}\binom{a}{2u+1}^{p-v}\binom{a}{2u+3}^{v} & (v > 0), \\
\binom{a}{2u+1}^{p} & (v = 0). \\
\end{dcases}
\end{align}
Then
\begin{align}
\vt{N(2, q, a)} &\ge \f{2(2q-q)!}{2^{2q-p-q+1}a^{p}p!}B(a, p, q)
= \f{q!}{2^{q-p}a^{p}p!}B(a, p, q).
\end{align}
Combining this with \eqref{eq:b2age5C} and using $n=2q=pa$, we obtain
\begin{align}
\f{\vt{N(2, q, a)}}{\vt{C(2, q)} } &\ge \f{q!}{2^{q-p}a^{p}p!}B(a, p, q)\left(\kappa_{0}\left(\f{pa}{2}\right)!\left(\f{4e}{pa}\right)^{\f{pa}{4}}\exp\left(\sqrt{\f{pa}{4}}\right)\right)^{-1} \\
\label{eq:N2C2}
&= \f{\left(\f{pa}{4e}\right)^{\f{pa}{4}}\exp(-\f{\sqrt{pa}}{2})}{\kappa_{0}\cdot 2^{\left(\f{a}{2}-1\right)p}a^{p}p!}B(a, p, q).
\end{align}

Let $R_{0}(a, p)$ denote the right-hand side.
Here, we give a more refined estimate of $B(a,p,q)$ than the one obtained in \secref{sec:Nlower}.
We divide the argument into cases according to the residue class of $a$ modulo~$4$
and estimate $R_{0}$ in each case.

\tabref{tab:b2age5} summarizes the admissible values of $p$, the corresponding
values of $u$ and $v$, and the method used to estimate $R_{0}$ for each
residue class of $a$ modulo~$4$.
\HL

\begin{table}[htbp]
\centering
\small
\begin{tabular}{|c|c|c|c|l|}
\hline
$a \!\!\pmod{4}$ & $p$ & $u$ & $v$ & \multicolumn{1}{c|}{Method} \\
\hline
0 & odd, $p \ge 3$ & $a/4-1$ & $(p+1)/2$ &
\makecell[l]{Monotonicity in $a$ and, for $a=8$, in $p$; \\ reduce to $(8,3)$.} \\
\hline
\multirow[c]{2}{*}[-1.6ex]{1} & $p=2$ & $(a-1)/4$ & $0$ &
\makecell[l]{Monotonicity in $a$; check $a=5, 9$\\ separately.} \\
\cline{2-5}
& \makecell[c]{$p \equiv 2 \!\pmod{4}$, \\ $p \ge 6$} & $(a-5)/4$ & $(3p+2)/4$ &
\makecell[l]{Monotonicity in $a$ and, for $a=5$, in $p$;\\
check $(5,6)$, \!$(5,10)$, \!$(5,14)$, \!$(9,6)$, \!$(9,10)$.} \\
\hline
2 & -- & -- & -- & Does not occur. \\
\hline
3 & \makecell[c]{$p \equiv 2 \!\pmod{4}$, \\ $p\ge 2$} & $(a-3)/4$ & $(p+2)/4$ &
\makecell[l]{Monotonicity in $a$ and, for $a=7$, in $p$;\\
check $(7,2)$, $(7,6)$, $(11,2)$.} \\
\hline
\end{tabular}
\HL
\caption{Case analysis according to $a \pmod{4}$}\label{tab:b2age5}
\end{table}

\noindent
(i) The case $a \equiv 0 \pmod{4}$

Since $a \ge 5$, we have $a \ge 8$. Let $a = 4m$. Then
\begin{align}
m &\ge 2, \q q = \f{pa}{2} = 2pm, \\*
g &= \f{2pm-p+1}{2} = pm - \f{p-1}{2},\q p\text{ is odd},\ p\ge 3,\\*
q &= 2pm \text{ is even},\ q \ge 12, \\*
u &= \fl{\f{g}{p}} = \fl{m - \f{p-1}{2p}} = m - 1 = \f{a}{4}-1, \\*
v &= g-pu = pm - \f{p-1}{2} - p(m-1) = \f{p+1}{2} > 0.
\end{align}
Thus, by \eqref{eq:b2B} and \eqref{eq:N2C2},
\begin{align}
R_{0}(a, p) &= \f{\left(\f{pa}{4e}\right)^{\f{pa}{4}}\exp(-\f{\sqrt{pa}}{2})}{\kappa_{0}\cdot 2^{\left(\f{a}{2}-1\right)p}a^{p}p!}
\binom{p}{\f{p+1}{2}}\binom{a}{\f{a}{2}-1}^{\f{p-1}{2}}\binom{a}{\f{a}{2}+1}^{\f{p+1}{2}} \\
&= \f{\left(\f{pa}{4e}\right)^{\f{pa}{4}}\exp(-\f{\sqrt{pa}}{2})}{\kappa_{0}\cdot 2^{\left(\f{a}{2}-1\right)p}a^{p}p!}
\cdot\f{p!}{\left(\f{p+1}{2}\right)!\left(\f{p-1}{2}\right)!} \\
&\sq\cdot\left(\f{a!}{\left(\f{a}{2}-1\right)!\left(\f{a}{2}+1\right)!}\right)^{\f{p-1}{2}}
\left(\f{a!}{\left(\f{a}{2}+1\right)!\left(\f{a}{2}-1\right)!}\right)^{\f{p+1}{2}} \\
&= \f{\left(\f{pa}{4e}\right)^{\f{pa}{4}}\exp(-\f{\sqrt{pa}}{2})}{\kappa_{0}\cdot 2^{\left(\f{a}{2}-1\right)p}\left(\f{p+1}{2}\right)!\left(\f{p-1}{2}\right)!}
\left(\f{(a-1)!}{\left(\f{a}{2}+1\right)!\left(\f{a}{2}-1\right)!}\right)^{p}.
\end{align}

To prove that $R_{0}(a,p)>1$, we first show the following monotonicity properties:
\begin{itemize}
\item For each odd $p \ge 3$, $R_{0}(a,p)$ is strictly increasing as $a$ ranges over
integers $a \ge 8$ with $a \equiv 0 \pmod{4}$.
\item $R_{0}(8,p)$ is strictly increasing as $p$ ranges over odd integers $p \ge 3$.
\end{itemize}

Since successive values of $a$ in this case differ by $4$, the ratio of successive values of $R_{0}$ with respect to $a$ is
\begin{align}
R_{0,a}(a, p) &\ceq \f{R_{0}(a+4, p)}{R_{0}(a, p)} \\
&= \f{\left(\f{p(a+4)}{4e}\right)^{\f{p(a+4)}{4}}\exp(-\f{\sqrt{p(a+4)}}{2})}{\kappa_{0}\cdot 2^{\left(\f{a}{2}+1\right)p}\left(\f{p+1}{2}\right)!\left(\f{p-1}{2}\right)!}
\left(\f{(a+3)!}{\left(\f{a}{2}+3\right)!\left(\f{a}{2}+1\right)!}\right)^{p} \\
&\sq\cdot\f{\kappa_{0}\cdot 2^{\left(\f{a}{2}-1\right)p}\left(\f{p+1}{2}\right)!\left(\f{p-1}{2}\right)!}{\left(\f{pa}{4e}\right)^{\f{pa}{4}}\exp(-\f{\sqrt{pa}}{2})}
\left(\f{\left(\f{a}{2}+1\right)!\left(\f{a}{2}-1\right)!}{(a-1)!}\right)^{p} \\
&= \left(\f{\left(1+\f{4}{a}\right)^{\f{a}{4}}p(a+1)(a+3)}{e(a+6)}\right)^{p}\exp(-\f{\sqrt{p}}{2}(\sqrt{a+4}-\sqrt{a})).
\end{align}
Since
\begin{align}
\dv{a}\f{(a+1)(a+3)}{a+6} = \f{a^{2}+12a+21}{(a+6)^{2}} > 0 \q (a \ge 8),
\end{align}
$(a+1)(a+3)/(a+6)$ is strictly increasing in $a$ for $a \ge 8$.
Together with the monotonicity facts stated in \secref{sec:strategy}, this gives
\begin{align}
R_{0,a}(a, p) &\ge R_{0,a}(8, p) = \left(\f{\left(1+\f{4}{8}\right)^{\f{8}{4}}p(8+1)(8+3)}{e(8+6)}\right)^{p}\exp(-\f{\sqrt{p}}{2}(\sqrt{8+4}-\sqrt{8})) \\
&= \left(\f{891p}{56e\exp(\f{\sqrt{3}-\sqrt{2}}{\sqrt{p}})}\right)^{p}.
\end{align}
Since
\begin{align}
\f{p}{\exp\left(\f{\sqrt{3}-\sqrt{2}}{\sqrt{p}}\right)}
\end{align}
is strictly increasing in $p$ for $p>0$ and
\begin{align}
\f{891\cdot3}{56e\exp\left(\f{\sqrt{3}-\sqrt{2}}{\sqrt{3}}\right)} = 14.615\ldots >1,
\end{align}
we have, for $p\ge3$,
\begin{align}
R_{0,a}(a, p) &\ge \left(\f{891\cdot 3}{56e\exp(\f{\sqrt{3}-\sqrt{2}}{\sqrt{3}})}\right)^{3} = 3122.236\ldots > 1.
\end{align}
Hence, for each fixed odd $p \ge 3$, $R_0(a,p)$ is strictly increasing as $a$ ranges over
integers $a \ge 8$ congruent to $0$ modulo~$4$.

Since $p$ is odd, successive admissible values of $p$ differ by $2$. Hence the ratio of successive values of
$R_{0}(8, p)$ with respect to $p$ is
\begin{align}
\f{R_{0}(8, p+2)}{R_{0}(8, p)} &= \f{\left(\f{2(p+2)}{e}\right)^{2(p+2)}\exp(-\sqrt{2(p+2)})}{\kappa_{0}\cdot 2^{3(p+2)}\left(\f{p+3}{2}\right)!\left(\f{p+1}{2}\right)!}
\left(\f{7!}{5!\cdot 3!}\right)^{p+2} \\
&\sq\cdot\f{\kappa_{0}\cdot 2^{3p}\left(\f{p+1}{2}\right)!\left(\f{p-1}{2}\right)!}{\left(\f{2p}{e}\right)^{2p}\exp(-\sqrt{2p})}
\left(\f{7!}{5!\cdot 3!}\right)^{-p} \\
&= \f{49\left(1+\f{2}{p}\right)^{2p}(p+2)^{4}}{e^{4}(p+1)(p+3)}\exp(-\sqrt{2}(\sqrt{p+2}-\sqrt{p})).
\end{align}
Using $(p+1)(p+3) = p^{2}+4p+3 < p^{2}+4p+4 = (p+2)^{2}$ together with the monotonicity facts stated in
\secref{sec:strategy}, we obtain
\begin{align}
\f{R_{0}(8, p+2)}{R_{0}(8, p)} &> \f{49\left(1+\f{2}{p}\right)^{2p}(p+2)^{2}}{e^{4}}\exp(-\sqrt{2}(\sqrt{p+2}-\sqrt{p})) \\
&\ge \f{49\left(1+\f{2}{3}\right)^{2\cdot 3}(3+2)^{2}}{e^{4}}\exp(-\sqrt{2}(\sqrt{3+2}-\sqrt{3})) = 235.771\ldots > 1.
\end{align}
Therefore, $R_{0}(8, p)$ is strictly increasing as $p$ ranges over its admissible values.

Thus, by the monotonicity properties established above and a direct computation,
\begin{align}
\f{\vt{N(2, q, a)}}{\vt{C(2, q)}} 
&\ge R_{0}(a, p) \ge R_{0}(8, p) \ge R_{0}(8, 3)
= 3.262\ldots > 1.
\end{align}
\HL

\noindent
(ii) The case $a \equiv 1 \pmod{4}$

We have $a \ge 5$. Let $a = 4m+1$. Then
\begin{align}
m &\ge 1,\q q = \f{pa}{2} = \f{p(4m+1)}{2} = 2pm+\f{p}{2}, \\
g &= \f{2pm+\f{p}{2}-p+1}{2} = pm - \f{p-2}{4},\q p \equiv 2 \npmod{4},\ p \ge 2, \\
q &= 2pm+\f{p}{2}\text{ is odd},\ q \ge 5, \\
u &= \fl{\f{g}{p}} = \fl{m - \f{p-2}{4p}} =
\begin{dcases}
m = \f{a-1}{4} & (p = 2), \\
m-1 = \f{a-5}{4} & (p \ge 6),
\end{dcases} \\
v &= g-pu = 
\begin{dcases}
2m - 2m = 0 & (p=2), \\
pm-\f{p-2}{4} - p(m-1) = \f{3p+2}{4} > 0 & (p \ge 6).
\end{dcases}
\end{align}
\HL

\noindent
(ii-1) The subcase $p=2$

By \eqref{eq:b2B} and \eqref{eq:N2C2},
\begin{align}
R_{0}(a, 2) &= \f{\left(\f{2a}{4e}\right)^{\f{2a}{4}}\exp(-\f{\sqrt{2a}}{2})}{\kappa_{0}\cdot 2^{\left(\f{a}{2}-1\right)\cdot 2}a^{2}2!}
\binom{a}{\f{a+1}{2}}^{2}
= \f{\left(\f{a}{2e}\right)^{\f{a}{2}}\exp(-\sqrt{\f{a}{2}})}{\kappa_{0}\cdot 2^{a-1}}\left(\f{(a-1)!}{\left(\f{a+1}{2}\right)!\left(\f{a-1}{2}\right)!}\right)^{2}.
\end{align}
Since successive admissible values of $a$ differ by $4$, the ratio of successive values of
$R_{0}(a,2)$ with respect to $a$ is
\begin{align}
&\f{R_{0}(a+4, 2)}{R_{0}(a,2)} \\
&\sq = \f{\left(\f{a+4}{2e}\right)^{\f{a+4}{2}}\exp(-\sqrt{\f{a+4}{2}})}{\kappa_{0}\cdot 2^{a+3}}\left(\f{(a+3)!}{\left(\f{a+5}{2}\right)!\left(\f{a+3}{2}\right)!}\right)^{2}
\f{\kappa_{0}\cdot 2^{a-1}}{\left(\f{a}{2e}\right)^{\f{a}{2}}\exp(-\sqrt{\f{a}{2}})}\left(\f{\left(\f{a+1}{2}\right)!\left(\f{a-1}{2}\right)!}{(a-1)!}\right)^{2} \\
&\sq= \f{\left(1+\f{4}{a}\right)^{\f{a}{2}}\left(\f{a+4}{2}\right)^{2}}{2^{4}e^{2}}\left(\f{a(a+1)(a+2)(a+3)}{\f{a+3}{2}\f{a+5}{2}\f{a+1}{2}\f{a+3}{2}}\right)^{2}\exp(-\f{\sqrt{a+4}-\sqrt{a}}{\sqrt{2}}) \\
&\sq= \f{4\left(1+\f{4}{a}\right)^{\f{a}{2}}}{e^{2}}\left(\f{a(a+2)(a+4)}{(a+3)(a+5)}\right)^{2}\exp(-\f{\sqrt{a+4}-\sqrt{a}}{\sqrt{2}}).
\end{align}
Let
\begin{align}
h(a) &\ceq \f{a(a+2)(a+4)}{(a+3)(a+5)}.
\end{align}
Then
\begin{align}
\dv{a}\log h(a) &= \f{1}{a} + \f{1}{a+2} + \f{1}{a+4} - \f{1}{a+3} - \f{1}{a+5} \\
\label{eq:4m1inc2}
&= \f{1}{a} + \f{1}{(a+2)(a+3)} + \f{1}{(a+4)(a+5)} > 0.
\end{align}
Hence $h(a)$ is strictly increasing in $a$. Therefore,
\begin{align}
\f{R_{0}(a+4, 2)}{R_{0}(a,2)} &\ge \f{R_{0}(9, 2)}{R_{0}(5,2)} = \f{4\left(1+\f{4}{5}\right)^{\f{5}{2}}}{e^{2}}\left(\f{5(5+2)(5+4)}{(5+3)(5+5)}\right)^{2}\exp(-\f{\sqrt{5+4}-\sqrt{5}}{\sqrt{2}}) \\
&= 21.256\ldots > 1.
\end{align}
Hence $R_{0}(a, 2)$ is strictly increasing in $a$. A direct computation gives
\begin{align}
&R_{0}(5, 2) < 1,\q R_{0}(9,2) < 1, \\
&R_{0}(13, 2) > 1.
\end{align}
Therefore, in this subcase, $\vt{N}>\vt{C}$ for $a\ge13$.

For the passports corresponding to $(a, p) = (5,2)$ and $(9,2)$, we obtain the following results:
\begin{itemize}
\item $[10, 2^{5}, 5^{2}]$ \\
We have $\vt{N} = 33 < 85 = \vt{C}$. However, for
\begin{align}
x = (1\ 2\ 3\ 4\ 5\ 6\ 7\ 8\ 9\ 10),\q y=(1\ 3)(2\ 6)(4\ 8)(5\ 10)(7\ 9),
\end{align}
$(xy)^{-1}$ has cycle type $(5^{2})$ and $\AD \cong \{1\}$.
\item $[18, 2^{9}, 9^{2}]$ \\
We have $\vt{N} = 383985 > 27171 = \vt{C}$.
\end{itemize}
\HL

\noindent
(ii-2) The subcase $p \ge 6$

By \eqref{eq:b2B} and \eqref{eq:N2C2},
\begin{align}
R_{0}(a, p) &= \f{\left(\f{pa}{4e}\right)^{\f{pa}{4}}\exp(-\f{\sqrt{pa}}{2})}{\kappa_{0}\cdot 2^{\left(\f{a}{2}-1\right)p}a^{p}p!}\binom{p}{\f{3p+2}{4}}\binom{a}{\f{a-3}{2}}^{\f{p-2}{4}}\binom{a}{\f{a+1}{2}}^{\f{3p+2}{4}} \\
&= \f{\left(\f{pa}{4e}\right)^{\f{pa}{4}}\exp(-\f{\sqrt{pa}}{2})}{\kappa_{0}\cdot 2^{\left(\f{a}{2}-1\right)p}a^{p}p!}\cdot
\f{p!}{\left(\f{3p+2}{4}\right)!\left(\f{p-2}{4}\right)!}\left(\f{a!}{\left(\f{a-3}{2}\right)!\left(\f{a+3}{2}\right)!}\right)^{\f{p-2}{4}}\left(\f{a!}{\left(\f{a+1}{2}\right)!\left(\f{a-1}{2}\right)!}\right)^{\f{3p+2}{4}} \\
&= \f{((a-1)!)^{p}\left(\f{pa}{4e}\right)^{\f{pa}{4}}\exp(-\f{\sqrt{pa}}{2})}{\kappa_{0}\cdot 2^{\left(\f{a}{2}-1\right)p}\left(\f{3p+2}{4}\right)!\left(\f{p-2}{4}\right)!\left(\left(\f{a-3}{2}\right)!\left(\f{a+3}{2}\right)!\right)^{\f{p-2}{4}}\left(\left(\f{a+1}{2}\right)!\left(\f{a-1}{2}\right)!\right)^{\f{3p+2}{4}}}.
\end{align}

To reduce the argument to finitely many cases, we first establish the following monotonicity properties:
\begin{itemize}
\item For each $p \ge 6$ with $p \equiv 2 \pmod{4}$, $R_{0}(a,p)$ is strictly increasing as $a$ ranges over
integers $a \ge 5$ with $a \equiv 1 \pmod{4}$.
\item $R_{0}(5,p)$ is strictly increasing as $p$ ranges over integers $p \ge 6$ with
$p \equiv 2 \pmod{4}$.
\end{itemize}

Since successive admissible values of $a$ differ by $4$, the ratio of successive values of
$R_0(a,p)$ with respect to $a$ is
\begin{align}
R_{0,a}(a,p) &\ceq \f{R_{0}(a+4, p)}{R_{0}(a, p)} \\
&= \f{((a+3)!)^{p}\left(\f{p(a+4)}{4e}\right)^{\f{p(a+4)}{4}}\exp(-\f{\sqrt{p(a+4)}}{2})}{\kappa_{0}\cdot 2^{\left(\f{a}{2}+1\right)p}\left(\f{3p+2}{4}\right)!\left(\f{p-2}{4}\right)!\left(\left(\f{a+1}{2}\right)!\left(\f{a+7}{2}\right)!\right)^{\f{p-2}{4}}\left(\left(\f{a+5}{2}\right)!\left(\f{a+3}{2}\right)!\right)^{\f{3p+2}{4}}} \\
&\sq\cdot\f{\kappa_{0}\cdot 2^{\left(\f{a}{2}-1\right)p}\left(\f{3p+2}{4}\right)!\left(\f{p-2}{4}\right)!\left(\left(\f{a-3}{2}\right)!\left(\f{a+3}{2}\right)!\right)^{\f{p-2}{4}}\left(\left(\f{a+1}{2}\right)!\left(\f{a-1}{2}\right)!\right)^{\f{3p+2}{4}}}{((a-1)!)^{p}\left(\f{pa}{4e}\right)^{\f{pa}{4}}\exp(-\f{\sqrt{pa}}{2})} \\
&= \f{(a(a+1)(a+2)(a+3))^{p}\left(1+\f{4}{a}\right)^{\f{pa}{4}}\left(\f{p(a+4)}{4e}\right)^{p}\exp(-\f{\sqrt{p}}{2}(\sqrt{a+4}-\sqrt{a}))}
{2^{2p}\left(\f{a-1}{2}\f{a+1}{2}\f{a+5}{2}\f{a+7}{2}\right)^{\f{p-2}{4}}\left(\f{a+3}{2}\f{a+5}{2}\f{a+1}{2}\f{a+3}{2}\right)^{\f{3p+2}{4}}} \\
&= \left(\f{ap(a+2)(a+3)(a+4)\left(1+\f{4}{a}\right)^{\f{a}{4}}}{e(a+5)}\right)^{p}
\f{\exp(-\f{\sqrt{p}}{2}(\sqrt{a+4}-\sqrt{a}))}{((a-1)(a+7))^{\f{p-2}{4}}(a+3)^{\f{3p+2}{2}}}.
\end{align}
Since $(a-1)(a+7) = a^{2}+6a-7 < a^{2}+6a+9 = (a+3)^{2}$,
\begin{align}
R_{0,a}(a,p) &\ge \left(\f{pa(a+2)(a+3)(a+4)\left(1+\f{4}{a}\right)^{\f{a}{4}}}{e(a+5)}\right)^{p}
\f{\exp(-\f{\sqrt{p}}{2}(\sqrt{a+4}-\sqrt{a}))}{(a+3)^{\f{p-2}{2}}(a+3)^{\f{3p+2}{2}}} \\
\label{eq:La4aii}
&= \left(\f{pa(a+2)(a+4)\left(1+\f{4}{a}\right)^{\f{a}{4}}}{e(a+3)(a+5)}\right)^{p}
\exp(-\f{\sqrt{p}}{2}(\sqrt{a+4}-\sqrt{a})).
\end{align}
Let $f(a,p)$ denote the right-hand side of \eqref{eq:La4aii}.
By \eqref{eq:4m1inc2} and the monotonicity facts stated in \secref{sec:strategy},
$f(a,p)$ is strictly increasing in $a$.
Hence
\needspace{\baselineskip} 
\begin{align}
R_{0,a}(a,p) &\ge f(a, p) \ge f(5, p) \\
&= \left(\f{5p(5+2)(5+4)\left(1+\f{4}{5}\right)^{\f{5}{4}}}{e(5+3)(5+5)}\right)^{p}
\exp(-\f{\sqrt{p}}{2}(\sqrt{5+4}-\sqrt{5})) \\
&= \left(\f{63p}{16e}\left(\f{9}{5}\right)^{\f{5}{4}}\right)^{p}\exp(-\f{\sqrt{p}}{2}(3-\sqrt{5})).
\end{align}

Then,
\begin{align}
\log f(5, p) &= p\left(\log p+\log\f{63}{16e}+\f{5}{4}\log\f{9}{5}\right)-\f{\sqrt{p}}{2}(3-\sqrt{5}),
\end{align}
and for $p \ge 6$,
\begin{align}
\dv{p}\log f(5, p) &= \log p+\log\f{63}{16e}+\f{5}{4}\log\f{9}{5}+1-\f{3-\sqrt{5}}{4\sqrt{p}} \\
&\ge \log 6+\log\f{63}{16e}+\f{5}{4}\log\f{9}{5}+1-\f{3-\sqrt{5}}{4\sqrt{6}} \\
&=3.819\ldots > 0.
\end{align}
Hence $f(5, p)$ is strictly increasing for $p \ge 6$. Therefore,
\begin{align}
R_{0,a}(a,p) &\ge f(5, p) \ge f(5, 6) \\
&= \left(\f{63 \cdot 6}{16e}\left(\f{9}{5}\right)^{\f{5}{4}}\right)^{6}\exp\left(-\f{\sqrt{6}}{2}(3-\sqrt{5})\right) \\
&=13889045.590\ldots>1.
\end{align}
Thus, for each admissible $p\ge6$, $R_0(a,p)$ is strictly increasing as $a$ ranges over
integers $a\ge5$ with $a\equiv1\pmod{4}$.

Since $p\equiv2\pmod{4}$, successive admissible values of $p$ differ by $4$, and the ratio of
successive values of $R_{0}(5, p)$ with respect to $p$ is
\begin{align}
\f{R_{0}(5, p+4)}{R_{0}(5, p)} &=
\f{(4!)^{p+4}\left(\f{5(p+4)}{4e}\right)^{\f{5(p+4)}{4}}\exp(-\f{\sqrt{5(p+4)}}{2})}{\kappa_{0}\cdot 2^{\f{3}{2}(p+4)}\left(\f{3p+14}{4}\right)!\left(\f{p+2}{4}\right)!(4!)^{\f{p+2}{4}}(3!\cdot 2!)^{\f{3p+14}{4}}} \\
&\sq\cdot\f{\kappa_{0}\cdot 2^{\f{3}{2}p}\left(\f{3p+2}{4}\right)!\left(\f{p-2}{4}\right)!(4!)^{\f{p-2}{4}}(3!\cdot 2!)^{\f{3p+2}{4}}}{(4!)^{p}\left(\f{5p}{4e}\right)^{\f{5p}{4}}\exp(-\f{\sqrt{5p}}{2})} \\
&= \f{(4!)^{4}\left(1+\f{4}{p}\right)^{\f{5p}{4}}\left(\f{5(p+4)}{4e}\right)^{5}\exp(-\f{\sqrt{5}}{2}(\sqrt{p+4}-\sqrt{p}))}
{2^{6}\f{3p+14}{4}\f{3p+10}{4}\f{3p+6}{4}\f{p+2}{4}4!(3!\cdot 2!)^{3}} \\
&= \f{3125\left(1+\f{4}{p}\right)^{\f{5p}{4}}(p+4)^{5}\exp(-\f{\sqrt{5}}{2}(\sqrt{p+4}-\sqrt{p}))}
{96e^{5}(p+2)^{2}(3p+10)(3p+14)}.
\end{align}
Since
\begin{align}
p+2 < p+4,\q 3p+10 < 3(p+4),\q 3p+14 < 4(p+4),
\end{align}
we obtain
\begin{align}
\f{R_{0}(5, p+4)}{R_{0}(5, p)} &> \f{3125\left(1+\f{4}{p}\right)^{\f{5p}{4}}(p+4)^{5}\exp(-\f{\sqrt{5}}{2}(\sqrt{p+4}-\sqrt{p}))}
{96e^{5}(p+4)^{2}\cdot 3(p+4)\cdot 4(p+4)} \\
&= \f{3125\left(1+\f{4}{p}\right)^{\f{5p}{4}}(p+4)\exp(-\f{\sqrt{5}}{2}(\sqrt{p+4}-\sqrt{p}))}{1152e^{5}}.
\end{align}
By the monotonicity facts stated in \secref{sec:strategy},
\begin{align}
\f{R_{0}(5, p+4)}{R_{0}(5, p)} &\ge \f{3125\left(1+\f{4}{6}\right)^{\f{5\cdot 6}{4}}(6+4)\exp(-\f{\sqrt{5}}{2}(\sqrt{6+4}-\sqrt{6}))}{1152e^{5}} \\
&= 3.799\ldots > 1.
\end{align}
Therefore, $R_0(5,p)$ is strictly increasing as $p$ ranges over integers
$p\ge6$ with $p\equiv2\pmod{4}$.

A direct computation gives
\begin{align}
&R_{0}(5, 6) < 1,\q R_{0}(5, 10) < 1, \\
&R_{0}(5, 14) > 1,\q R_{0}(9, 6) > 1,\q R_{0}(9, 10) > 1.
\end{align}
By the monotonicity properties established above, if $p\ge14$, then
\begin{align}
R_{0}(a,p)\ge R_{0}(5,p)\ge R_{0}(5,14)>1,
\end{align}
while if $p=6$ or $10$ and $a\ge9$, then
\begin{align}
R_{0}(a,p)\ge R_{0}(9,p)>1.
\end{align}
Hence $R_{0}(a,p)>1$ for all admissible pairs $(a,p)$ except for
$(5,6)$ and $(5,10)$.

For the passports corresponding to $(a, p) = (5, 6)$ and $(5, 10)$, we obtain the following results:
\begin{itemize}
\item $[30, 2^{15}, 5^{6}]$ \\
$\vt{N} = 1038647610 > 664631175 = \vt{C}$.
\item $[50, 2^{25}, 5^{10}]$ \\
$\vt{N} = 3178849676735117385 > 142389639795904325 = \vt{C}$.
\end{itemize}

Thus, we obtain $\vt{N}>\vt{C}$ for all admissible pairs with
$a\equiv1\pmod{4}$ and $p\ge6$.
\HL

\noindent
(iii) The case $a \equiv 2 \pmod{4}$

Let $a = 4m+2$. Then
\begin{align}
q &= \f{pa}{2} = p(2m+1), \\
g &= \f{p(2m+1)-p+1}{2} = pm + \f{1}{2}.
\end{align}
Since the genus $g$ must be an integer, this case does not occur.
\HL

\noindent
(iv) The case $a \equiv 3 \pmod{4}$

Since $a \ge 5$, we have $a \ge 7$. Let $a = 4m+3$. Then
\begin{align}
m &\ge 1,\q q = \f{pa}{2} = \f{p(4m+3)}{2} = 2pm+ \f{3p}{2}, \\
g &= \f{2pm+\f{3p}{2}-p+1}{2} = pm + \f{p+2}{4},\q p \equiv 2 \npmod{4},\ p \ge 2, \\
q &=  2pm+ \f{3p}{2}\text{ is odd},\ q \ge 7, \\
u &= \fl{\f{g}{p}} = \fl{m + \f{p+2}{4p}} = m = \f{a-3}{4}, \\
v &= g-pu = pm + \f{p+2}{4} - pm = \f{p+2}{4} > 0.
\end{align}

By \eqref{eq:b2B} and \eqref{eq:N2C2},
\begin{align}
R_{0}(a, p) &= \f{\left(\f{pa}{4e}\right)^{\f{pa}{4}}\exp(-\f{\sqrt{pa}}{2})}{\kappa_{0}\cdot 2^{\left(\f{a}{2}-1\right)p}a^{p}p!}
\binom{p}{\f{p+2}{4}}\binom{a}{\f{a-1}{2}}^{\f{3p-2}{4}}\binom{a}{\f{a+3}{2}}^{\f{p+2}{4}} \\
&= \f{\left(\f{pa}{4e}\right)^{\f{pa}{4}}\exp(-\f{\sqrt{pa}}{2})}{\kappa_{0}\cdot 2^{\left(\f{a}{2}-1\right)p}a^{p}p!}
\f{p!}{\left(\f{p+2}{4}\right)!\left(\f{3p-2}{4}\right)!}
\left(\f{a!}{\left(\f{a-1}{2}\right)!\left(\f{a+1}{2}\right)!}\right)^{\f{3p-2}{4}}\left(\f{a!}{\left(\f{a+3}{2}\right)!\left(\f{a-3}{2}\right)!}\right)^{\f{p+2}{4}} \\
&= \f{((a-1)!)^{p}\left(\f{pa}{4e}\right)^{\f{pa}{4}}\exp(-\f{\sqrt{pa}}{2})}
{\kappa_{0}\cdot 2^{\left(\f{a}{2}-1\right)p}\left(\f{p+2}{4}\right)!\left(\f{3p-2}{4}\right)!\left(\left(\f{a-1}{2}\right)!\left(\f{a+1}{2}\right)!\right)^{\f{3p-2}{4}}\left(\left(\f{a+3}{2}\right)!\left(\f{a-3}{2}\right)!\right)^{\f{p+2}{4}}}.
\end{align}

To reduce the argument to finitely many cases, we first establish the following monotonicity properties:
\begin{itemize}
\item For each $p \ge 2$ with $p \equiv 2 \pmod{4}$, $R_{0}(a,p)$ is strictly increasing as $a$ ranges over
integers $a \ge 7$ with $a \equiv 3 \pmod{4}$.
\item $R_{0}(7,p)$ is strictly increasing as $p$ ranges over integers $p \ge 2$ with
$p \equiv 2 \pmod{4}$.
\end{itemize}

Since successive admissible values of $a$ in this case differ by $4$,
the ratio of successive values of $R_{0}$ with respect to $a$ is
\begin{align}
R_{0,a}(a, p) &\ceq \f{R_{0}(a+4, p)}{R_{0}(a, p)} \\
&= \f{((a+3)!)^{p}\left(\f{p(a+4)}{4e}\right)^{\f{p(a+4)}{4}}\exp(-\f{\sqrt{p(a+4)}}{2})}
{\kappa_{0}\cdot 2^{\left(\f{a+4}{2}-1\right)p}\left(\f{p+2}{4}\right)!\left(\f{3p-2}{4}\right)!\left(\left(\f{a+3}{2}\right)!\left(\f{a+5}{2}\right)!\right)^{\f{3p-2}{4}}\left(\left(\f{a+7}{2}\right)!\left(\f{a+1}{2}\right)!\right)^{\f{p+2}{4}}} \\
&\sq\cdot\f{\kappa_{0}\cdot 2^{\left(\f{a}{2}-1\right)p}\left(\f{p+2}{4}\right)!\left(\f{3p-2}{4}\right)!\left(\left(\f{a-1}{2}\right)!\left(\f{a+1}{2}\right)!\right)^{\f{3p-2}{4}}\left(\left(\f{a+3}{2}\right)!\left(\f{a-3}{2}\right)!\right)^{\f{p+2}{4}}}
{((a-1)!)^{p}\left(\f{pa}{4e}\right)^{\f{pa}{4}}\exp(-\f{\sqrt{pa}}{2})} \\
&= \f{(a(a+1)(a+2)(a+3))^{p}\left(1+\f{4}{a}\right)^{\f{pa}{4}}\left(\f{p(a+4)}{4e}\right)^{p}\exp(-\f{\sqrt{p}}{2}(\sqrt{a+4}-\sqrt{a}))}
{2^{2p}\left(\f{a+1}{2}\f{a+3}{2}\f{a+3}{2}\f{a+5}{2}\right)^{\f{3p-2}{4}}\left(\f{a+5}{2}\f{a+7}{2}\f{a-1}{2}\f{a+1}{2}\right)^{\f{p+2}{4}}} \\
&= \f{(a(a+2)(a+4)p)^{p}\left(1+\f{4}{a}\right)^{\f{pa}{4}}}{e^{p}((a-1)(a+7))^{\f{p+2}{4}}(a+3)^{\f{p-2}{2}}(a+5)^{p}}\exp(-\f{\sqrt{p}}{2}(\sqrt{a+4}-\sqrt{a})).
\end{align}
Since $(a-1)(a+7) = a^{2}+6a-7 < a^{2}+6a+9 = (a+3)^{2}$,
\begin{align}
\f{R_{0}(a+4, p)}{R_{0}(a, p)} &\ge \f{(a(a+2)(a+4)p)^{p}\left(1+\f{4}{a}\right)^{\f{pa}{4}}}{e^{p}(a+3)^{\f{p+2}{2}}(a+3)^{\f{p-2}{2}}(a+5)^{p}}\exp(-\f{\sqrt{p}}{2}(\sqrt{a+4}-\sqrt{a})) \\
\label{eq:La4aiv}
&= \left(\f{pa(a+2)(a+4)\left(1+\f{4}{a}\right)^{\f{a}{4}}}{e(a+3)(a+5)}\right)^{p}
\exp(-\f{\sqrt{p}}{2}(\sqrt{a+4}-\sqrt{a})).
\end{align}
Let $f(a,p)$ denote the right-hand side of \eqref{eq:La4aiv}.
This is the same function as the one considered in (ii-2), and hence it is
strictly increasing in $a$. Hence
\needspace{\baselineskip} 
\begin{align}
R_{0,a}(a,p) &\ge f(a, p) \ge f(7, p) \\
&= \left(\f{7p(7+2)(7+4)\left(1+\f{4}{7}\right)^{\f{7}{4}}}{e(7+3)(7+5)}\right)^{p}
\exp(-\f{\sqrt{p}}{2}(\sqrt{7+4}-\sqrt{7})) \\
&= \left(\f{363p}{40e}\left(\f{11}{7}\right)^{\f{3}{4}}\right)^{p}\exp(-\f{\sqrt{p}}{2}(\sqrt{11}-\sqrt{7})).
\end{align}

Then,
\begin{align}
\log f(7, p) &= p\left(\log p + \log\f{363}{40e} + \f{3}{4}\log\f{11}{7}\right) - \f{\sqrt{p}}{2}(\sqrt{11}-\sqrt{7}),
\end{align}
and for $p \ge 2$,
\begin{align}
\dv{p}\log f(7, p) &= \log p + \log\f{363}{40e} + \f{3}{4}\log\f{11}{7} + 1 - \f{\sqrt{11}-\sqrt{7}}{4\sqrt{p}} \\
&\ge \log 2 + \log\f{363}{40e} + \f{3}{4}\log\f{11}{7} + 1 - \f{\sqrt{11}-\sqrt{7}}{4\sqrt{2}} \\
&= 3.119\ldots > 0.
\end{align}
Hence $f(7, p)$ is strictly increasing for $p \ge 2$. Therefore,
\begin{align}
R_{0,a}(a,p) &\ge f(7, p) \ge f(7, 2) \\
&= \left(\f{363\cdot 2}{40e}\left(\f{11}{7}\right)^{\f{3}{4}}\right)^{2}\exp(-\f{\sqrt{2}}{2}(\sqrt{11}-\sqrt{7})) \\
&= 54.649\ldots > 1.
\end{align}
Thus, for each admissible $p \ge 2$, $R_{0}(a, p)$ is strictly increasing as $a$ ranges over integers $a \ge 7$
with $a \equiv 3 \pmod{4}$.

Since $p \equiv 2 \pmod{4}$, successive admissible values of $p$ differ by $4$, and the ratio of successive values
of $R_{0}(7, p)$ with respect to $p$ is
\begin{align}
\f{R_{0}(7, p+4)}{R_{0}(7, p)} &=  \f{(6!)^{p+4}\left(\f{7(p+4)}{4e}\right)^{\f{7(p+4)}{4}}\exp(-\f{\sqrt{7(p+4)}}{2})}
{\kappa_{0}\cdot 2^{\f{5}{2}(p+4)}\left(\f{p+6}{4}\right)!\left(\f{3p+10}{4}\right)!(3!\cdot4!)^{\f{3p+10}{4}}(5!\cdot2!)^{\f{p+6}{4}}} \\
&\sq\cdot\f{\kappa_{0}\cdot 2^{\f{5}{2}p}\left(\f{p+2}{4}\right)!\left(\f{3p-2}{4}\right)!(3!\cdot4!)^{\f{3p-2}{4}}(5!\cdot2!)^{\f{p+2}{4}}}
{(6!)^{p}\left(\f{7p}{4e}\right)^{\f{7p}{4}}\exp(-\f{\sqrt{7p}}{2})} \\
&= \f{(6!)^{4}\left(1+\f{4}{p}\right)^{\f{7p}{4}}\left(\f{7(p+4)}{4e}\right)^{7}\exp(-\f{\sqrt{7}}{2}(\sqrt{p+4}-\sqrt{p}))}
{2^{10}\f{p+6}{4}\f{3p+10}{4}\f{3p+6}{4}\f{3p+2}{4}(3!\cdot4!)^{3}\cdot5!\cdot2!} \\
&= \f{5^{3}7^{7}\left(1+\f{4}{p}\right)^{\f{7p}{4}}(p+4)^{7}\exp(-\f{\sqrt{7}}{2}(\sqrt{p+4}-\sqrt{p}))}
{2^{16}e^{7}(p+2)(p+6)(3p+2)(3p+10)}.
\end{align}
Here we have
\begin{align}
p + 2 < p+4,\q p+6 < 2(p+4),\q 3p+2 < 3(p+4),\q 3p+10 < 3(p+4).
\end{align}
Combining these inequalities with the monotonicity facts stated in
\secref{sec:strategy}, we obtain, for $p \ge 2$,
\begin{align}
\f{R_{0}(7, p+4)}{R_{0}(7, p)} &> \f{5^{3}7^{7}\left(1+\f{4}{p}\right)^{\f{7p}{4}}(p+4)^{7}\exp(-\f{\sqrt{7}}{2}(\sqrt{p+4}-\sqrt{p}))}
{2^{16}e^{7}\cdot2\cdot3\cdot3(p+4)^{4}} \\
&= \f{5^{3}7^{7}\left(1+\f{4}{p}\right)^{\f{7p}{4}}(p+4)^{3}\exp(-\f{\sqrt{7}}{2}(\sqrt{p+4}-\sqrt{p}))}
{2^{17}3^{2}e^{7}} \\
&\ge \f{5^{3}7^{7}\left(1+\f{4}{2}\right)^{\f{7\cdot 2}{4}}(2+4)^{3}\exp(-\f{\sqrt{7}}{2}(\sqrt{2+4}-\sqrt{2}))}{2^{17}3^{2}e^{7}} \\
&= 204.350\ldots > 1.
\end{align}
Therefore, $R_{0}(7, p)$ is strictly increasing as $p$ ranges over integers $p \ge 2$ with $p \equiv 2 \pmod{4}$.

A direct computation gives
\begin{align}
&R_{0}(7, 2) < 1, \\
&R_{0}(7, 6) > 1,\q R_{0}(11, 2) > 1.
\end{align}
By the monotonicity properties established above, if $p \ge 6$, then
\begin{align}
R_{0}(a, p)\ge R_{0}(7, p)\ge R_{0}(7, 6) > 1,
\end{align}
while if $p=2$ and $a \ge 11$, then
\begin{align}
R_{0}(a, p)\ge R_{0}(11, 2)>1.
\end{align}
Hence $R_{0}(a,p)>1$ for all admissible pairs $(a,p)$ except for $(7, 2)$.

For the passport corresponding to $(a,p)=(7,2)$, namely $[7^{2},2^{7},14]$, we have
\begin{align}
\vt{N}=2385>1309=\vt{C}.
\end{align}
Thus, $\vt{N}>\vt{C}$ for all admissible pairs with $a \equiv 3 \pmod{4}$.
\HL

By (i)--(iv), we conclude that every passport $[n,2^{q},a^{p}]$ with
$2\le p<q$, $a\ge5$, and genus at least~$2$ admits a \dde with trivial automorphism group.
\end{proof}
\HL


\subsection{\texorpdfstring{The Subcase $b=3$}{The Subcase b=3}}
\label{sec:class4b3}

\begin{prop}
\label{prop:class4b3}
If a uniform passport $[n, 3^{q}, a^{p}]$ with $n = pa = 3q$ and $2 \le p < q$ has genus
at least~$2$, then it admits a \dde with trivial automorphism group.
\end{prop}

\begin{proof}
By \eqref{eq:genus-uc}, the genus is
\begin{align}
g &= \f{3q-(p+q)+1}{2} = q - \f{p-1}{2}.
\end{align}
Since $g$ is an integer and $p\ge2$, $p$ is odd and hence $p\ge3$.
Since $p<q$, we also have $q\ge4$.

Substituting $b=3$ into \eqref{eq:N0g}, we obtain
\begin{align}
\vt{N(3, q, a)}
&\ge \f{a(n-p)!}{2^{2g}3^{q}q!}
\sum_{(j_{1},\dotsc,j_{q})\vDash g}
\prod_{k=1}^{q}\binom{3}{2j_{k}+1}
= \f{\f{3q}{p}(3q-p)!}{2^{2q-p+1}3^{q}q!}
\sum_{(j_{1},\dotsc,j_{q})\vDash g}
\prod_{k=1}^{q}\binom{3}{2j_{k}+1} \\
&= \f{q(3q-p)!}{2^{2q-p+1}3^{q-1}pq!}
\sum_{(j_{1},\dotsc,j_{q})\vDash g}
\prod_{k=1}^{q}\binom{3}{2j_{k}+1}.
\end{align}
The product in each summand is nonzero only if $2j_{k} + 1 \le 3$ for all $1 \le k \le q$. Hence only compositions whose parts
are $0$ or $1$ contribute to the sum.
Since such compositions have $g$ copies of $1$ and $q-g$ copies of $0$,
\begin{align}
\sum_{(j_{1},\dotsc,j_{q})\vDash g}
\prod_{k=1}^{q}\binom{3}{2j_{k}+1} &= \binom{q}{g}\binom{3}{3}^{g}\binom{3}{1}^{q-g} = 3^{q-g}\binom{q}{g}
= \f{3^{\f{p-1}{2}}q!}{\left(q-\f{p-1}{2}\right)!\left(\f{p-1}{2}\right)!}.
\end{align}
Therefore,
\begin{align}
\vt{N(3, q, a)} &\ge \f{q(3q-p)!}{2^{2q-p+1}3^{q-1}pq!}\cdot\f{3^{\f{p-1}{2}}q!}{\left(q-\f{p-1}{2}\right)!\left(\f{p-1}{2}\right)!} \\
\label{eq:N3qalow}
&= \f{q(3q-p)!}{2^{2q-p+1}3^{q-\f{p+1}{2}}p\left(q-\f{p-1}{2}\right)!\left(\f{p-1}{2}\right)!}.
\end{align}
We regard the right-hand side as a function of the real variable $p$, with $q$ fixed and $3 \le p \le q-1$. Let
\begin{align}
f_{q}(p) \ceq \f{q\Gamma(3q-p+1)}{2^{2q-p+1}3^{q-\f{p+1}{2}}p\Gamma\left(q-\f{p}{2}+\f{3}{2}\right)\Gamma\left(\f{p}{2}+\f{1}{2}\right)}.
\end{align}
Then $f_q(p)$ coincides with the right-hand side of \eqref{eq:N3qalow}
for odd integers $p$ with $3\le p\le q-1$.

Moreover,
\begin{align}
\log f_{q}(p) &= \log q + \log\Gamma(3q-p+1) - (2q-p+1)\log 2 - \left(q-\f{p+1}{2}\right)\log 3 \\
&\sq - \log p - \log\Gamma\left(q-\f{p}{2}+\f{3}{2}\right) - \log\Gamma\left(\f{p}{2}+\f{1}{2}\right), \\
\dv{p}\log f_{q}(p) &= -\psi(3q-p+1) + \log (2\sqrt{3}) - \f{1}{p} + \f{1}{2}\psi\left(q-\f{p}{2}+\f{3}{2}\right)
- \f{1}{2}\psi\left(\f{p}{2}+\f{1}{2}\right),
\end{align}
where
\begin{align}
\psi(x) = \dv{x}\log \Gamma(x) = \f{\Gamma'(x)}{\Gamma(x)}
\end{align}
is the digamma function.

By \lemref{lem:digamma}, we have
\begin{align}
\psi(3q-p+1) &> \log\left(3q-p+\f{1}{2}\right), \\
\psi\left(q-\f{p}{2}+\f{3}{2}\right) &< \log\left(q-\f{p}{2}+\f{3}{2}\right), \\
\psi\left(\f{p}{2}+\f{1}{2}\right) &> \log\f{p}{2}.
\end{align}
Hence,
\begin{align}
\dv{p}\log f_{q}(p) &< -\log\left(3q-p+\f{1}{2}\right)+ \log (2\sqrt{3}) - \f{1}{p} + \f{1}{2}\log\left(q-\f{p}{2}+\f{3}{2}\right) -\f{1}{2}\log\f{p}{2} \\
&= -\f{1}{p} + \log\f{4\sqrt{3}\sqrt{\f{2q+3}{p}-1}}{6q-2p+1}.
\end{align}
Since $3 \le p \le q-1$,
\begin{align}
\f{4\sqrt{3}\sqrt{\f{2q+3}{p}-1}}{6q-2p+1} &\le \f{4\sqrt{3}\sqrt{\f{2q+3}{3}-1}}{6q-2(q-1)+1} = \f{4\sqrt{2q}}{4q+3}
= \f{4\sqrt{2}}{4\sqrt{q}+\f{3}{\sqrt{q}}} \\
&\le \f{4\sqrt{2}}{2\sqrt{4\sqrt{q}\cdot\f{3}{\sqrt{q}}}} = \sqrt{\f{2}{3}}.
\end{align}
Therefore,
\begin{align}
\dv{p}\log f_{q}(p) &< -\f{1}{p} + \log \sqrt{\f{2}{3}} < -\f{1}{p} < 0.
\end{align}
Thus, $f_{q}(p)$ is strictly decreasing, and attains its minimum at the largest admissible value of $p$.

Let $p_{\max}$ denote the largest odd integer $p$ such that
$3\le p\le q-1$ and $p\mid 3q$. Then
\begin{itemize}
\item When $q \equiv 0 \pmod{4}$ \\
Let $q=4\cdot2^l k$, where $l\ge0$ and $k$ is odd.
Then $n=3q=3\cdot2^{l+2}k$, and the largest odd divisor of $n$ is $3k$.
Hence
\begin{align}
p_{\max} = 3k = \f{3}{2^{l+2}}q \le \f{3}{4}q.
\end{align}
\item When $q \equiv 2 \pmod{4}$ \\
Let $q=4m+2=2(2m+1)$ with $m\in\Zp$.
Then $n=3q=6(2m+1)$.
The largest odd divisor of $n$ is $3(2m+1)>q$, while the largest odd divisor
less than $q$ is $2m+1$. Hence
\begin{align}
p_{\max} = 2m+1 = \f{q}{2}.
\end{align}
\item When $q \equiv 1\text{ or }3 \pmod{4}$ \\
Then $n=3q$ is odd.
Since $p_{\max}<q=n/3$ and $p_{\max}\mid n$, the quotient $n/p_{\max}$
is an odd integer greater than $3$, and hence $n/p_{\max}\ge5$.
Therefore,
\begin{align}
p_{\max} \le \f{n}{5} = \f{3}{5}q.
\end{align}
\end{itemize}
In all cases, we have $p_{\max} \le 3q/4$.

Therefore,
\begin{align}
\vt{N(3, q, a)} \ge f_{q}\left(\f{3q}{4}\right)
&= \f{q\Gamma\left(3q-\f{3q}{4}+1\right)}{2^{2q-\f{3q}{4}+1}3^{q-\f{3q}{8}-\f{1}{2}}\cdot\f{3q}{4}\cdot\Gamma\left(q-\f{3q}{8}+\f{3}{2}\right)\Gamma\left(\f{3q}{8}+\f{1}{2}\right)} \\
\label{eq:Nlowerb3}
&= \f{\Gamma\left(\f{9q}{4}+1\right)}{2^{\f{5q}{4}-1}3^{\f{5q}{8}+\f{1}{2}}\Gamma\left(\f{5q}{8}+\f{3}{2}\right)\Gamma\left(\f{3q}{8}+\f{1}{2}\right)}.
\end{align}
Let $L(q)$ denote the right-hand side.

On the other hand, since $n=3q$, \eqref{eq:Cbqsum} gives
\begin{align}
\label{eq:Cupperb3}
\vt{C(3, q)} &\le \vt{C_{\f{n}{3}}(3, q)} + \sum_{\substack{\ell\colon \text{prime},\\ \ell\, \mid\, q,\ \ell \ge 5}} \vt{C_{\f{n}{\ell}}(3, q)} +
\begin{dcases}
\vt{C_{\f{n}{2}}(3, q)} & (2 \mid q) \\
0 & (2 \nmid q) \\
\end{dcases}.
\end{align}
Here, by \lemref{lem:C2q3q4q},
\begin{align}
\label{eq:Cn3b3}
\vt{C_{\f{n}{3}}(3, q)} &\le q!\left(\f{9e}{q}\right)^{\f{q}{3}}\exp\left(2\left(\f{q}{9}\right)^{\f{1}{3}}\right).
\end{align}
Moreover, when $2\mid q$, since $2\nmid b=3$, \propref{prop:Cnl} gives
\begin{align}
\vt{C_{\f{n}{2}}(3, q)} &= \left(\f{3q}{2}\right)!\f{2^{\f{(3-1)q}{2}}}{3^{\f{q}{2}}\left(\f{q}{2}\right)!}
\label{eq:Cn3b2}
= \left(\f{4}{3}\right)^{\f{q}{2}}\f{\left(\f{3q}{2}\right)!}{\left(\f{q}{2}\right)!}.
\end{align}

For odd primes at least~$5$, we use an argument similar to that in \secref{sec:class4b2age5}.

For an odd prime $\ell\mid n=3q$ with $\ell \ge 5$, we have $\ell\mid q$.
Let $m=q/\ell \in \Zp$. Since $\ell\nmid3$, \propref{prop:Cnl} gives
\begin{align}
\vt{C_{\f{n}{\ell}}(3, q)} &= \left(\f{3q}{\ell}\right)!\f{\ell^{\f{2q}{\ell}}}{3^{\f{q}{\ell}}\left(\f{q}{\ell}\right)!}
= \f{\ell^{2m}(3m)!}{3^{m}m!} \q (\ell\colon\text{prime}, \ell \mid q, \ell \ge 5).
\end{align}
By the bounds from Stirling's formula \eqref{eq:Stirling}, we obtain
\begin{align}
\vt{C_{\f{n}{\ell}}(3, q)} &\le \f{\ell^{2m}\sqrt{6 \pi m}\left(\f{3m}{e}\right)^{3m}}{3^{m}\sqrt{2 \pi m}\left(\f{m}{e}\right)^{m}}
\exp(\f{1}{36m}-\f{1}{12m+1}) \\
&= \sqrt{3}\left(\f{3\ell m}{e}\right)^{2m}\exp(\f{1}{36m}-\f{1}{12m+1}) \le \sqrt{3}\left(\f{3\ell m}{e}\right)^{2m}
= \sqrt{3}\left(\f{3q}{e}\right)^{\f{2q}{\ell}}.
\end{align}
Since $3q/e\ge12/e>1$, the right-hand side is decreasing as $\ell$ increases.

Let $\ell_{1}, \ldots, \ell_{s}$ be the distinct odd prime divisors of $n = 3q$ with $\ell_{i} \ge 5$. Since each $\ell_{i}$ divides $q$,
\begin{align}
\ell_{1}\cdots\ell_{s} \le q.
\end{align}
Since $\ell_{i} \ge 5$ for each $i$,
\begin{align}
\ell_{1}\cdots\ell_{s} \ge 5^{s}.
\end{align}
Hence $5^{s} \le q$ and
\begin{align}
s \le \log_{5}q = \f{\log q}{\log 5}.
\end{align}
Therefore, since $3q/e \ge 12/e > 1$,
\begin{align}
\sum_{\substack{\ell\colon \text{prime},\\ \ell\, \mid\, q,\ \ell \ge 5}} \vt{C_{\f{n}{\ell}}(3, q)}
&\le \sum_{\substack{\ell\colon \text{prime},\\ \ell\, \mid\, q,\ \ell \ge 5}} \sqrt{3}\left(\f{3q}{e}\right)^{\f{2q}{\ell}}
\le \sum_{\substack{\ell\colon \text{prime},\\ \ell\, \mid\, q,\ \ell \ge 5}} \sqrt{3}\left(\f{3q}{e}\right)^{\f{2q}{5}} \\
\label{eq:Cn3b5}
&\le \f{\sqrt{3}\log q}{\log 5}\left(\f{3q}{e}\right)^{\f{2q}{5}}.
\end{align}
Thus, by \eqref{eq:Cupperb3}, \eqref{eq:Cn3b3}, \eqref{eq:Cn3b2}, and
\eqref{eq:Cn3b5},
\begin{align}
\vt{C(3, q)} &\le V_{1}(q) + V_{2}(q) + V_{3}(q), \\
V_{1}(q) &\ceq q!\left(\f{9e}{q}\right)^{\f{q}{3}}\exp\left(2\left(\f{q}{9}\right)^{\f{1}{3}}\right), \\
V_{2}(q) &\ceq \f{\sqrt{3}\log q}{\log 5}\left(\f{3q}{e}\right)^{\f{2q}{5}}, \\
\label{eq:b3V3}
V_{3}(q) &\ceq \left(\f{4}{3}\right)^{\f{q}{2}}\f{\left(\f{3q}{2}\right)!}{\left(\f{q}{2}\right)!}
= \left(\f{4}{3}\right)^{\f{q}{2}}\f{\Gamma\left(\f{3q}{2}+1\right)}{\Gamma\left(\f{q}{2}+1\right)}.
\end{align}
Here, for simplicity, we include $V_3(q)$ in the upper bound also when $q$ is odd,
although there is no contribution corresponding to $\ell=2$ in that case.
For noninteger $M$, we define $M!\ceq\Gamma(M+1)$.

Combining this with \eqref{eq:Nlowerb3}, we obtain
\begin{align}
\label{eq:NCV}
\f{\vt{N(3, q, a)}}{\vt{C(3, q)}} &\ge \f{L(q)}{V_{1}(q) + V_{2}(q) + V_{3}(q)}.
\end{align}

Let
\begin{align}
R_{0}(q) \ceq \f{L(q)}{V_{3}(q)},\q R_{1}(q) &\ceq \f{V_{3}(q)}{V_{1}(q)}, \q R_{2}(q) \ceq \f{V_{3}(q)}{V_{2}(q)}.
\end{align}
Then
\begin{align}
\label{eq:NVR}
\f{\vt{N(3, q, a)}}{\vt{C(3, q)}} &\ge \f{L(q)}{\f{V_{3}(q)}{R_{1}(q)} + \f{V_{3}(q)}{R_{2}(q)} + V_{3}(q)}
= \f{R_{0}(q)}{1+ \f{1}{R_{1}(q)} + \f{1}{R_{2}(q)}}.
\end{align}

A direct computation gives
\begin{align}
R_{0}(q) &< 1\q (4 \le q \le 10), \\
R_{0}(11) &= 1.397\ldots > 1.
\end{align}

We first consider the case $q\ge11$ and estimate $R_0(q)$, $R_1(q)$, and $R_2(q)$ in this range.
The cases $4\le q\le10$ will be treated separately later by direct computation.
We have
\begin{align}
\f{R_{0}(q+1)}{R_{0}(q)} &= \f{L(q+1)}{L(q)}\cdot\f{V_{3}(q)}{V_{3}(q+1)} \\
&= \f{\Gamma\left(\f{9q}{4}+\f{13}{4}\right)}{2^{\f{5q}{4}+\f{1}{4}}3^{\f{5q}{8}+\f{9}{8}}\Gamma\left(\f{5q}{8}+\f{17}{8}\right)\Gamma\left(\f{3q}{8}+\f{7}{8}\right)}\cdot
\f{2^{\f{5q}{4}-1}3^{\f{5q}{8}+\f{1}{2}}\Gamma\left(\f{5q}{8}+\f{3}{2}\right)\Gamma\left(\f{3q}{8}+\f{1}{2}\right)}{\Gamma\left(\f{9q}{4}+1\right)} \\
&\sq\cdot\left(\f{4}{3}\right)^{\f{q}{2}}\f{\Gamma\left(\f{3q}{2}+1\right)}{\Gamma\left(\f{q}{2}+1\right)}
\left(\f{4}{3}\right)^{-\f{q+1}{2}}\f{\Gamma\left(\f{q}{2}+\f{3}{2}\right)}{\Gamma\left(\f{3q}{2}+\f{5}{2}\right)} \\
&= \f{1}{2^{\f{9}{4}}3^{\f{1}{8}}}\cdot\f{\Gamma\left(\f{9q}{4}+\f{13}{4}\right)\Gamma\left(\f{5q}{8}+\f{3}{2}\right)\Gamma\left(\f{3q}{8}+\f{1}{2}\right)\Gamma\left(\f{3q}{2}+1\right)\Gamma\left(\f{q}{2}+\f{3}{2}\right)}
{\Gamma\left(\f{9q}{4}+1\right)\Gamma\left(\f{5q}{8}+\f{17}{8}\right)\Gamma\left(\f{3q}{8}+\f{7}{8}\right)\Gamma\left(\f{3q}{2}+\f{5}{2}\right)\Gamma\left(\f{q}{2}+1\right)}.
\end{align}
By \lemref{lem:gamma},
\begin{align}
\f{\Gamma\left(\f{9q}{4}+\f{13}{4}\right)}{\Gamma\left(\f{9q}{4}+1\right)}
&= \f{\left(\f{9q}{4}+\f{9}{4}\right)\left(\f{9q}{4}+\f{5}{4}\right)\Gamma\left(\f{9q}{4}+\f{5}{4}\right)}{\Gamma\left(\f{9q}{4}+1\right)}
> \f{\left(\f{9q}{4}+\f{9}{4}\right)\left(\f{9q}{4}+\f{5}{4}\right)\left(\f{9q}{4}+1\right)}{\left(\f{9q}{4}+\f{5}{4}\right)^{\f{3}{4}}} \\
&= \left(\f{9q}{4}+\f{9}{4}\right)\left(\f{9q}{4}+\f{5}{4}\right)^{\f{1}{4}}\left(\f{9q}{4}+1\right), \\
\f{\Gamma\left(\f{5q}{8}+\f{3}{2}\right)}{\Gamma\left(\f{5q}{8}+\f{17}{8}\right)}
&> \f{1}{\left(\f{5q}{8}+\f{3}{2}\right)^{\f{5}{8}}}, \\
\f{\Gamma\left(\f{3q}{8}+\f{1}{2}\right)}{\Gamma\left(\f{3q}{8}+\f{7}{8}\right)}
&> \f{1}{\left(\f{3q}{8}+\f{1}{2}\right)^{\f{3}{8}}}, \\
\f{\Gamma\left(\f{3q}{2}+1\right)}{\Gamma\left(\f{3q}{2}+\f{5}{2}\right)}
&= \f{\Gamma\left(\f{3q}{2}+1\right)}{\left(\f{3q}{2}+\f{3}{2}\right)\Gamma\left(\f{3q}{2}+\f{3}{2}\right)}
> \f{1}{\left(\f{3q}{2}+\f{3}{2}\right)\left(\f{3q}{2}+1\right)^{\f{1}{2}}}, \\
\f{\Gamma\left(\f{q}{2}+\f{3}{2}\right)}{\Gamma\left(\f{q}{2}+1\right)} &> \f{\f{q}{2}+1}{\left(\f{q}{2}+\f{3}{2}\right)^{\f{1}{2}}}.
\end{align}
Therefore,
\begin{align}
\f{R_{0}(q+1)}{R_{0}(q)} &>
\f{\left(\f{9q}{4}+\f{9}{4}\right)\left(\f{9q}{4}+\f{5}{4}\right)^{\f{1}{4}}\left(\f{9q}{4}+1\right)\left(\f{q}{2}+1\right)}
{2^{\f{9}{4}}3^{\f{1}{8}}\left(\f{5q}{8}+\f{3}{2}\right)^{\f{5}{8}}\left(\f{3q}{8}+\f{1}{2}\right)^{\f{3}{8}}\left(\f{3q}{2}+\f{3}{2}\right)\left(\f{3q}{2}+1\right)^{\f{1}{2}}\left(\f{q}{2}+\f{3}{2}\right)^{\f{1}{2}}} \\
&= \f{3^{\f{7}{8}}(q+2)(9q+4)(9q+5)^{\f{1}{4}}}{2^{\f{11}{4}}(q+3)^{\f{1}{2}}(3q+2)^{\f{1}{2}}(3q+4)^{\f{3}{8}}(5q+12)^{\f{5}{8}}}.
\end{align}
Let $h(q)$ be the right-hand side. Then
\begin{align}
h(11) &= 1.400\ldots > 1, \\
\log h(q) &= \f{7}{8}\log 3 + \log(q+2) + \log(9q+4) + \f{1}{4}\log(9q+5) - \f{11}{4}\log 2 \\
&\sq - \f{1}{2}\log(q+3) - \f{1}{2}\log(3q+2)-\f{3}{8}\log(3q+4)-\f{5}{8}\log(5q+12), \\
\dv{q}\log h(q) &= \f{1}{q+2} + \f{9}{9q+4} + \f{9}{4(9q+5)} \\
&\sq - \f{1}{2(q+3)} - \f{3}{2(3q+2)} - \f{9}{8(3q+4)} - \f{25}{8(5q+12)}.
\end{align}
For $q \ge 11$,
\begin{align}
&q+2 \le q+\f{2}{11}q = \f{13}{11}q,\q 9q+4 \le 9q+\f{4}{11}q = \f{103}{11}q,\q 9q+5 \le 9q+\f{5}{11}q = \f{104}{11}q, \\
&q+3 > q,\q 3q+2 > 3q,\q 3q+4 > 3q,\q 5q+12 > 5q.
\end{align}
Hence
\begin{align}
\dv{q}\log h(q) &> \f{1}{\f{13}{11}q} + \f{9}{\f{103}{11}q} + \f{9}{4\cdot\f{104}{11}q}
- \f{1}{2q} - \f{3}{2\cdot3q} - \f{9}{8\cdot3q} - \f{25}{8\cdot5q} \\
&= \f{1941}{42848q} > 0.
\end{align}
Therefore, $h(q)$ is strictly increasing and $h(q)>1$ for $q\ge11$.
Hence $R_0(q+1)/R_0(q)>1$ for $q\ge11$.
It follows that $R_0(q)$ is strictly increasing. 

Thus,
\begin{align}
\label{eq:RV0}
R_{0}(q) \ge R_{0}(11) =  1.397\ldots > 1\q (q \ge 11).
\end{align}

We estimate $R_{1}(q)$ and $R_{2}(q)$ for $q \ge 11$.

The ratio of successive values of $V_{3}(q)$ is
\begin{align}
\f{V_{3}(q+1)}{V_{3}(q)} &= \left(\f{4}{3}\right)^{\f{q+1}{2}}\f{\Gamma\left(\f{3(q+1)}{2}+1\right)}{\Gamma\left(\f{q+1}{2}+1\right)}
\left(\f{4}{3}\right)^{-\f{q}{2}}\f{\Gamma\left(\f{q}{2}+1\right)}{\Gamma\left(\f{3q}{2}+1\right)} \\
\label{eq:V3q}
&= \f{2\Gamma\left(\f{3q}{2}+\f{5}{2}\right)\Gamma\left(\f{q}{2}+1\right)}{\sqrt{3}\Gamma\left(\f{3q}{2}+1\right)\Gamma\left(\f{q}{2}+\f{3}{2}\right)}.
\end{align}
By \lemref{lem:gamma},
\begin{align}
\f{\Gamma\left(\f{3q}{2}+\f{5}{2}\right)}{\Gamma\left(\f{3q}{2}+1\right)}
&= \f{\left(\f{3q}{2}+\f{3}{2}\right)\Gamma\left(\f{3q}{2}+\f{3}{2}\right)}{\Gamma\left(\f{3q}{2}+1\right)}
> \f{\left(\f{3q}{2}+\f{3}{2}\right)\left(\f{3q}{2}+1\right)}{\left(\f{3q}{2}+\f{3}{2}\right)^{\f{1}{2}}}
= \left(\f{3q}{2}+\f{3}{2}\right)^{\f{1}{2}}\left(\f{3q}{2}+1\right), \\
\f{\Gamma\left(\f{q}{2}+1\right)}{\Gamma\left(\f{q}{2}+\f{3}{2}\right)}
&> \f{1}{\left(\f{q}{2}+1\right)^{\f{1}{2}}}.
\end{align}
Hence
\begin{align}
\f{V_{3}(q+1)}{V_{3}(q)} &> \f{2\left(\f{3q}{2}+\f{3}{2}\right)^{\f{1}{2}}\left(\f{3q}{2}+1\right)}{\sqrt{3}\left(\f{q}{2}+1\right)^{\f{1}{2}}} = (3q+2)\sqrt{\f{q+1}{q+2}}.
\end{align}

Therefore,
\begin{align}
\f{R_{1}(q+1)}{R_{1}(q)} &= \f{V_{3}(q+1)}{V_{3}(q)}\f{V_{1}(q)}{V_{1}(q+1)} \\
&\ge (3q+2)\sqrt{\f{q+1}{q+2}}\cdot
q!\left(\f{9e}{q}\right)^{\f{q}{3}}\exp\left(2\left(\f{q}{9}\right)^{\f{1}{3}}\right) \\
&\sq\cdot\f{1}{(q+1)!}\left(\f{9e}{q+1}\right)^{-\f{q+1}{3}}\exp\left(-2\left(\f{q+1}{9}\right)^{\f{1}{3}}\right) \\
&= (9e)^{-\f{1}{3}}\left(1+\f{1}{q}\right)^{\f{q}{3}}\f{3q+2}{(q+1)^{\f{1}{6}}(q+2)^{\f{1}{2}}}
\exp(-\f{2}{\sqrt[3]{9}}(\sqrt[3]{q+1}-\sqrt[3]{q})).
\end{align}
For $q \ge 11$, we have
\begin{align}
3q+2 > 3q,\q q+1 \le \f{12}{11}q,\q q+2 \le \f{13}{11}q,
\end{align}
hence
\begin{align}
\f{3q+2}{(q+1)^{\f{1}{6}}(q+2)^{\f{1}{2}}} > \f{3q}{\left(\f{12}{11}q\right)^{\f{1}{6}}\left(\f{13}{11}q\right)^{\f{1}{2}}}
= \left(\f{3^{5}11^{4}q^{2}}{2^{2}13^{3}}\right)^{\f{1}{6}}.
\end{align}
Therefore, by the monotonicity facts stated in \secref{sec:strategy},
\begin{align}
\f{R_{1}(q+1)}{R_{1}(q)} &> (9e)^{-\f{1}{3}}\left(1+\f{1}{q}\right)^{\f{q}{3}}
\left(\f{3^{5}11^{4}q^{2}}{2^{2}13^{3}}\right)^{\f{1}{6}}\exp(-\f{2}{\sqrt[3]{9}}(\sqrt[3]{q+1}-\sqrt[3]{q})) \\
&\ge (9e)^{-\f{1}{3}}\left(1+\f{1}{11}\right)^{\f{11}{3}}
\left(\f{3^{5}11^{4}11^{2}}{2^{2}13^{3}}\right)^{\f{1}{6}}\exp(-\f{2}{\sqrt[3]{9}}(\sqrt[3]{11+1}-\sqrt[3]{11})) \\
&= 2.691\ldots > 1.
\end{align}
Thus, $R_{1}(q)$ is strictly increasing for $q\ge11$. A direct computation gives
\begin{align}
\label{eq:RV1}
R_{1}(q) &\ge R_{1}(11) = 227.981\ldots \q (q \ge 11).
\end{align}

We also have
\begin{align}
\f{R_{2}(q+1)}{R_{2}(q)} &= \f{V_{3}(q+1)}{V_{3}(q)}\f{V_{2}(q)}{V_{2}(q+1)} \\
&\ge (3q+2)\sqrt{\f{q+1}{q+2}}
\f{\sqrt{3}\log q}{\log 5}\left(\f{3q}{e}\right)^{\f{2q}{5}}
\f{\log 5}{\sqrt{3}\log (q+1)}\left(\f{3q+3}{e}\right)^{-\f{2q+2}{5}} \\
&= \left(\f{e}{3}\right)^{\f{2}{5}}\f{(3q+2)(q+1)^{\f{1}{10}}}{(q+2)^{\f{1}{2}}}\f{\log q}{\log (q+1)}\left(1+\f{1}{q}\right)^{-\f{2q}{5}}.
\end{align}
Here, for $q \ge 11$,
\begin{align}
3q+2 > 3q,\q q+1 > q,\q q+2 \le \f{13}{11}q,
\end{align}
hence
\begin{align}
\f{(3q+2)(q+1)^{\f{1}{10}}}{(q+2)^{\f{1}{2}}} > \f{3q\cdot q^{\f{1}{10}}}{(\f{13}{11}q)^{\f{1}{2}}} = 3\sqrt{\f{11}{13}}q^{\f{3}{5}}.
\end{align}
Therefore, by the monotonicity facts stated in \secref{sec:strategy},
\begin{align}
\f{R_{2}(q+1)}{R_{2}(q)} &> \left(\f{e}{3}\right)^{\f{2}{5}}3\sqrt{\f{11}{13}}11^{\f{3}{5}}\f{\log 11}{\log (11+1)}e^{-\f{2}{5}}
= 7.233\ldots > 1\q(q \ge 11).
\end{align}
Thus, $R_{2}(q)$ is strictly increasing for $q\ge11$. A direct computation gives
\begin{align}
\label{eq:RV2}
R_{2}(q) &\ge R_{2}(11) = 9512891.457\ldots \q (q \ge 11).
\end{align}

By \eqref{eq:NVR}, \eqref{eq:RV0}, \eqref{eq:RV1}, and \eqref{eq:RV2},
\begin{align}
\f{\vt{N(3, q, a)}}{\vt{C(3, q)}} &\ge \f{R_{0}(11)}{1+ \f{1}{R_{1}(11)} + \f{1}{R_{2}(11)}} = 1.391\ldots > 1 \q (q \ge 11).
\end{align}
Thus, we obtain $\vt{N} > \vt{C}$ for $q \ge 11$.

\tabref{tab:beq3NC} shows the exact values of $\vt{N}$ and $\vt{C}$
for the admissible passports of genus at least~$2$ with $b = 3$ and $4 \le q \le 10$.
It follows from the table that $\vt{N}>\vt{C}$ also holds for $4 \le q\le10$.

\begin{table}[htbp]
\centering
\small
\begin{tabular}{|l|c|r|r|}
\hline
Passport & Genus & \multicolumn{1}{c|}{$\vt{N(3, q, a)}$} & \multicolumn{1}{c|}{$\vt{C(3, q)}$} \\
\hline
$[12, 3^{4}, 4^{3}]$ & 3 & 1220 & 796 \\
\hline
$[15, 3^{5}, 5^{3}]$ & 4 & 112448 & 800 \\
\hline
$[18, 3^{6}, 6^{3}]$ & 5 & 17942400 & 149464 \\
\hline
$[21, 3^{7}, 7^{3}]$ & 6 & 4324710400 & 55664 \\
\hline
$[24, 3^{8}, 8^{3}]$ & 7 & 1471678700800 & 63473200 \\
\hline
$[27, 3^{9}, 9^{3}]$ & 8 & 674307938304000 & 3907520 \\
\hline
$[30, 3^{10}, 6^{5}]$ & 8 & 2573333642880000 & 45964920190 \\
\hline
$[30, 3^{10}, 10^{3}]$ & 9 & 400339523173990400 & 45964920190\\
\hline
\end{tabular}
\HL
\caption{Exact values of $\vt{N(3, q, a)}$ and $\vt{C(3, q)}$}\label{tab:beq3NC}
\end{table}

Therefore, every uniform passport $[n,3^{q},a^{p}]$ with $n=pa=3q$, $2\le p<q$, and genus
at least~$2$ admits a \dde with trivial automorphism group.
\end{proof}
\HL


\subsection{\texorpdfstring{The Subcase $b=4$}{The Subcase b=4}}
\label{sec:class4b4}

\begin{prop}
\label{prop:class4b4}
If a uniform passport $[n, 4^{q}, a^{p}]$ with $n = pa = 4q$ and $2 \le p < q$ has genus
at least~$2$, then it admits a \dde with trivial automorphism group.
\end{prop}

\begin{proof}
By \eqref{eq:genus-uc}, the genus is
\begin{align}
g &= \f{4q-(p+q)+1}{2} = \f{3q-p+1}{2}.
\end{align}
Hence $p$ and $q$ have opposite parity.

By \eqref{eq:Nbabq},
\begin{align}
\vt{N(4, q, a)} &\ge \f{4(q(4-1))!}{2^{q(4-2)}3\sqrt{2}(q-2)!}\left(\f{4}{2(4+1)}\right)^{\f{q-1}{2}}
\label{eq:Nlowerb4}
= \f{(3q)!}{2^{2q-2}3\sqrt{2}(q-2)!}\left(\f{2}{5}\right)^{\f{q-1}{2}}.
\end{align}

Let $L(q)$ denote the right-hand side.

On the other hand, by \eqref{eq:Cbqsum} and \lemref{lem:C2q3q4q},
\begin{align}
\vt{C(4, q)} &\le \sum_{\substack{\ell\colon \text{prime}\\ \ell\, \mid\, n}} \vt{C_{\f{n}{\ell}}(4, q)}
= \vt{C_{\f{n}{2}}(4, q)} + \sum_{\substack{\ell\colon \text{odd prime}\\ \ell\, \mid\, n}} \vt{C_{\f{n}{\ell}}(4, q)} \\
&\le (2q)!\left(\f{4e}{q}\right)^{\f{q}{2}}\exp\left(\f{\sqrt{q}}{2}\right) + \sum_{\substack{\ell\colon \text{odd prime}\\ \ell\, \mid\, n}} \vt{C_{\f{n}{\ell}}(4, q)}.
\end{align}
For odd primes, we use an argument similar to that in \secref{sec:class4b2age5}.
For an odd prime $\ell\mid n=4q$, we have $\ell\mid q$.
Let $m=q/\ell \in \Zp$. Since $\ell \nmid 4$, \propref{prop:Cnl} gives
\begin{align}
\vt{C_{\f{n}{\ell}}(4, q)} &= \left(\f{4q}{\ell}\right)!\f{\ell^{\f{3q}{\ell}}}{4^{\f{q}{\ell}}\left(\f{q}{\ell}\right)!}
= \f{\ell^{3m}(4m)!}{4^{m}m!}.
\end{align}
By the bounds from Stirling's formula \eqref{eq:Stirling},
\begin{align}
\vt{C_{\f{n}{\ell}}(4, q)} &\le \f{\ell^{3m}\sqrt{8 \pi m}\left(\f{4m}{e}\right)^{4m}}{4^{m}\sqrt{2 \pi m}\left(\f{m}{e}\right)^{m}}
\exp(\f{1}{48m}-\f{1}{12m+1}) \\
&= 2\left(\f{4\ell m}{e}\right)^{3m}\exp(\f{1}{48m}-\f{1}{12m+1}) \le 2\left(\f{4\ell m}{e}\right)^{3m}
= 2\left(\f{4q}{e}\right)^{\f{3q}{\ell}}.
\end{align}
Since $4q/e \ge 12/e>1$, the right-hand side is decreasing as $\ell$ increases.

Let $\ell_{1}, \ldots, \ell_{s}$ be the distinct odd prime divisors of $n = 4q$. Since each $\ell_{i}$ divides $q$,
\begin{align}
\ell_{1}\cdots\ell_{s} \le q.
\end{align}
Since $\ell_{i} \ge 3$ for each $i$,
\begin{align}
\ell_{1}\cdots\ell_{s} \ge 3^{s}.
\end{align}
Hence $3^{s} \le q$ and
\begin{align}
s \le \log_{3}q = \f{\log q}{\log 3}.
\end{align}
Therefore, since $4q/e \ge 12/e > 1$,
\begin{align}
\sum_{\substack{\ell\colon \text{odd prime}\\ \ell\, \mid\, n}} \vt{C_{\f{n}{\ell}}(4, q)}
&\le \sum_{\substack{\ell\colon \text{odd prime}\\ \ell\, \mid\, n}} 2\left(\f{4q}{e}\right)^{\f{3q}{\ell}}
\le \sum_{\substack{\ell\colon \text{odd prime}\\ \ell\, \mid\, n}} 2\left(\f{4q}{e}\right)^{\f{3q}{3}} \\
&\le  \f{2\log q}{\log 3}\left(\f{4q}{e}\right)^{q}.
\end{align}
Thus,
\begin{align}
\vt{C(4, q)} &\le V_{1}(q) + V_{2}(q), \\
V_{1}(q) &\ceq (2q)!\left(\f{4e}{q}\right)^{\f{q}{2}}\exp\left(\f{\sqrt{q}}{2}\right), \\
V_{2}(q) &\ceq \f{2\log q}{\log 3}\left(\f{4q}{e}\right)^{q}.
\end{align}

Combining this with \eqref{eq:Nlowerb4}, we obtain
\begin{align}
\f{\vt{N(4, q, a)}}{\vt{C(4, q)}} &\ge \f{L(q)}{V_{1}(q)+V_{2}(q)}.
\end{align}
Let
\begin{align}
R_{0}(q) \ceq \f{L(q)}{V_{1}(q)},\q R_{1}(q) \ceq \f{V_{1}(q)}{V_{2}(q)}.
\end{align}
Then
\begin{align}
\label{eq:N4RV}
\f{\vt{N(4, q, a)}}{\vt{C(4, q)}} &\ge \f{L(q)}{V_{1}(q) + \f{V_{1}(q)}{R_{1}(q)}} = \f{R_{0}(q)}{1+\f{1}{R_{1}(q)}}.
\end{align}

We have
\begin{align}
\f{R_{0}(q+1)}{R_{0}(q)} &= \f{L(q+1)}{L(q)}\cdot\f{V_{1}(q)}{V_{1}(q+1)} \\
&= \f{(3q+3)!}{2^{2q}3\sqrt{2}(q-1)!}\left(\f{2}{5}\right)^{\f{q}{2}}
\f{2^{2q-2}3\sqrt{2}(q-2)!}{(3q)!}\left(\f{2}{5}\right)^{-\f{q-1}{2}} \\
&\sq\cdot(2q)!\left(\f{4e}{q}\right)^{\f{q}{2}}\exp\left(\f{\sqrt{q}}{2}\right)
\f{1}{(2(q+1))!}\left(\f{4e}{q+1}\right)^{-\f{q+1}{2}}\exp\left(-\f{\sqrt{q+1}}{2}\right) \\
&= \f{3}{8\sqrt{10e}}\cdot\f{(3q+1)(3q+2)}{(q-1)(2q+1)}\left(1+\f{1}{q}\right)^{\f{q}{2}}\sqrt{q+1}
\exp(-\f{1}{2}(\sqrt{q+1}-\sqrt{q})).
\end{align}
Let $f(q)$ denote the right-hand side. Then
\begin{align}
f(3) = 1.521\ldots > 1.
\end{align}
For $q \ge 4$, by the monotonicity facts stated in \secref{sec:strategy},
\begin{align}
f(q) &\ge \f{3}{8\sqrt{10e}}\cdot\f{(3q+1)(3q+2)}{(q-1)(2q+1)}\left(1+\f{1}{4}\right)^{\f{4}{2}}\sqrt{4+1}
\exp(-\f{1}{2}(\sqrt{4+1}-\sqrt{4})) \\
&= \f{75}{128\sqrt{2e}}\exp(-\f{1}{2}(\sqrt{5}-2))\f{(3q+1)(3q+2)}{(q-1)(2q+1)}.
\end{align}
Let $h(q) \ceq (3q+1)(3q+2)/((q-1)(2q+1))$. Then
\begin{align}
h'(q) &= -\f{27q^{2}+26q+7}{(q-1)^{2}(2q+1)^{2}} < 0 \q(q \ge 4), \\
\lim_{q\to\infty}h(q) &= \f{9}{2}.
\end{align}
Hence $h(q) > 9/2$. Therefore, we obtain
\begin{align}
f(q) &> \f{75}{128\sqrt{2e}}\exp(-\f{1}{2}(\sqrt{5}-2))\cdot\f{9}{2} \\
&= \f{675}{256\sqrt{2}}\exp(-\f{1}{2}(\sqrt{5}-1))= 1.004\ldots > 1 \q(q \ge 4).
\end{align}
Thus,
\begin{align}
\f{R_{0}(q+1)}{R_{0}(q)} &= f(q) > 1\q(q \ge 3).
\end{align}
Hence $R_{0}(q)$ is strictly increasing. A direct computation gives
\begin{align}
&R_{0}(q) < 1\q(3 \le q \le 6), \\
\label{eq:R07}
&R_{0}(7) = 1.025\ldots > 1.
\end{align}

For $q \ge 7$,
\begin{align}
\f{R_{1}(q+1)}{R_{1}(q)} &= \f{V_{1}(q+1)}{V_{1}(q)}\cdot\f{V_{2}(q)}{V_{2}(q+1)} \\
&= (2q+2)!\left(\f{4e}{q+1}\right)^{\f{q+1}{2}}\exp\left(\f{\sqrt{q+1}}{2}\right)
\f{1}{(2q)!}\left(\f{4e}{q}\right)^{-\f{q}{2}}\exp\left(-\f{\sqrt{q}}{2}\right) \\*
&\sq\cdot\f{2\log q}{\log 3}\left(\f{4q}{e}\right)^{q}
\f{\log 3}{2\log (q+1)}\left(\f{e}{4(q+1)}\right)^{q+1} \\
&= \f{e^{\f{3}{2}}(2q+1)}{\sqrt{q+1}}\left(1+\f{1}{q}\right)^{-\f{3q}{2}}
\exp(\f{1}{2}(\sqrt{q+1}-\sqrt{q}))\f{\log q}{\log(q+1)}.
\end{align}
By the monotonicity facts stated in \secref{sec:strategy},
\begin{align}
\f{R_{1}(q+1)}{R_{1}(q)} &\ge \f{e^{\f{3}{2}}(2q+1)}{\sqrt{q+1}}e^{-\f{3}{2}}\cdot 1 \cdot\f{\log 7}{\log 8}
\ge \f{(2q+1)}{\sqrt{2q+1}}\f{\log 7}{\log 8}
\ge \sqrt{2\cdot 7+1}\cdot\f{\log 7}{\log 8} \\
&= 3.624\ldots > 1.
\end{align}
Hence $R_{1}(q)$ is strictly increasing and
\begin{align}
\label{eq:R17}
R_{1}(q) &\ge R_{1}(7) = 35072.788\ldots\q (q \ge 7).
\end{align}

Therefore, by \eqref{eq:N4RV}, \eqref{eq:R07}, and \eqref{eq:R17},
\begin{align}
\f{\vt{N(4, q, a)}}{\vt{C(4, q)}} &\ge \f{R_{0}(7)}{1+\f{1}{R_{1}(7)}} = 1.025\ldots > 1 \q (q \ge 7).
\end{align}
Thus, $\vt{N}>\vt{C}$ for $q\ge7$.

\tabref{tab:beq4NC} lists the exact values of $\vt{N}$ and $\vt{C}$ for the admissible passports
with $b = 4$ and $3 \le q \le 6$. These values show that $\vt{N}>\vt{C}$ also holds in this range.

\begin{table}[htbp]
\centering
\small
\begin{tabular}{|l|c|r|r|}
\hline
Passport & Genus & \multicolumn{1}{c|}{$\vt{N(4, q, a)}$} & \multicolumn{1}{c|}{$\vt{C(4, q)}$} \\
\hline
$[12, 4^{3}, 6^{2}]$ & 4 & 31860 & 1716\\
\hline
$[20, 4^{5}, 5^{4}]$ & 6 & 2508616656 & 8498180\\
\hline
$[20, 4^{5}, 10^{2}]$ & 7 &188376224400 & 8498180\\
\hline
$[24, 4^{6}, 8^{3}]$ & 8 & 131414201226000 & 1159169616\\
\hline
\end{tabular}
\HL
\caption{Exact values of $\vt{N(4, q, a)}$ and $\vt{C(4, q)}$}\label{tab:beq4NC}
\end{table}

Therefore, every uniform passport $[n,4^{q},a^{p}]$ with $n=pa=4q$, $2\le p<q$, and genus
at least~$2$ admits a \dde with trivial automorphism group.
\end{proof}


\subsection{\texorpdfstring{The Subcase $b \ge 5$}{The Subcase b >= 5}}
\label{sec:class4bge5}

\begin{prop}
\label{prop:class4bge5}
If a uniform passport $[n, b^{q}, a^{p}]$ with $n = pa = qb$, $2 \le p < q$, and $b \ge 5$ has genus
at least~$2$, then it admits a \dde with trivial automorphism group.
\end{prop}

\begin{proof}
We prove the statement by showing that $\vt{N}>\vt{D}$ except for finitely many
triples $(n,a,b)$.
Here we have $b \ge 5$ and $q \ge 3$.

By the lower bound for $\vt{N}$ in \eqref{eq:Nbabq} and the upper bound for
$\vt{D}$ in \lemref{lem:Dupper},
\begin{align}
\f{\vt{N(b, q, a)}}{\vt{D(n)}} &\ge \f{b(q(b-1))!}{2^{q(b-2)}3\sqrt{2}(q-2)!}\left(\f{b}{2(b+1)}\right)^{\f{q-1}{2}}
\left(\kappa_{1}\cdot 2^{\f{n}{2}}\left(\f{n}{2}\right)!\right)^{-1} \\
&= \f{b(q(b-1))!}{2^{q(\f{3b}{2}-2)}3\sqrt{2}\kappa_{1}(q-2)!\left(\f{bq}{2}\right)!}\left(\f{b}{2(b+1)}\right)^{\f{q-1}{2}},
\end{align}
where $\kappa_{1} = 2623/1894$ and $M! \ceq \Gamma(M+1)$.

Let $R(b, q)$ denote the right-hand side. To reduce the argument to finitely many cases, we first establish
the following monotonicity properties:
\begin{itemize}
\item For each $q \ge 3$, $R(b, q)$ is strictly increasing as $b$ ranges over
integers $b \ge 5$.
\item $R(5, q)$ is strictly increasing as $q$ ranges over integers $q \ge 3$.
\end{itemize}

Consider the ratio
\begin{align}
R_{b}(b, q) &\ceq \f{R(b+1, q)}{R(b, q)} = \f{(b+1)(qb)!}{2^{q(\f{3(b+1)}{2}-2)}3\sqrt{2}\kappa_{1}(q-2)!\left(\f{(b+1)q}{2}\right)!}\left(\f{b+1}{2(b+2)}\right)^{\f{q-1}{2}} \\
&\sq\cdot\f{2^{q(\f{3b}{2}-2)}3\sqrt{2}\kappa_{1}(q-2)!\left(\f{bq}{2}\right)!}{b(q(b-1))!}\left(\f{b}{2(b+1)}\right)^{-\f{q-1}{2}} \\
&= \f{1}{2^{\f{3q}{2}}}\cdot\f{b+1}{b}\cdot\f{(qb)!}{(q(b-1))!}\cdot\f{\left(\f{bq}{2}\right)!}{\left(\f{(b+1)q}{2}\right)!}
\left(\f{(b+1)^{2}}{b(b+2)}\right)^{\f{q-1}{2}}.
\end{align}
Here, we have
\begin{align}
\f{b+1}{b} &> 1, \\
\f{(qb)!}{(q(b-1))!} &= \prod_{i=1}^{q}(q(b-1)+i) > (q(b-1))^{q}, \\
\f{(b+1)^{2}}{b(b+2)} &= \f{b^{2}+2b+1}{b^{2}+2b} > 1,
\end{align}
and, since $q/2 > 1$, by \lemref{lem:gamma}\ref{itm:gamma-3},
\begin{align}
\f{\left(\f{bq}{2}\right)!}{\left(\f{(b+1)q}{2}\right)!} = \f{\Gamma\left(\f{bq}{2}+1\right)}{\Gamma\left(\f{bq}{2}+1+\f{q}{2}\right)}
> \left(\f{bq}{2}+1+\f{q}{2}\right)^{-\f{q}{2}} > \left(\f{bq}{2}+\f{q}{2}+\f{q}{2}\right)^{-\f{q}{2}} = \left(\f{(b+2)q}{2}\right)^{-\f{q}{2}}.
\end{align}
Therefore,
\begin{align}
R_{b}(b, q) &> \f{1}{2^{\f{3q}{2}}}(q(b-1))^{q}\left(\f{(b+2)q}{2}\right)^{-\f{q}{2}} \\
& = \left(\f{q(b-1)^{2}}{4(b+2)}\right)^{\f{q}{2}} = \left(\f{b-1}{b+2}\cdot\f{q(b-1)}{4}\right)^{\f{q}{2}}
= \left(\left(1-\f{3}{b+2}\right)\f{q(b-1)}{4}\right)^{\f{q}{2}} \\
&\ge \left(\left(1-\f{3}{5+2}\right)\cdot\f{3(5-1)}{4}\right)^{\f{3}{2}} = \left(\f{12}{7}\right)^{\f{3}{2}} > 1.
\end{align}
Hence $R(b, q)$ is strictly increasing in $b$.

The ratio of successive values of $R(5, q)$ with respect to $q$ is
\begin{align}
R_{q}(5, q) &\ceq \f{R(5, q+1)}{R(5, q)} = 
\f{5(4(q+1))!}{2^{\f{11}{2}(q+1)}3\sqrt{2}\kappa_{1}(q-1)!\left(\f{5(q+1)}{2}\right)!}\left(\f{5}{12}\right)^{\f{q}{2}} \\
&\sq\cdot\f{2^{\f{11}{2}q}3\sqrt{2}\kappa_{1}(q-2)!\left(\f{5q}{2}\right)!}{5(4q)!}\left(\f{5}{12}\right)^{-\f{q-1}{2}} \\
&=\f{1}{2^{\f{11}{2}}(q-1)}\cdot\f{(4(q+1))!}{(4q)!}\cdot\f{\Gamma\left(\f{5q}{2}+1\right)}{\Gamma\left(\f{5q}{2}+\f{7}{2}\right)}
\left(\f{5}{12}\right)^{\f{1}{2}}\\
&= \f{(4q+1)(4q+2)(4q+3)(4q+4)}{2^{\f{11}{2}}(q-1)}\cdot
\f{\Gamma\left(\f{5q}{2}+1\right)}{\left(\f{5q}{2}+\f{5}{2}\right)\left(\f{5q}{2}+\f{3}{2}\right)\Gamma\left(\f{5q}{2}+\f{3}{2}\right)}
\left(\f{5}{12}\right)^{\f{1}{2}}.
\end{align}
By \lemref{lem:gamma}\ref{itm:gamma-1},
\begin{align}
\f{\Gamma\left(\f{5q}{2}+1\right)}{\Gamma\left(\f{5q}{2}+\f{3}{2}\right)} > \left	(\f{5q}{2}+1\right)^{-\f{1}{2}}.
\end{align}
Hence
\begin{align}
R_{q}(5, q) &> \f{(4q+1)(4q+2)(4q+3)(4q+4)}{2^{\f{11}{2}}(q-1)\left(\f{5q}{2}+\f{5}{2}\right)\left(\f{5q}{2}+\f{3}{2}\right)\sqrt{\f{5q}{2}+1}}\left(\f{5}{12}\right)^{\f{1}{2}} \\
&= \f{(2q+1)(4q+1)(4q+3)}{2\sqrt{15}(q-1)(5q+3)\sqrt{5q+2}}.
\end{align}
Therefore,
\begin{align}
R_{q}(5, q)^{2} - 1 &> \f{(2q+1)^{2}(4q+1)^{2}(4q+3)^{2}}{60(q-1)^{2}(5q+3)^{2}(5q+2)} - 1.
\end{align}
Let $t = q-3$. Then
\begin{align}
R_{q}(5, q)^{2} - 1 &> \f{(2(t+3)+1)^{2}(4(t+3)+1)^{2}(4(t+3)+3)^{2}}{60(t+3-1)^{2}(5(t+3)+3)^{2}(5(t+3)+2)} - 1 \\
&= \f{1024 t^6+ 14004 t^5+ 78532 t^4 + 247488 t^3 + 523876 t^2 + 759900 t + 541305}{60(t+2)^{2}(5t+18)^{2}(5t+17)} \\
&> 0\q(t \ge 0).
\end{align}
Hence $R_{q}(5, q) > 1$. Thus, $R(5, q)$ is strictly increasing as $q$ ranges over integers $q \ge 3$.

A direct computation gives
\begin{align}
&R(5, q) < 1\ (3 \le q \le 11),\q R(5, 12) > 1, \\
&R(6, 3) < 1,\q R(6, q) > 1\ (4 \le q \le 11), \\
&R(7, 3) > 1.
\end{align}
These computations and the monotonicity properties established above
give the following conclusions. If $q \ge 12$, then
\begin{align}
R(b, q) &\ge R(5, q) \ge R(5, 12) > 1.
\end{align}
If $4 \le q \le 11$ and $b \ge 6$, then
\begin{align}
R(b, q) &\ge R(6, q) > 1.
\end{align}
Finally, if $q = 3$ and $b \ge 7$, then
\begin{align}
R(b, 3) &\ge R(7, 3) > 1.
\end{align}
It follows that $R(b, q) > 1$, hence $\vt{N} > \vt{D}$, for all passports except for the following:
\begin{align}
&[18, 6^{3}, 9^{2}],\ [30, 5^{6}, 6^{5}],\ [30, 5^{6}, 10^{3}],\ [35, 5^{7}, 7^{5}],\ [40, 5^{8}, 8^{5}], \\
&[45, 5^{9}, 9^{5}],\ [45, 5^{9}, 15^{3}],\ [50, 5^{10}, 10^{5}],\ [55, 5^{11}, 11^{5}].
\end{align}

\tabref{tab:bge5ND} shows the exact values of $\vt{N}$ and $\vt{D}$ for these passports.
In all cases, we have $\vt{N}>\vt{D}$.

\begin{table}[htbp]
\centering
\footnotesize
\begin{tabular}{|l|c|r|}
\hline
Passport & Genus &
\makecell[c]{
$\vt{N(b,q,a)}$ (top), $\vt{D(n)}$ (bottom)} \\
\hline
$[18, 6^{3}, 9^{2}]$ & 7 &
\makecell[r]{57767270400\\186318144} \\
\hline
$[30, 5^{6}, 6^{5}]$ & 10 &
\makecell[r]{48909810817444921344\\42850087977945660} \\
\hline
$[30, 5^{6}, 10^{3}]$ & 11 &
\makecell[r]{7606132222378640670720\\42850087977945660} \\
\hline
$[35, 5^{7}, 7^{5}]$ & 12 &
\makecell[r]{25301336298397996859523072\\395766805} \\                                                                                                                                                                                                                                       
\hline
$[40, 5^{8}, 8^{5}]$ & 14 &
\makecell[r]{25706613944092720665201487183872\\2551082656125844214400000} \\
\hline
$[45, 5^{9}, 9^{5}]$ & 16 &
\makecell[r]{46602888248571069531099154183801012224\\18763698601465755750} \\
\hline
$[45, 5^{9}, 15^{3}]$ & 17 &
\makecell[r]{16307284324171335435312109500945200578560\\18763698601465755750} \\
\hline
$[50, 5^{10}, 10^{5}]$ & 18 &
\makecell[r]{140236937177154356771171774119984509977886720\\520469842636666622728518576000000} \\
\hline
$[55, 5^{11}, 11^{5}]$ & 20 &
\makecell[r]{662088899432150365442833716144463291652766418599936\\1949062519326065} \\
\hline
\end{tabular}
\HL
\caption{Exact values of $\vt{N(b, q, a)}$ and $\vt{D(n)}$ ($b \ge 5$)}
\label{tab:bge5ND}
\end{table}

Therefore, every uniform passport $[n,b^{q},a^{p}]$ with $n=pa=qb$, $2\le p<q$, genus
at least~$2$, and $b\ge5$ admits a \dde with trivial automorphism group.
\end{proof}
\OL


\section{Main Results}

Combining the results of \secref{sec:class4}, we obtain the following.

\begin{thm}
\label{thm:class4}
\thmclassfour
\end{thm}

\begin{proof}
Since permuting black vertices, white vertices, and faces preserves the automorphism group,
it suffices to prove the statement for the passport $[n, b^{q}, a^{p}]$,
where $n = pa = qb$ and $2 \le p < q$.

Since the genus is at least~$2$, by \eqref{eq:genus-uc},
\begin{align}
\f{n-(p+q)+1}{2} \ge 2,
\end{align}
hence
\begin{align}
q &\le n-p-3 < n.
\end{align}
It follows that
\begin{align}
b=\frac{n}{q}>1.
\end{align}
Moreover, since $p<q$,
\begin{align}
a=\frac{n}{p}>\frac{n}{q}=b.
\end{align}
Therefore,
\begin{align}
2\le b<a,
\end{align}
and the following cases exhaust all possibilities.
\HL

The case $b = 2$, $a = 3$ was proved in \propref{prop:class4b2a3}.

The case $b = 2$, $a = 4$ was proved in \propref{prop:class4b2a4}.

The case $b = 2$, $a \ge 5$ was proved in \propref{prop:c4b2age5}.

The case $b = 3$ was proved in \propref{prop:class4b3}.

The case $b = 4$ was proved in \propref{prop:class4b4}.

The case $b \ge 5$ was proved in \propref{prop:class4bge5}.
\end{proof}

Moreover, combining this result with those of our previous papers, we obtain the following theorem
for general uniform unicellular passports of genus at least~$2$.

\begin{thm}
\label{thm:class1-4}
\thmclassonetofour
\end{thm}

\begin{proof}
Since permuting black vertices, white vertices, and faces preserves the automorphism group,
it suffices to prove the statement for the passport $[a^{p},b^{q},n]$ with $n=pa=qb$ and $a\ge b$.

The case $a = n$ was proved in \cite[Theorem~7.6]{Ohnishi26}.

The case $b = a < n$ was proved in \cite[Theorem~4.7]{Ohnishi2606}.

The case $b < a < n$ was proved in \thmref{thm:class4}.
\end{proof}
\HL

\FloatBarrier
\subsection*{Acknowledgements}

I would like to express my sincere gratitude to Associate Professor Yasuhiro Wakabayashi for his continued guidance, insightful advice, and encouragement throughout my research.

I am also deeply grateful to all the members of the Wakabayashi Laboratory for their continued support and valuable discussions.
I would also like to thank the researchers who have provided helpful comments and engaged in stimulating discussions
at seminars and conferences.


\bibliographystyle{amsalpha}
\bibliography{References}

\end{document}